\documentclass[10pt]{amsart}
\usepackage{amssymb,bm,mathrsfs}

\newcommand{\map}[1]{\xrightarrow{#1}}
\newcommand{\iso}{\cong}

\newcommand{\Gal}{\mathrm{Gal}}
\newcommand{\Hom}{\mathrm{Hom}}
\newcommand{\Aut}{\mathrm{Aut}}
\newcommand{\End}{\mathrm{End}}
\newcommand{\Spec}{\mathrm{Spec}}
\newcommand{\Q}{\mathbb Q}
\newcommand{\Z}{\mathbb Z}
\newcommand{\R}{\mathbb R}
\newcommand{\C}{\mathbb C}
\newcommand{\F}{\mathbb F}

\newcommand{\co}{\mathcal O}
\newcommand{\alg}{\mathrm{alg}}
\newcommand{\action}{\bullet}
\newcommand{\Lie}{\mathrm{Lie}}
\newcommand{\naive}{\mathrm{naive}}
\newcommand{\pappas}{\mathrm{Pap}}
\newcommand{\kk}{{\bm{k}}}

\newcommand{\Gr}{\mathrm{Gr}}
\newcommand{\interior}{\mathrm{int}}
\newcommand{\bndry}{\mathrm{bnd}}
\newcommand{\ord}{\mathrm{ord}}
\newcommand{\serre}{\mathrm{Serre}}
\newcommand{\hh}{{\bm{h}}}

\input xy
\xyoption{all}

\begin{document}

\author{Benjamin Howard}
\title{Complex multiplication cycles and Kudla-Rapoport divisors II}
\thanks{This research was supported in part by NSF grant DMS-0901753.}
\address{Department of Mathematics\\Boston College\\Chestnut Hill, MA 02467}
\date{}
\subjclass[2000]{14G35, 14G40,11F41}
\keywords{Shimura varieties, Arakelov theory, arithmetic intersection theory}

\theoremstyle{plain}
\newtheorem{Thm}{Theorem}[subsection]
\newtheorem{Prop}[Thm]{Proposition}
\newtheorem{Lem}[Thm]{Lemma}
\newtheorem{Cor}[Thm]{Corollary}
\newtheorem{Conj}[Thm]{Conjecture}
\newtheorem{BigThm}{Theorem}

\theoremstyle{definition}
\newtheorem{Def}[Thm]{Definition}

\theoremstyle{remark}
\newtheorem{Rem}[Thm]{Remark}

\numberwithin{equation}{subsection}
\renewcommand{\theBigThm}{\Alph{BigThm}}


\begin{abstract}
This paper is about the arithmetic of  \emph{Kudla-Rapoport divisors} on  
 Shimura varieties of type $\mathrm{GU}(n-1,1)$.   In the first part of the paper we construct a 
toroidal  compactification of N.~Kr\"amer's integral model of the Shimura variety.  
This extends work of K.-W.~Lan, who constructed a compactification at unramified primes.

In the second, and main, part of the paper
we use ideas of Kudla to construct Green functions for the Kudla-Rapoport divisors on the open Shimura variety, 
and analyze the behavior of these functions near the boundary of the compactification.  
The Green functions turn out to have logarithmic singularities along certain components of the boundary,
up to log-log error terms.  Thus, by adding a prescribed linear combination of 
boundary components to a Kudla-Rapoport divisor one obtains a class  in the 
arithmetic Chow group  of  Burgos-Kramer-K\"uhn.  

In the third and 
final part of the paper we compute the arithmetic intersection of each of these divisors with 
a cycle of complex multiplication points. The computation is quickly reduced to the calculations
of the author's earlier work \emph{Complex multiplication cycles and Kudla-Rapoport divisors}.
The arithmetic intersection multiplicities are shown to appear as Fourier coefficients
of the diagonal restriction of the central derivative of a Hilbert modular Eisenstein series.
\end{abstract}

\maketitle

\section{Introduction}

Let $\kk$ be an imaginary quadratic field with discriminant $-d_\kk$,  let $x\mapsto \overline{x}$ denote complex conjugation on $\kk$, and fix an embedding $\iota: \kk\to \C$.   \emph{We assume throughout that $d_\kk$ is odd}.


\subsection{Integral models of unitary Shimura varieties}


For every pair $(r,s)$ of nonnegative integers one may  attempt to define a Deligne-Mumford stack
$\mathcal{M}_{(r,s)}$ over $\co_\kk$ as the moduli space of triples $(A,\kappa,\psi)$, 
where $A\to S$ is an abelian scheme of dimension $r+s$ over an $\co_\kk$-scheme $S$,
$\kappa:\co_\kk\to \End(A)$ is an action of $\co_\kk$ on $A$, and $\psi$ is an $\co_\kk$-linear
principal polarization of $A$ (see Section \ref{ss:notation} for the meaning of $\co_\kk$-linear).  
One  further demands that the induced action of $\co_\kk$ on 
$\Lie(A)$ satisfies a suitable \emph{signature $(r,s)$ condition}.  Such a signature condition
asserts, roughly speaking,
that $\Lie(A)$ behaves like the $\co_\kk\otimes_\Z \co_S$-module 
$\co_S^r\oplus \co_S^s$, where $\co_\kk$ acts on the factor $\co_S^r$ through the structure morphism
$\co_\kk\to \co_S$, and acts on the factor $\co_S^s$ through the complex conjugate
of the structure morphism.  Of course the subtlety lies in the precise meaning of ``behaves like",
and different definitions lead to different moduli spaces.

 In the case of 
signature $(m,0)$ one simply demands that $\co_\kk$ acts on $\Lie(A)$ through the 
structure map $\co_\kk\to \co_S$, and this gives an unambiguous definition of 
$\mathcal{M}_{(m,0)}$.    The resulting stack $\mathcal{M}_{(m,0)}$ has a very
simple structure (Proposition \ref{Prop:zero etale}): it is proper and smooth over 
$\co_\kk$ of relative dimension $0$. In particular its complex fiber is a $0$-dimensional complex orbifold.

In the case of signature $(m,1)$ there are at least  three 
competing definitions of the signature condition: a \emph{naive} definition, and more
refined definitions introduced by Pappas \cite{pappas00} and Kr\"amer \cite{kramer}.  
These three definitions, recalled in Section \ref{ss:good models},  
lead to three  stacks over $\co_\kk$, related by canonical morphisms
\begin{equation}\label{intro models}
\mathcal{M}_{(m,1)} \to \mathcal{M}^\pappas_{(m,1)} \to \mathcal{M}^\naive_{(m,1)}
\end{equation}
which become isomorphisms after restriction to $\co_\kk[1/d_\kk]$.  The  naive model
$\mathcal{M}^\naive_{(m,1)}$ is known, by work of Pappas, to be neither flat nor regular.  The refined model 
$ \mathcal{M}^\pappas_{(m,1)}$ of Pappas  is flat but not regular.   Kr\"amer's model $\mathcal{M}_{(m,1)}$ 
is both flat and regular.    It is essential here that $d_\kk$ be odd.  While the definitions of the stacks in (\ref{intro models})
make sense also for even $d_\kk$, there seems to be no reason to expect in this generality  that 
$\mathcal{M}_{(m,1)}$ is regular.

If  $m>0$ then $\mathcal{M}_{(m,1)}$ is typically not proper,
and our first main result (Theorem \ref{Thm:compactify})
is the existence of a canonical toroidal compactification of Kr\"amer's model.

\begin{BigThm}\label{Big A}
Fix $m\ge 0$. There is an open immersion  $\mathcal{M}_{(m,1)} \to \mathcal{M}^*_{(m,1)}$ with 
dense image having the following properties.
The stack $\mathcal{M}^*_{(m,1)}$ is regular, is proper and flat over $\co_\kk$,   and is smooth over 
$\co_\kk[1/d_\kk]$.  The boundary
\[
\partial\mathcal{M}^*_{(m,1)} = \mathcal{M}^*_{(m,1)} \smallsetminus \mathcal{M}_{(m,1)}
\] 
is a divisor, and is smooth and proper over $\co_\kk$. 
\end{BigThm}

One can say much more about the boundary divisor $\partial\mathcal{M}^*_{(m,1)}$.  For example, the universal 
abelian scheme $A$ over $\mathcal{M}_{(m,1)}$ extends to a semi-abelian scheme $G$
over $\mathcal{M}^*_{(m,1)}$.  At a  geometric point $\Spec(\F) \to \partial\mathcal{M}^*_{(m,1)}$ of the 
boundary the fiber $G_{/\F}$ sits in an exact sequence
\begin{equation}\label{intro extension}
0 \to T \to G_{/\F} \to B\to 0
\end{equation}
where $T$ is a  torus over $\F$ whose character group is a projective 
$\co_\kk$-module of rank one, and $B\in \mathcal{M}_{(m-1,0)}(\F)$.  As the geometric point varies over a 
single irreducible component of $\partial\mathcal{M}^*_{(m,1)/\F}$, the torus $T$ and the abelian variety $B$
are constant, but the isomorphism class of the extension (\ref{intro extension}) varies.

The proof of Theorem \ref{Big A} follows the  methods of Chai-Faltings \cite{faltings-chai} and  Lan \cite{kai-wen}.  
Indeed, if one works over $\co_\kk[1/d_\kk]$ instead of $\co_\kk$
 then Theorem \ref{Big A} is a special  case of Lan's results.


\subsection{Arithmetic Kudla-Rapoport divisors}


Our motivation for constructing a good 
integral model $\mathcal{M}^*_{(m,1)}$ over $\co_\kk$ is to have 
a suitable space on which to do arithmetic intersection theory.  In particular, 
we wish to address some the speculative questions raised in the introduction 
of \cite{howardCM} concerning the arithmetic intersections of \emph{Kudla-Rapoport}
divisors with cycles of complex multiplication points.  The set up is as follows.  
Fix an integer $n\ge 1$ and consider the $\co_\kk$-stack
\[
\mathcal{M}^* = \mathcal{M}_{(1,0)} \times_{\co_\kk} \mathcal{M}^*_{(n-1,1)}.
\]
It is regular of dimension $n$,  is flat and proper over $\co_\kk$, and contains
\[
\mathcal{M} = \mathcal{M}_{(1,0)} \times_{\co_\kk} \mathcal{M}_{(n-1,1)}
\]
as a dense open substack. Set
\[
\partial\mathcal{M}^* = \mathcal{M}^* \smallsetminus \mathcal{M} 
\]
so that $\partial\mathcal{M}^*\iso \mathcal{M}_{(1,0)} \times_{\co_\kk} \partial\mathcal{M}^*_{(n-1,1)}$.

Following work of Kudla and Rapoport \cite{KRunitaryI, KRunitaryII}, we define a family of 
divisors $\mathcal{Z}(m)$ on $\mathcal{M}$ indexed by nonzero integers $m$.  
 Let $\mathcal{Z}^*(m)$ be the Zariski closure of $\mathcal{Z}(m)$ in $\mathcal{M}^*$.  
In \cite{howardCM}  one finds the construction of a Green function 
$\Gr(m,v)$ on $\mathcal{M}(\C)$ for $\mathcal{Z}(m)$, which depends on an auxiliary choice 
of positive real parameter $v$.   In Section \ref{ss:arith KR} we define, 
for each geometric component $\mathcal{C}$ of the boundary $\partial\mathcal{M}^*$, an integer
$\mathrm{Ind}_\mathcal{C}(m)$ in such a way that the formal sum
\[
\mathcal{B}(m,v) = \frac{1}{4\pi v} \sum_\mathcal{C} \mathrm{Ind}_\mathcal{C}(m) \cdot \mathcal{C}
\]
is a divisor on $\mathcal{M}^*$ with real coefficients.
Our second main result  amounts to an examination of these Green functions near the
boundary of $\mathcal{M}^*$.  We prove  that
each $\Gr(m,v)$ has logarithmic singularities along certain boundary components, so 
that $\Gr(m,v)$, when viewed as a function on $\mathcal{M}^*(\C)$, is a Green function 
for the divisor $\mathcal{Z}^*(m) + \mathcal{B}(m,v)$.  Here the term \emph{Green function}
must be interpreted as in the work of Burgos-Kramer-K\"uhn \cite{BKK};
$\Gr(m,v)$ is a Green function   \emph{up to log-log error terms} along the boundary.  The main results of   \cite{BKK} show that the theory of 
arithmetic Chow groups developed by Gillet-Soul\'e \cite{gillet-soule90,soule92} can be extended to 
allow for  such log-log error terms.   The following is  Theorem \ref{Thm:good green}
of the text.

\begin{BigThm}\label{Big B}
For any $m\not=0$ and any $v\in \R^+$, the pair
\[
 \widehat{\mathcal{Z}}(m,v) = \big(\mathcal{Z}^*(m) + \mathcal{B}(m,v) , \Gr(m,v)\big)
\] 
defines a class
in the Burgos-Kramer-K\"uhn codimension one arithmetic Chow group 
$\widehat{\mathrm{CH}}_\R^1(\mathcal{M}^*)$ of Section \ref{ss:BKK}.
\end{BigThm}

\begin{Rem}
If $m<0$ then both $\mathcal{Z}^*(m)$ and  $\mathcal{B}(m,v)$ are zero, but the 
class $ \widehat{\mathcal{Z}}(m,v)$ may still be nontrivial.
\end{Rem}

A  \emph{Hermitian lattice} is a pair 
$(\Lambda,h_\Lambda)$ in which $\Lambda$
is a finitely generated projective $\co_\kk$-module, and $h_\Lambda$ is an $\co_\kk$-valued
Hermitian form on $\Lambda$.  The Hermitian lattice $(\Lambda,h_\Lambda)$ is 
\emph{self dual} if the map $\Lambda \to \Hom_{\co_\kk}(\Lambda,\co_\kk)$ 
defined by $y\mapsto h_\Lambda(\cdot, y)$ is a bijection.

The integers $\mathrm{Ind}_\mathcal{C}(m)$ admit an elegant description in 
terms of Hermitian lattices.  
The complex orbifold $\mathcal{M}^*(\C)$ is disconnected.  Its connected components
are indexed by isomorphism classes of pairs $(\mathfrak{A}_0,\mathfrak{A})$ of 
self dual Hermitian lattices, where $\mathfrak{A}_0$ has signature $(1,0)$ and $\mathfrak{A}$
has signature $(n-1,1)$.  The $\mathfrak{A}_0$'s index the connected components of 
$\mathcal{M}_{(1,0)}(\C)$, while the $\mathfrak{A}$'s index the connected components of
$\mathcal{M}^*_{(n-1,1)}(\C)$.   The boundary components of $\mathcal{M}^*(\C)$ are indexed
by isomorphism classes of triples $(\mathfrak{A}_0, \mathfrak{m} \subset \mathfrak{A})$
where $\mathfrak{A}_0$ and $\mathfrak{A}$ are as above, and $\mathfrak{m}\subset\mathfrak{A}$
is an isotropic $\co_\kk$-module direct summand of rank one.  Suppose $\mathcal{C}$ is the  boundary component
of $\mathcal{M}^*(\C)$ corresponding to the triple $(\mathfrak{A}_0, \mathfrak{m} \subset \mathfrak{A})$.  
The $\co_\kk$-module $\Hom_{\co_\kk}(\mathfrak{A}_0,\mathfrak{A})$
is a self dual Hermitian lattice  of signature $(n-1,1)$,  and contains
\[
\mathfrak{a} = \Hom_{\co_\kk}(\mathfrak{A}_0,\mathfrak{m})
\]
as an isotropic direct summand.   The quotient 
$\Lambda=\mathfrak{a}^\perp/\mathfrak{a}$ is a self dual Hermitian lattice of 
signature $(n-2,0)$, and 
\[
\mathrm{Ind}_\mathcal{C}(m) = \# \{ x\in\Lambda : \langle x,x\rangle =m \}.
\]

The closest result to Theorem \ref{Big B} in the literature is found in work of Bruinier-Burgos-K\"uhn
\cite{bruinier-burgos-kuhn} on Hilbert modular surfaces.  
There are   two natural constructions 
of Green functions for the Hirzebruch-Zagier divisors on a Hilbert modular surface $\mathcal{Y}$: 
the \emph{automorphic} Green functions of Bruinier  \cite{bruinier99}, and the \emph{Kudla-style} Green functions 
as constructed in \cite{HY-HZ} following the ideas of  \cite{kudla97}.
One of  the  the main results of \cite{bruinier-burgos-kuhn} is a Hilbert modular surface analogue of 
Theorem \ref{Big B}:  the  automorphic  Green function for a Hirzebruch-Zagier divisor on $\mathcal{Y}$  extends to  a Green function  on 
 a toroidal compactification $\mathcal{Y}^*$,  provided one adds to the divisor a certain linear combinations of 
 boundary components.   However, the analogous result for the Kudla-style Green functions on $\mathcal{Y}$ is false, 
 as recently shown by Berndt-K\"uhn \cite{BK}.
 The Green function appearing in Theorem \ref{Big B} is constructed in the same way as the
Kudla-style Green function of \cite{HY-HZ}, and thus Theorem \ref{Big B} may be somewhat surprising.  
Of course,  the compactification
$\mathcal{M} \hookrightarrow \mathcal{M}^*$ is much nicer than the toroidal compactification
$\mathcal{Y} \hookrightarrow \mathcal{Y}^*$ (for example, the boundary of $\mathcal{Y}^*$ is not smooth), 
and so one should not be too surprised.


\subsection{Intersections with CM cycles}


Having constructed arithmetic cycle classes
\[
\widehat{\mathcal{Z}} (m,v) \in \widehat{\mathrm{CH}}_\R^1(\mathcal{M}^*)
\]
we may complete a small part of the speculative program laid out in the introduction 
of \cite{howardCM}.
 Let $F$ be a totally real number field with $[F:\Q]=n$ (the same $n$
used in the definition of $\mathcal{M}^*$), and define a 
CM field $K=\kk\otimes_\Q F$.  Assume that the discriminant of $F$ is odd and coprime to $d_\kk$.
For a suitable choice of CM type $\Phi$ of $K$, we define in 
Section \ref{ss:cm cycles} a cycle $\mathcal{X}_\Phi$ of dimension one on $\mathcal{M}^*_{/\co_\Phi}$,
where $\co_\Phi$ is the ring of integers in a particular finite extension $K_\Phi/\Q$.
If $n>2$ then $K_\Phi$ is the reflex field of $(K,\Phi)$.  The cycle $\mathcal{X}_\Phi$  is 
essentially the cycle of points with complex multiplication by $\co_K$ and CM type $\Phi$.
Associated to $\mathcal{X}_\Phi$  is a canonical linear functional
\[
 \widehat{\mathrm{CH}}_\R^1(\mathcal{M}^*) \to \R
\]
denoted  $\widehat{\mathcal{D}} \mapsto [ \widehat{\mathcal{D}} : \mathcal{X}_\Phi ]$
and called \emph{arithmetic intersection against  $\mathcal{X}_\Phi$}.

The main results of \cite{howardCM} consist of calculations of the intersection 
multiplicity of naive versions of $\mathcal{X}_\Phi$ and $\mathcal{Z}(m)$ on the 
(non-compact, non-regular, and non-flat) Shimura variety 
$\mathcal{M}_{(1,0)} \times_{\co_\kk} \mathcal{M}^\naive_{(n-1,1)}$.
By reducing to the calculations of \cite{howardCM}, we are able to prove in Section \ref{ss:final results}
a precise  formula for $[ \widehat{\mathcal{Z}} (m,v)  : \mathcal{X}_\Phi ]$, and show that
this value is related to the Fourier coefficients of Eisenstein series.

More precisely, let $\mathcal{E}_\Phi(\tau,s)$ be the Hilbert modular Eisenstein series 
of weight one of \cite[Definition 4.1.1]{howardCM}.  It is a nonholomorphic function of the variable 
$\tau \in \mathcal{H}^n$  in the product  of complex upper half planes indexed by the $n$ 
embeddings $F\to \R$, 
and a meromorphic function of the complex variable $s$, vanishing at the center $s=0$
of its functional equation.  
If we pull back by the diagonal embedding $i_F :\mathcal{H} \to \mathcal{H}^n$ and take the 
central derivative, we obtain a nonholomorphic modular form
\[
\mathcal{E}_\Phi '( i_F(\tau) , 0) = \sum_{m\in \Z}  c_\Phi(m,v) \cdot q^m
\]
of the variable $\tau=u+iv \in \mathcal{H}$, where $q=e^{2\pi i \tau}$ as usual.  
From the calculations of \cite{howardCM}  we deduce the following result 
(Theorem \ref{Thm:main intersection} of the text), which is in accordance with 
the general yoga of Kudla's program \cite{kudla04b} predicting relations between
arithmetic intersections and Fourier coefficients of Eisenstein series.

\begin{BigThm}\label{Big C}
For any nonzero $m\in \Z$, and any $v\in \R^+$,
\[
[ \widehat{\mathcal{Z}}(m,v)  :  \mathcal{X}_\Phi ]  = -  \frac{ h(\kk)}{   w(\kk)  } 
 \frac{ \sqrt{\mathrm{N}( d_{ K / F} ) }  }{2^{r-1} } 
 \cdot c_\Phi(m,v).
\]
Here $h(\kk)$ is the class number of $\kk$, $w(\kk)$ is the number of roots of unity in $\kk^\times$,
$d_{K/F}$ is the discriminant of $K/F$, and 
$r$ is the number of primes of $F$ ramified in $K$, including the archimedean primes.
\end{BigThm}

Of course the right hand side of the equality of Theorem \ref{Big C} is defined
even for $m=0$, and  it is natural to ask if the result can be extended to $m=0$.
The problem is finding the correct definition of $\widehat{\mathcal{Z}}(0,v)$.  For example, it is 
reasonable to conjecture that there is a (necessarily unique) choice of 
$\widehat{\mathcal{Z}}(0,v)$ for which the formal generating series
\[
\sum_{ m\in \Z} \widehat{\mathcal{Z}}(m,v) \cdot q^m \in \widehat{\mathrm{CH}}^1_\R(\mathcal{M}^*) ((q))
\]
(here $q=e^{2\pi i \tau}$ and $\tau=u+i v$) is a nonholomorphic modular form,
and that Theorem \ref{Big C} holds at $m=0$ for this choice.


\subsection{Notation}
\label{ss:notation}


Throughout the article $\kk$, as above, is a quadratic imaginary field of 
odd discriminant $-d_\kk$ with a chosen embedding $\iota:\kk\to \C$.
Let  $\delta_\kk\in \kk$  be the element determined by $\iota(\delta_\kk) = i\cdot \sqrt{d_\kk}$,
where $\sqrt{d_\kk}$ is the positive real square root of $d_\kk$.  Let $\mu(\kk)$ be the group of roots of 
unity in $\kk$.

Schemes are always  assumed to be locally Noetherian and separated, 
and \emph{stack} means  locally Noetherian and separated  Deligne-Mumford stack.

If $A\to S$ is an abelian scheme over an arbitrary base scheme, equipped with 
an action $\kappa:\co_\kk\to \End(A)$, there is an induced action $x\mapsto  \kappa(\overline{x})^\vee$
on the dual abelian scheme $A^\vee$.

 \subsection{Acknowledgements}

The author thanks Kai-Wen Lan for helpful correspondence.  Portions of this research 
were conducted at the  Fields Institute during the Spring of 2012, and the author thanks the 
Institute for its hospitality.


\section{Integral models of unitary Shimura varieties}
\label{s:integral models}


In this section we prove Theorem \ref{Big A} of the introduction.  The proof uses the 
same machinery as \cite{faltings-chai, kai-wen, stroh}, and  we omit those details that 
are adequately documented in the literature.

We use the following notation throughout Section \ref{s:integral models}.
 If $S$ is any  $\co_\kk$-scheme, denote by $i_S:\co_\kk\to \co_S$ the structure morphism.
Define two ideals  $\bm{J},\bm{J}'\subset \co_\kk\otimes_\Z\co_S$ by
\begin{align*}
\bm{J} &= \mathrm{ker} \big(  \co_\kk\otimes_\Z\co_S \map{x\otimes y\mapsto i_S(x) \cdot y }\co_S \big) \\
\bm{J}' &= \mathrm{ker}\big(  \co_\kk\otimes_\Z\co_S \map{x\otimes y\mapsto i_S(\overline{x}) \cdot y }\co_S \big).
\end{align*}
Each is a  locally free  $\co_S$-module of rank one. In fact, if we choose any 
$\pi\in \co_\kk$ for which $\co_\kk=\Z\oplus \Z\pi$, then $\bm{J}$ and $\bm{J}'$
are generated as $\co_S$-modules by 
\[
\bm{j} =   \pi \otimes 1-1\otimes i_S(\pi) 
\qquad
\mbox{and}
\qquad
\bm{j}' =    \pi  \otimes 1-1\otimes i_S(\overline{\pi}) ,
\]
respectively, and the sequence \begin{equation}\label{period exact}
\cdots \map{\bm{j} }   \co_\kk\otimes_\Z\co_S
\map{\bm{j}'}  \co_\kk\otimes_\Z\co_S \map{\bm{j}}  \co_\kk\otimes_\Z\co_S \map{\bm{j}'} \cdots.
\end{equation}
is exact.


\subsection{The work of Pappas and Kr\"amer}
\label{ss:good models}


Fix nonnegative integers $r$ and $s$, and  let 
$\mathcal{M}_{(r,s)}^\naive$ be the moduli stack of  triples  $(A,\kappa,\psi)$, in which 
 \begin{itemize}
\item
$A\to S$ is an abelian scheme   over an $\co_\kk$-scheme $S$,
\item
$\kappa: \co_\kk \to \End(A)$ is an action of $\co_\kk$ on $A$ satisfying the 
\emph{signature $(r,s)$ condition}:   locally on $S$, 
$\kappa(x)$ acts on the  $\co_S$-module $\Lie(A)$ with  characteristic polynomial 
\[
(T-i_S(x))^r(T- i_S(\overline{x}))^s \in \co_S[T]
\]
for every $x\in \co_\kk$,
\item
$\psi :A \to A^\vee$ is an $\co_\kk$-linear principal polarization.
\end{itemize}
The  $\co_\kk$-stack $\mathcal{M}_{(r,s)}^\naive$ is smooth over $\co_\kk[1/d_\kk]$
of relative dimension $rs$, but   at primes dividing $d_\kk$ it  need be neither  flat nor regular.
To remedy this, Pappas \cite{pappas00} defines  
\[
\mathcal{M}_{(r,s)}^\pappas \hookrightarrow \mathcal{M}_{(r,s)}^\naive
\]
as the closed substack of triples $(A,\kappa,\psi) \in \mathcal{M}_{(r,s)}^\naive(S)$
satisfying  \emph{Pappas's wedge conditions}:  the endomorphisms
\[
\wedge^{s+1} \bm{j} : \wedge^{s+1} \Lie(A) \to  \wedge^{s+1} \Lie(A) 
\]
and 
\[
\wedge^{r+1} \bm{j}' : \wedge^{r+1} \Lie(A) \to  \wedge^{r+1} \Lie(A) 
\]
are trivial.

\begin{Rem}\label{naive wedge}
Suppose   $S=\Spec(\F)$ for a field  $\F$, so that any  finitely generated  
$\co_\kk\otimes_\Z \F$-module  is a direct sum of copies of $\co_\kk\otimes_\Z \F$, 
$(\co_\kk\otimes_\Z \F)/\bm{J}$,  and   $(\co_\kk\otimes_\Z \F)/\bm{J}'$.
If $\mathrm{char}(\F) \nmid d_\kk$ then  Pappas's conditions are equivalent to 
\[
\Lie(A) \iso 
\big( (\co_\kk\otimes_\Z \F)/\bm{J}\big)^r  \oplus \big( (\co_\kk\otimes_\Z \F)/\bm{J}'\big)^s.
\]
If $\mathrm{char}(\F)\mid d_\kk$ then $\bm{J}=\bm{J}'$, and Pappas's conditions are equivalent to 
\[
\Lie(A) \iso 
  (\co_\kk\otimes_\Z \F)^a \oplus \big(   (\co_\kk\otimes_\Z \F)/\bm{J}  \big)^b
\]
for some $a\le \mathrm{min}\{ r,s\}$  and $b=r+s-2a$.
\end{Rem}

Fix an integer $m\ge 1$.   In the easy case of signature $(m,0)$
Pappas's wedge conditions are equivalent 
to the single condition $\bm{j}\cdot \Lie(A)=0$, which is equivalent to
$\co_\kk$ acting  on  $\Lie(A)$  through the  structure morphism $i_S:\co_\kk\to \co_S$.  
Proposition \ref{Prop:zero etale} below shows that the stack $\mathcal{M}_{(m,0)}^\pappas$ 
has all the nice properties one could hope for, and hence from now on we will abbreviate
\[
\mathcal{M}_{(m,0)}= \mathcal{M}_{(m,0)}^\pappas.
\]

The category  $\mathcal{M}_{(m,0)}(\C)$ is  easy to describe.
There are only  finitely
many isomorphism classes of objects, and they are in bijection with the isomorphism classes of
self dual Hermitian lattices  $(\mathfrak{B},h_\mathfrak{B})$  of signature $(m,0)$.  The bijection
identifies the pair $(\mathfrak{B},h_\mathfrak{B})$ with the complex torus 
\[
B(\C)=(\mathfrak{B} \otimes_{\co_\kk} \C) /\mathfrak{B}
\]
equipped with its obvious action of $\co_\kk$, and with the polarization induced by 
the perfect symplectic form
 $\psi :  \mathfrak{B}\otimes_\Z\mathfrak{B} \to \Z$ defined by 
\begin{equation}\label{herm to pol}
   \psi(v,w) =  \frac{ h_\mathfrak{B}(v,w) -h_\mathfrak{B}(w,v) } {  \delta_\kk   }.
\end{equation}
Note that each such $B$ is isomorphic to a product of CM elliptic curves:
if we fix an $\co_\kk$-module isomorphism $\mathfrak{B} \iso \oplus \mathfrak{a}_i$,
where each $\mathfrak{a}_i$ is a fractional $\co_\kk$-ideal,  then
$B(\C) \iso \prod_i \C/\mathfrak{a}_i$.  This isomorphism need not 
identify the polarization on the left with the product polarization on the right.

\begin{Prop}\label{Prop:zero etale}
The morphism $\mathcal{M}_{(m,0)} \to \Spec(\co_\kk)$ is proper and smooth  of relative 
dimension $0$.
\end{Prop}

\begin{proof}
Suppose that $z\in \mathcal{M}_{(m,0)}(\F)$ is a geometric point, and let 
$y \in \Spec(\co_\kk)(\F)$ be the geometric point below $z$.  
 To show that $\mathcal{M}_{(m,0)} \to \Spec(\co_\kk)$ is smooth of relative dimension $0$, it suffices to 
 prove that  the corresponding map $R_y\to R_z$ on completed  \'etale local rings  is an isomorphism.
 Let $(B_z,\kappa_z,\psi_z)$ be the triple over $\F$ corresponding to $z$.
The algebraic de Rham homology 
\[
H_1^\mathrm{dR}(B_z/\F) = \Hom( H^1_\mathrm{dR}(B_z/\F) ,\F)
\]
is free of rank  $m$ over $\co_\kk\otimes_\Z \F$.  If $\F$ has characteristic zero this is 
obvious by comparison with  Betti homology.  If $\F$ has characteristic $p>0$ one first checks that
the covariant Dieudonn\'e module $\mathbf{D}_z$ of $B_z$ is free of rank $m$ over 
 $\co_\kk\otimes_\Z W(\F)$, and then uses the canonical isomorphism 
 $\mathbf{D}_z\otimes_{W(\F)} \F \iso H_1^\mathrm{dR}(B_z/\F)$.

Recall the deformation theory of abelian schemes as in  \cite[Chapter 2]{kai-wen}
(which is essentially Grothendieck-Messing theory, but without any mention of divided powers).
Denote by $\mathrm{CLN}$ the category of complete local Noetherian $R_y$-algebras with 
residue field $\F$.  A \emph{square-zero thickening} is a surjection $\tilde{S}\to S$ in $\mathrm{CLN}$
whose kernel $I$ satisfies $I^2=0$.
If $B\to S$ is an abelian scheme then $B$ always admits (by \cite[IV.2.8]{messing72}) 
a deformation  $\tilde{B}\to \tilde{S}$ to any square zero thickening,
and the de Rham homology $D_B(\tilde{S}) = H_1^\mathrm{dR}(\tilde{B}/\tilde{S})$ 
is canonically independent of the choice of $\tilde{B}$.   Of course
\[
D_B(\tilde{S})  \otimes_{\tilde{S}} S \iso D_B(S).
\]
The induced Hodge filtration 
$\mathrm{Fil}\,H_1^\mathrm{dR}(\tilde{B}/\tilde{S}) \subset D_B(\tilde{S})$ \emph{does} 
depend on the choice of $\tilde{B}$.  The fundamental result is that
\[
\tilde{B} \mapsto \mathrm{Fil}\,H_1^\mathrm{dR}(\tilde{B}/\tilde{S}) 
\]
establishes a bijection between the set of deformations of $B$ to $\tilde{S}$ and the 
set of $\tilde{S}$-module direct summands of $D_B(\tilde{S})$ lifting the Hodge filtration 
\[
\mathrm{Fil}\, H_1^\mathrm{dR}(B/S)\subset D_B(S).
\]

Now apply the above discussion with $S=\F$ and $B=B_z$, and abbreviate $D=D_{B_z}$  so that
\begin{equation}\label{freeness}
D(\F) \iso H_1^\mathrm{dR}(B_z/\F) \iso ( \co_\kk\otimes_\Z \F)^m.
\end{equation}
From the Hodge short exact sequence
\[
0\to \mathrm{Fil}\, H_1^\mathrm{dR}(B/S)   \to D(\F) \to \Lie(B_z)\to 0
\]
 and $\bm{j}  \Lie(B_z)=0$,  it is clear that 
 $ \bm{j}  D(\F) \subset  \mathrm{Fil}\, H_1^\mathrm{dR}(B/S)$.  But both are
$\F$-module direct summmands  of $D(\F)$ of  rank $m$, and so 
 $ \bm{j} D(\F) =  \mathrm{Fil}\, H_1^\mathrm{dR}(B/S)$.

 Suppose $\tilde{S}\to \F$ is a square-zero thickening.
 The deformations of $(B_z,\kappa_z,\psi_z)$ to objects of $\mathcal{M}_{(m,0)}(\tilde{S})$ 
 are now  in bijection with the   $\co_\kk$-stable  $\tilde{S}$-module direct summands
 $\mathrm{Fil}\,\subset D(\tilde{S})$ of rank $m$ lifting $\bm{j} D(\F) \subset D(\F)$ that
 satisfy
 \[
 \bm{j}  ( D(\tilde{S})/  \mathrm{Fil})=0,
 \]
 and are isotropic for the symplectic form $\langle\cdot,\cdot\rangle$
 on $D(\tilde{S})$ induced by the polarization 
 $\psi_z$.  Using (\ref{freeness}), a Nakayama's lemma argument shows that
 $D(\tilde{S})\iso (\co_\kk\otimes_\Z \tilde{S})^m$, and it follows easily that 
 \[
 \mathrm{Fil}\,=\bm{j} D(\tilde{S})
 \]
 is the unique such direct summand (the isotropy condition is satisfied because 
 $\langle \bm{j}x,\bm{j}y\rangle = \langle \bm{j}'\bm{j} x,y\rangle$, and 
 $\bm{j}'\bm{j}=0$).   In other words, $(B_z,\kappa_z,\psi_z)$
 admits a unique deformation to an object of $\mathcal{M}_{(m,0)}(\tilde{S})$.

 Repeating the argument   through successive square-zero thickenings in $\mathrm{CLN}$ shows that
 $(B_z,\kappa_z,\psi_z)$ admits a unique deformation to every Artinian object $\tilde{S}$
 of $\mathrm{CLN}$.  Such deformations are classified by $\Hom_{R_y}(R_z,\tilde{S})$, which 
 therefore contains a single point.  Letting $\tilde{S}$ vary over all Artinian quotients of $R_y$ and 
 passing to the inverse limit, we find that $\mathrm{Hom}_{R_y}(R_z,R_y)$ contains 
 a unique element, which is easily seen to be the desired  inverse of $R_y\to R_z$.

The properness claim follows from the valuative criterion.
  \end{proof}

\begin{Rem}
Nowhere in the  proof of Proposition \ref{Prop:zero etale} did we use our permanent 
hypothesis that $d_\kk$ is odd.
\end{Rem}

Pappas proved that the  stack $\mathcal{M}_{(m,1)}^\pappas$  becomes 
regular after blowing up its nonsmooth locus.     Kr\"amer \cite{kramer} then gave a moduli 
interpretation of the blowup, which we now recall.  Define a  stack $\mathcal{M}_{(m,1)}$
as the moduli space of quadruples $(A,\kappa,\psi, \mathcal{F})$ over $\co_\kk$-schemes $S$
in which $(A,\kappa,\psi) \in \mathcal{M}_{(m,1)}^\naive(S)$, and $\mathcal{F}\subset\Lie(A)$
is an $\co_\kk$-stable $\co_S$-submodule  satisfying \emph{Kr\"amer's conditions}:
\begin{itemize}
\item
the quotient sheaf $\Lie(A)/\mathcal{F}$ is a locally free $\co_S$-module of rank one,
\item
the action of $\co_\kk$ on $\mathcal{F}$ is through the 
structure map $i_S:\co_\kk\to \co_S$, while the action  on $\Lie(A)/\mathcal{F}$ is through 
the complex conjugate of the structure map.  
\end{itemize}
The second condition can be rephrased succinctly as  $\bm{j} \mathcal{F}=0$
and $\bm{j}'  \Lie(A) \subset\mathcal{F}$.

The following theorem is a summary of some of the results of Pappas and Kr\"amer 
 on these moduli spaces.

\begin{Thm}[Kr\"amer \cite{kramer}, Pappas \cite{pappas00}]\label{Thm:KP}
The stacks   $\mathcal{M}_{(m,1) }^\pappas$ and 
$\mathcal{M}_{(m,1)}$ are  flat over $\co_\kk$ of relative dimension $m$,
and satisfy the following properties.
\begin{enumerate}
\item
The stack $\mathcal{M}_{(m,1)}$ is regular.

\item
The set $\mathrm{Sing} \subset \mathcal{M}_{(m,1)}^\pappas$ of  points at which 
$\mathcal{M}_{(m,1)}^\pappas\to \Spec(\co_\kk)$ is not smooth has dimension zero,
and is supported in characteristics dividing $d_\kk$.

\item
A geometric point $(A,\kappa,\psi)\in  \mathcal{M}_{(m,1)}^\pappas ( \F )$ 
defines an element of $\mathrm{Sing}$ if and only if 
 $\bm{j} \cdot \Lie(A)=0$.
This condition is equivalent to $\co_\kk$ acting on $\Lie(A)$ through the 
structure morphism $i_\F :\co_\kk \to \F$.

\item
Forgetting the subsheaf $\mathcal{F}$ defines a proper surjection
$\rho : \mathcal{M}_{(m,1)} \to \mathcal{M}_{(m,1)}^\pappas$,
which restricts to an isomorphism
\[
\mathcal{M}_{(m,1)}  \smallsetminus \rho^{-1}( \mathrm{Sing} ) 
\to \mathcal{M}_{(m,1)}^\pappas \smallsetminus \mathrm{Sing}.
\]
The inverse of this isomorphism is $(A,\kappa,\psi) \mapsto (A,\kappa,\psi,\mathcal{F})$,
where
\[
\mathcal{F} = \mathrm{ker}\left( \bm{j} : \Lie(A) \to \Lie(A) \right).
\]

\item
Suppose $z\in  \mathcal{M}_{(m,1)}^\pappas(\F)$
is a geometric point contained in the singular locus $\mathrm{Sing}$.  The fiber 
of $\mathcal{M}_{(m,1)}  \to  \mathcal{M}_{(m,1)}^\pappas$ over $z$
is isomorphic to the projective space $\mathbb{P}^m$ over $\F$.

\item
After base change to  $\co_\kk[1/d_\kk]$, the maps 
\[
 \mathcal{M}_{(m,1)} \to \mathcal{M}_{(m,1)}^\pappas \to \mathcal{M}_{(m,1)}^\naive
\]
become isomorphisms.
\end{enumerate}
\end{Thm}


\subsection{The Kodaira-Spencer map}


Let $\co$ be an $\co_\kk$-algebra, and $S$ a smooth $\co$-scheme.
For any  $S$-valued point
\begin{equation}\label{KS point}
(A,\kappa,\psi) \in \mathcal{M}_{(m,1)}^\pappas(S)
\end{equation}
there is a \emph{Kodaira-Spencer} map
\begin{equation}\label{KS}
\Phi_{\mathrm{KS}}: \Lie(A)^* \otimes_{\co_S} \Lie(A)^* \to \Omega^1_{S/\co}
\end{equation}
defined in \cite[Chapter 2]{kai-wen}. Here $\Lie(A)^*$ is the $\co_S$-dual of $\Lie(A)$.
The Kodaira-Spencer map factors through the quotient sheaf 
\[
 \Lie(A)^* \otimes_{\co_S}  \Lie(A)^* 
 / \langle \substack{    a\otimes b-b\otimes a   \\  (\kappa(x) a) \otimes b - a\otimes (\kappa(\overline{x})b) }   
 :  \substack{  a,b\in \Lie(A)^* \\  x\in \co_\kk }\rangle.
\]
This quotient sheaf is locally free if $d_\kk\in \co^\times$, but this seems to 
be rarely the case otherwise.  Because of this, the statement analogous 
to \cite[Proposition 2.3.5.2]{kai-wen}
 in our setting is not quite correct, but the same proof  yields the following  weaker result.

\begin{Prop}\label{Prop:KS etale}
Suppose the tuple (\ref{KS point})  corresponds to a  morphism 
\[
f: S \to  (\mathcal{M}_{(m,1)}^\pappas\smallsetminus \mathrm{Sing})_{/\co}.
\]
The morphism $f$ is unramified if and only if  the Kodaira-Spencer map 
(\ref{KS}) is surjective.
\end{Prop}


\subsection{Degenerating abelian schemes}


Fix a  projective $\co_\kk$-module 
$\mathfrak{n}$ of rank one,  let  $\underline{\mathfrak{n}}$ be the  associated constant 
$\co_\kk$-module scheme over $\Spec(\Z)$, and denote by $T_\mathfrak{n} = \Spec(\Z[\mathfrak{n}])$ 
the split torus  with character group $\Hom(T_\mathfrak{n},\mathbb{G}_m) \iso \mathfrak{n}.$
In what follows,  $X$ is a stack over $\co_\kk$,  $Z \hookrightarrow X$ is a closed substack, 
and  $U\subset X \smallsetminus Z$  is a dense open substack.

\begin{Def}
A \emph{semi-abelian scheme} over  $X$ is a 
smooth  commutative group scheme $G\to X$, such that for every 
geometric point $\Spec(\F) \to X$  the fiber $G_{/\F}$ is an extension 
\[
0 \to T \to G_{/\F} \to B \to 0
\]
of an abelian variety by a torus.  
\end{Def}

\begin{Def}
A \emph{degenerating abelian scheme of type $\mathfrak{n}$} relative to 
$(X,Z,U)$ is a quadruple $(G,\kappa, \psi ,\mathfrak{n})$ in which 
\begin{itemize}
\item $G$ is a semi-abelian scheme over $X$,  such that $G_{/U}$
is an abelian scheme,
\item
$\kappa : \co_\kk\to \End(G_{/U})$ is an action of $\co_\kk$ on $G_{/U}$,
\item
$\psi :G_{/U} \to G_{/U}^\vee$ is an $\co_\kk$-linear principal polarization,
\item
there is an abelian scheme $B_Z$ over $Z$ equipped with an $\co_\kk$-linear action, and 
an $\co_\kk$-linear exact sequence 
\[
0 \to T_{\mathfrak{n}/Z} \to G_{/Z} \to B_Z \to 0.
\]
\end{itemize}
By \cite[Proposition 3.3.15]{kai-wen}, the action $\kappa : \co_\kk \to \End(G_{/U})$ 
of the second condition  extends uniquely to an action of $\co_\kk$ on $G$,
so the property of $\co_\kk$-linearity in the final condition makes sense.
\end{Def}

\begin{Def}
\emph{Degeneration data of type $\mathfrak{n}$} relative to the triple $(X,Z,U)$
consists of a septuple $(B,\kappa,\psi, \mathfrak{n}, c,c^\vee,\tau)$ 
in which 
\begin{itemize}
\item
$B\to X$ is an abelian scheme,
\item
$\kappa: \co_\kk\to \End(B)$ is an action of $\co_\kk$ on $B$,
\item
$\psi:B\to B^\vee$ is an $\co_\kk$-linear principal polarization,
\item
$\mathfrak{n}$ is a projective $\co_\kk$-module of rank one,
\item
$c:\underline{\mathfrak{n}}_{/X} \to B^\vee$ and $c^\vee:\underline{\mathfrak{n}}_{/X} \to B$
are  $\co_\kk$-module maps  satisfying $c=\psi\circ c^\vee$,
\item
$\tau$ is a positive, symmetric, $\co_\kk$-linear isomorphism
\[
\tau: 1_{ ( \underline{\mathfrak{n}}\times\underline{\mathfrak{n}})_{/U}} \to
(c^\vee\times c)^*\mathcal{P}^{-1} \big|_{  ( \underline{\mathfrak{n}}\times\underline{\mathfrak{n}})_{/U} }
\]
 of  $\mathbb{G}_m$-biextensions on 
 $(\underline{\mathfrak{n}} \times\underline{\mathfrak{n}})_{/U}$.
 Here $\mathcal{P}$ is the Poincare sheaf on $B\times B^\vee$.
 \end{itemize}
\end{Def}

The entry $\tau$ requires some further explanation.  To give a  $\mathbb{G}_m$-biextension on 
$\underline{\mathfrak{n}} \times\underline{\mathfrak{n}}$ is equivalent to giving a collection of 
invertible sheaves 
$
 \{ \mathcal{L}(\mu,\nu) \}_{  (\mu,\nu)\in \mathfrak{n} \times \mathfrak{n} }
 $
  on $X$, together with isomorphisms 
\[
\mathcal{L}(\mu_1+\mu_2,\nu) \iso \mathcal{L}(\mu_1,\nu)\otimes \mathcal{L}(\mu_2,\nu)
\]
 and 
\[
\mathcal{L}(\mu,\nu_1+\nu_2) \iso \mathcal{L}(\mu,\nu_1)\otimes \mathcal{L}(\mu,\nu_2)
\]
 satisfying certain partial group law axioms.   Each pair $(\mu,\nu)$ determines sections
 $ c(\nu): X\to B^\vee$ and $c^\vee(\mu) : X\to B$, and  the biextension 
 $(c^\vee\times c)^*\mathcal{P}^{-1}$ corresponds to the collection of sheaves
 $\mathcal{L}_\mathcal{P}(\mu,\nu)^{-1}$, where
 \[
 \mathcal{L}_\mathcal{P} (\mu,\nu) = (c^\vee(\mu) \times c(\nu))^*\mathcal{P}.
 \] 
There are canonical isomorphisms
 \begin{equation}\label{poincare identities}
 \mathcal{L}_\mathcal{P}(\mu,\nu) \iso \mathcal{L}_\mathcal{P}(\nu,\mu)
 \qquad
 \mathcal{L}_\mathcal{P}(x\mu,\nu) \iso \mathcal{L}_\mathcal{P}(\mu,\overline{x}\nu)
 \end{equation}
 reflecting the symmetry and $\co_\kk$-linearity of the polarization $\psi$.
 The trivial biextension $1_{\underline{\mathfrak{n}} \times \underline{\mathfrak{n}} }$
 corresponds to the constant collection of invertible sheaves $\mathcal{L}_\mathrm{triv}(\mu,\nu) = \co_X$.
Thus  the isomorphism $\tau$  is determined by  a collection of trivializations
\[
\tau(\mu,\nu) : \co_{U}\iso \mathcal{L}_\mathcal{P}(\mu,\nu)^{-1} \big|_{U}.
\]
 The conditions of \emph{symmetry} and 
 \emph{$\co_\kk$-linearity} on $\tau$
 are that $\tau(\mu,\nu)=\tau(\nu,\mu)$ and $\tau(x\mu,\nu) = \tau(\mu,\overline{x}\nu)$ under the 
identifications (\ref{poincare identities}). The condition of \emph{positivity} is that for every 
$\mu\in\mathfrak{n}$,
the isomorphism 
$\tau(\mu,\mu)$ extends  (necessarily uniquely) to a homomorphism
\[
\tau(\mu,\mu) : \co_X\to \mathcal{L}_\mathcal{P}(\mu,\mu)^{-1},
\]
and that if $\mu\not=0$ this homomorphism becomes trivial after restricting to $Z$.

We next recall one of the fundamental results of Mumford 
and  Chai-Faltings \cite{faltings-chai}: an equivalence of categories 
between degenerating abelian schemes and  degeneration data.   
Suppose that $R$  is a Noetherian normal domain complete with respect to 
an ideal $I$ satisfying  $\mathrm{rad}(I)=I$.  For the remainder of this subsection,
\[
(X,Z,U) = ( \Spec(R) , \Spec(R/I) , \{ \eta \} ),
\]
where $\eta$ is the generic point of $\Spec(R)$.
For the proof the following fundamental result, see \cite[Corollary III.7.2]{faltings-chai} 
or \cite[Theorem 5.1.1.4]{kai-wen}.

\begin{Thm}\label{Thm:mumford}
The category of degeneration data relative to $(X,Z,U)$ is equivalent to the 
category of degenerating abelian schemes relative to $(X,Z,U)$ (in both categories,
morphisms  are isomorphisms in the obvious sense).
\end{Thm}

For us, the equivalence of the theorem is mostly a black box.  However, we do need
to know at least some information about how a degenerating abelian scheme 
$(G,\kappa,\psi,\mathfrak{n})$ is related
to its associated degeneration data $(B,\kappa,\psi,\mathfrak{n},c,c^\vee,\tau)$.  
In particular, we need to know how the Lie algebras of $G$ and $B$ are related.
The relation we need is provided by the theory of \emph{Raynaud extensions} as in 
\cite[Chapter II.1]{faltings-chai} or \cite[Section 3.3.3]{kai-wen}:
if we set  $X_\ell=\Spec(R/I^\ell)$, then for every $\ell$ there is an $\co_\kk$-linear short exact sequence
\[
0 \to T_{\mathfrak{n} / X_\ell} \to G_{/X_\ell} \to B_{/X_\ell} \to 0.
\]
Passing to Lie algebras and then taking the inverse limit over $\ell$, we find a short
exact sequence of $R$-modules
\begin{equation}\label{raynaud}
0\to \mathfrak{n}^*\otimes_\Z R \to \Lie(G) \to \Lie(B) \to 0,
\end{equation}
where $\mathfrak{n}^*=\Hom(\mathfrak{n},\Z)$ is the cocharacter group of $T_\mathfrak{n}$.

Our definition of degenerating abelian scheme is rather restrictive.  For example,
it only allows for $G_{/Z}$ to be an extension of an abelian scheme by  a torus of the form $T_\mathfrak{n}$.
That is, a torus whose character group is projective of rank \emph{one} over $\co_\kk$.
The following lemma tells us that such extensions are the only ones that need concern us.
 Keep $(X,Z,U)$ as above, but assume also that $R$ is an $\co_\kk$-algebra.

\begin{Lem}\label{Lem:restrictive degen}
Suppose $(A,\kappa,\psi) \in \mathcal{M}^\pappas_{(m,1)}(U)$,
and that $A = G_{/U}$ for some semi-abelian scheme $G$ over $X$.  Suppose also that
$G_{/Z}$ sits in an $\co_\kk$-linear exact sequence
\[
0\to T_{Z} \to G_{/Z} \to B_Z \to 0
\]
where $T_Z$ is a nontrivial split torus over $Z$, and $B_Z$ is an abelian scheme.  
There is an isomorphism  $T_Z\iso T_{\mathfrak{n}/Z}$
for some projective $\co_\kk$-module $\mathfrak{n}$  of rank one.  
In other words, $(G,\kappa,\psi,\mathfrak{n})$ is a degenerating abelian scheme.
Furthermore, if $(B,\kappa,\psi,\mathfrak{n},c,c^\vee,\tau)$ is the 
corresponding degeneration data, then 
\[
(B,\kappa,\psi) \in \mathcal{M}_{(m-1,0)}(X).
\]
\end{Lem}

\begin{proof}
By the theory of Raynaud extensions alluded to above (and Groethendieck's formal existence
theorem), the group schemes $T_Z$ and $B_Z$ lift to a split torus $T$ over $X$, and an abelian
scheme $B$ over $X$, both with $\co_\kk$-action, in such a way that for every positive integer 
$\ell$ one has an exact  sequence
\[
0 \to T_{ / X_\ell} \to G_{/X_\ell} \to B_{/X_\ell} \to 0.
\]
Taking Lie algebras and passing to the limit, we find an exact sequence
\[
0\to X_*(T) \otimes_\Z R \to \Lie(G) \to \Lie(B) \to 0
\]
where $X_*(T)$ is the cocharacter group of $T$.  We must first show that $X_*(T)$
has rank $1$ as an $\co_\kk$-module.  Suppose not, so that  $X_*(T)\otimes_\Z R_\eta$
is free  of  rank $r \ge 2$ over $\co_\kk\otimes_\Z R_\eta$. 
Let $e_1,e_2,\ldots,e_r$ be generators 
of $X_*(T)\otimes_\Z R_\eta$ as a $\co_\kk\otimes_\Z R_\eta$-module.
Recalling that $\co_\kk=\Z\oplus \Z\pi$, the set  $e_1,\pi e_1,e_2,\pi e_2$ can be extended
to an $R_\eta$-basis of $\Lie(A)=\Lie(G)\otimes_R R_\eta$.  From this it follows that
\[
 (\bm{j} e_1)\wedge (\bm{j} e_2)
= ( \pi e_1 + i_R (\pi) e_1 ) \wedge  ( \pi e_2 + i_R (\pi) e_2 )
\not=0, 
\]
contradicting our assumption that  $(A,\kappa,\psi)$ satisfies Pappas's wedge conditions.  
This shows that $r=1$, and so $T\iso T_\mathfrak{n}$ where $\mathfrak{n}^*\iso X_*(T)$.

Now that  we know $(G,\kappa,\psi,\mathfrak{n})$ is a degenerating abelian scheme,
to complete the proof we must show that the corresponding degeneration data
$(B,\kappa,\psi)$ satisfies  $\bm{j} \Lie(B)=0$.
From (\ref{raynaud}) we find the exact sequence
\[
0 \to \co_\kk\otimes_\Z R_\eta\to \Lie(A) \to \Lie(B) \otimes_R R_\eta \to 0.
\]
But Remark \ref{naive wedge} tells us that
\[
\Lie(A) \iso (\co_\kk\otimes_\Z R_\eta) \oplus \big( (\co_\kk\otimes_\Z R_\eta)/\bm{J} \big)^{m-1}.
\]
It  follows that  $\bm{j}$ kills $\Lie(B) \otimes_R R_\eta$, and hence also kills
 $\Lie(B)$.
\end{proof}

The following lemma is a partial converse to Lemma \ref{Lem:restrictive degen}.

\begin{Lem}\label{Lem:taut extension}
Suppose $(B,\kappa,\psi,\mathfrak{n},c,c^\vee,\tau)$ is degeneration data relative to 
$(X,Z,U)$, and  assume $(B,\kappa,\psi) \in \mathcal{M}_{(m-1,0)} ( X)$.
If  $(G,\kappa,\psi,\mathfrak{n})$ is the associated degenerating abelian scheme, then
\[
(G_{/U} ,\kappa,\psi) \in \mathcal{M}^\pappas_{(m,1)}(U).
\]
Moreover, the $R$-submodule $\mathcal{F} =  \mathrm{ker}\big( \bm{j} : \Lie(G ) \to \Lie(G) \big)$
of $\Lie(G)$ satisfies Kr\"amer's conditions, and in particular 
\[
(G_{/U} ,\kappa,\psi,\mathcal{F}) \in \mathcal{M}_{(m,1)}(U).
\]
\end{Lem}

\begin{proof}
As an $\co_\kk\otimes_\Z R$-module, $\Lie(B)\iso R^{m-1}$, where $\co_\kk$ acts on 
$R$ through the structure map $i_R: \co_\kk\to R$.  That is, $R\iso (\co_\kk\otimes_\Z R)/\bm{J}$.
The exact sequence (\ref{raynaud}) may therefore be rewritten as
\[
0 \to \mathfrak{n}^*\otimes_\Z R \to \Lie(G) \to R^{m-1}  \to 0.
\]
Using the exactness of (\ref{period exact}), the free resolution
\[
\cdots \map{\bm{j} }   \co_\kk\otimes_\Z\co_S
\map{\bm{j}'}  \co_\kk\otimes_\Z\co_S \map{\bm{j}}   \co_\kk\otimes_\Z\co_S 
\map{ x\otimes y\to i_R(x)y }  R \to 0
\]
of $R$ allows us to compute
\[
\mathrm{Ext}^i_{\co_\kk\otimes_\Z R} ( R, \mathfrak{n}^*\otimes_\Z R    ) =0 
\]
for $i>0$.  Thus (\ref{raynaud}) splits  and
\[
\Lie(G) \iso  (\mathfrak{n}^*\otimes_\Z R) \oplus R^{m-1}
\]
as $\co_\kk\otimes_\Z R$-modules.  All claims now follow easily.
\end{proof}


\subsection{Construction of  boundary charts}


Fix a projective $\co_\kk$-module $\mathfrak{n}$ of rank one.
We now construct a closed immersion of $\co_\kk$-stacks
$s_\mathfrak{n}:Z_\mathfrak{n} \hookrightarrow X_\mathfrak{n}$
equipped with tautological degeneration data of type $\mathfrak{n}$ relative to 
$(X_\mathfrak{n},Z_\mathfrak{n},U_\mathfrak{n})$, where 
$U_\mathfrak{n}=X_\mathfrak{n}\smallsetminus Z_\mathfrak{n}$.
The stack $X_\mathfrak{n}$ will have the structure of a 
 line bundle  over $Z_\mathfrak{n}$, and  $s_\mathfrak{n}$ is the zero section.

Using \cite[Proposition 5.2.3.9]{kai-wen} we may define a smooth proper stack 
\[
Z_\mathfrak{n} = \Hom_{\co_\kk} 
\big(  \underline{\mathfrak{n}}_{/S}  , B ^\vee \big),
\]
over $\mathcal{M}_{(m-1,0)}$ of relative dimension $m-1$, 
where $(B,\kappa,\psi)$ is the universal object over $\mathcal{M}_{ (m-1,0) }$. 
After pulling back the triple $(B,\kappa,\psi)$ to $Z_\mathfrak{n}$ one
obtains a tautological $\co_\kk$-linear morphism 
$c:\underline{\mathfrak{n}}  \to B^\vee$ of stacks over $Z_\mathfrak{n}$. 
Set  $c^\vee = \psi^{-1} \circ c:\underline{\mathfrak{n}} \to  B.$
The group $\mu(\kk)$ acts on $Z_\mathfrak{n}$ via its natural action on $\mathfrak{n}$.

Our assumption that $d_\kk$ is odd implies that 
\[
Q_\mathfrak{n} = \mathrm{Sym}^2_\Z (\mathfrak{n})/ \langle (x\mu)\otimes \nu- \mu\otimes(\overline{x}\nu) 
: x\in \co_\kk, \mu,\nu \in \mathfrak{n}\rangle,
\]
is a free $\Z$-module of rank one.  Let $E_\mathfrak{n}$ be the torus over $\Spec(\Z)$ defined by
\[
E_\mathfrak{n}\iso \Hom (Q_\mathfrak{n} ,\mathbb{G}_m).
\]
The group law on $Q_\mathfrak{n}$ will be written additively.
An ordered pair $(\mu,\nu) \in \mathfrak{n}\times\mathfrak{n}$ 
determines sections $\mu,\nu : Z_\mathfrak{n} \to \underline{\mathfrak{n}}$, which 
in turn define morphisms $c^\vee(\mu) : Z_\mathfrak{n} \to B$ and 
$c(\nu):Z_\mathfrak{n} \to B^\vee$. The pullback 
\[
\mathcal{L}_\mathcal{P} (\mu,\nu) = (c^\vee(\mu)\times c(\nu))^*\mathcal{P}
\]
 of the   Poincare sheaf $\mathcal{P}$ on $B\times B^\vee$
is an invertible sheaf on $Z_\mathfrak{n}$.  Up to canonical isomorphism,
the sheaf $\mathcal{L}_\mathcal{P} (\mu,\nu)$ depends only on the image of 
$\mu\otimes\nu$ in $Q_\mathfrak{n}$.
Thus we may associate an invertible sheaf $\mathcal{L}_\mathcal{P}(q)$ to each 
$q\in Q_\mathfrak{n}$  in such a way that
\[
\mathcal{L}_\mathcal{P}(q_1 + q_2)  \iso 
\mathcal{L}_\mathcal{P}(q_1)\otimes\mathcal{L}_\mathcal{P}(q_2)
\]
 canonically.   These isomorphisms define an $\co_{Z_\mathfrak{n}}$-algebra structure
 on the sheaf
 $ \co_{U_\mathfrak{n}} =  \bigoplus_{q\in Q_\mathfrak{n}} \mathcal{L}_\mathcal{P}(q)$, 
which is therefore  the structure sheaf of a $Z_\mathfrak{n}$-stack
\[
U_\mathfrak{n} = \underline{\Spec}_{Z_\mathfrak{n}}
 \Big( \bigoplus_{q\in Q_\mathfrak{n}} \mathcal{L}_\mathcal{P}(q) \Big).
\]
For each $q\in Q_\mathfrak{n}$, multiplication in $\co_{U_\mathfrak{n}}$ defines an 
isomorphism of $\co_{Z_\mathfrak{n}}$-modules
$
\mathcal{L}_\mathcal{P}(q) \otimes \co_{U_\mathfrak{n}} \to  \co_{U_\mathfrak{n}}.
$
Thus, after pullback via $U_\mathfrak{n}\to  Z_\mathfrak{n}$ each of the 
sheaves $\mathcal{L}_\mathcal{P}(q)$ acquires a canonical trivialization, 
and dualizing yields isomorphisms
\[
\tau(q) : \co_{U_\mathfrak{n}} \to \mathcal{L}_\mathcal{P}(q)^{-1}
\]
of sheaves of $\co_{U_\mathfrak{n}}$-modules.
This collection of isomorphisms defines an
isomorphism of $\mathbb{G}_m$-biextensions
\[
\tau : 1_{\underline{\mathfrak{n}} \times \underline{\mathfrak{n}} }\to 
(c^\vee \times c)^*\mathcal{P}^{-1}
\]
over $(\underline{\mathfrak{n}} \times \underline{\mathfrak{n}})_{/U_\mathfrak{n}}$.

The $\R$-vector space $Q_{\mathfrak{n},\R}=Q_\mathfrak{n}\otimes_\Z\R$ has a  notion of positivity 
\[
Q_{\mathfrak{n},\R}^{\ge 0} = \{ \mu\otimes\mu \in Q_{\mathfrak{n},\R} : \mu\in \mathfrak{n}_\R \},
\]
and there is an induced ordering $\ge$ on $Q_\mathfrak{n}$. 
   Define a partial compactification  $U_\mathfrak{n}\to X_\mathfrak{n}$ by
\[
X_\mathfrak{n} = \underline{\Spec}_{Z_\mathfrak{n}}  \Big( \bigoplus_{q \ge 0 } 
\mathcal{L}_\mathcal{P} (q)  \Big).
\]
The ideal sheaf 
\[
\mathcal{I}_\mathfrak{n} =  \bigoplus_{   q> 0  }\mathcal{L}_\mathcal{P}(q)
\]
defines a closed substack  of $X_\mathfrak{n}$ canonically identified with $Z_\mathfrak{n}$.
On the other hand, the inclusion $\co_{Z_\mathfrak{n}} \to \co_{X_\mathfrak{n}}$ induces a morphism
$X_\mathfrak{n} \to Z_\mathfrak{n}$, giving $X_\mathfrak{n}$ the structure of a 
vector bundle over $Z_\mathfrak{n}$.  In particular,
$X_{\mathfrak{n}}$ is smooth of relative dimension $m$ over $\co_\kk$.
We  now have tautological degeneration data 
$(B,\kappa,\psi,\mathfrak{n},c,c^\vee,\tau)$ 
relative to  $( X_\mathfrak{n} , Z_\mathfrak{n},U_\mathfrak{n})$. 
Each line bundle $\mathcal{L}_\mathcal{P}(q)$  is invariant under the action of $\mu(\kk)$
on $Z_\mathfrak{n}$, and hence the action lifts to an action of $\mu(\kk)$ on  $X_\mathfrak{n}$.

We have now constructed a smooth and proper divisor
\begin{equation}\label{the boundary}
 \bigsqcup_\mathfrak{n}  \mu(\kk) \backslash Z_{ \mathfrak{n} } \hookrightarrow
\bigsqcup_\mathfrak{n}  \mu(\kk) \backslash X_{ \mathfrak{n} },
\end{equation}
where the disjoint unions are over the isomorphism classes of projective $\co_\kk$-modules of 
rank one.  This divisor will soon become the boundary of our compactification of $\mathcal{M}_{(m,1)}$.


\subsection{Attaching the boundary}


For each  geometric point $z$  of  $Z_{\mathfrak{n} }$ let 
$R_z=\co_{X_\mathfrak{n},z}$
 be the \'etale local ring of $X_\mathfrak{n}$ at $z$,
and let $I_z \subset R_z$ be the ideal defined by 
the divisor $Z_\mathfrak{n} \to  X_\mathfrak{n}$.
Let $R_z^\wedge$ be the completion of  $R_z$ with respect to  $I_z$, and 
let $\eta_z$ and $\eta_z^\wedge$ denote the generic points of $R_z$ and $R_z^\wedge$, 
respectively.  As $X_\mathfrak{n}$ is 
smooth over $\co_\kk$, the $\co_\kk$-algebras $R_z$ and $R_z^\wedge$ are Noetherian normal domains.
Applying Theorem \ref{Thm:mumford} to  the pullback of the tautological degeneration data relative to 
$(X_\mathfrak{n},Z_\mathfrak{n},U_\mathfrak{n})$, one obtains  a degenerating abelian scheme 
$( {}^\heartsuit G_z,  {}^\heartsuit \kappa_z,  {}^\heartsuit \psi_z,  \mathfrak{n} )$
relative to 
\[
(X_z^\wedge, Z_z^\wedge, U_z^\wedge)=
(\Spec(R_z^\wedge) , \Spec(R_z^\wedge/I_z) ,  \{ \eta_z^\wedge\} ).
\]

For every   \'etale neighborhood  $X^{(z)} \to  X_\mathfrak{n}$ of a 
geometric point $z$, 
define a closed subscheme of $X^{(z)}$ by
\[
Z^{(z)}  = Z_\mathfrak{n} \times_{X_\mathfrak{n}}   X^{(z)},
\]
and an open subscheme
\[
U^{(z)} =U_\mathfrak{n} \times_{X_\mathfrak{n}}   X^{(z)}.
\]

\begin{Prop}
For every geometric point $z$ of $Z_\mathfrak{n}$
there is an \'etale neighborhood 
\[
X^{(z)} \to  X_\mathfrak{n}
\] 
of $z$ and  a degenerating abelian scheme $(G^{(z)},\kappa^{(z)},\psi^{(z)},\mathfrak{n})$
relative to $(X^{(z)} , Z^{(z)} , U^{(z)})$
with the following properties.
\begin{enumerate}
\item
There exists a ring automorphism  $\gamma: R_z^\wedge\to R_z^\wedge$ 
inducing the identity on $R_z^\wedge/I_z$ such that 
\[
(G^{(z)},\kappa^{(z)},\psi^{(z)} )_{/R_z^\wedge}  \iso 
\gamma^*( {}^\heartsuit G_z,  {}^\heartsuit \kappa_z ,  {}^\heartsuit \psi_z ),
\]
where the left hand side is the pullback of $(G^{(z)},\kappa^{(z)} ,\psi^{(z)} )$
via the canonical map $\Spec(R_z^\wedge) \to X^{(z)}$, and the
right hand side is the pullback of 
$( {}^\heartsuit G_z,  {}^\heartsuit \kappa_z ,  {}^\heartsuit \psi_z )$ via 
$\gamma: R_z^\wedge\to R_z^\wedge$.
\item
The  tuple $(G^{(z)},\kappa^{(z)} ,\psi^{(z)})_{/U^{(z)}}$
defines an \'etale morphism  
\[
U^{(z)} \to \mathcal{M}^\pappas_{(m,1)  } \smallsetminus\mathrm{Sing}.
\]
\item
The subsheaf $\mathcal{F}^{(z)}= \mathrm{ker} \big( \bm{j} : \Lie(G^{(z)} )  \to \Lie( G^{(z)}  \big)$
of $\Lie( G^{(z)} )$ satisfies Kr\"amer's conditions.
\end{enumerate}
\end{Prop}

\begin{proof}
The degenerating abelian scheme
$( {}^\heartsuit G_z,  {}^\heartsuit \kappa_z,   {}^\heartsuit \psi_z,   \mathfrak{n} )$
need not descend to a degenerating abelian scheme relative to
\begin{equation}\label{incomplete triple}
(X_z,Z_z,U_z) = (\Spec(R_z) , \Spec(R_z/I_z) , \{ \eta_z \} )
\end{equation}
  but it can be approximated arbitrarily closely
by degenerating abelian schemes that do descend.  This means that, as in
\cite[Proposition IV.4.3]{faltings-chai} or \cite[Proposition 6.3.2.1]{kai-wen},
there is  a degenerating abelian scheme  $(  G_z,   \kappa_z,   \psi_z,   \mathfrak{n} )$
relative to (\ref{incomplete triple})
and  a ring  automorphism $\gamma\in \Aut(R_z^\wedge)$ inducing the identity 
on $R_z^\wedge/I_z$ such that
\[
 (  G_z,   \kappa_z , \psi_z   )_{/X_z^\wedge}
\iso 
\gamma^*( {}^\heartsuit G_z,  {}^\heartsuit \kappa_z ,  {}^\heartsuit \psi_z ).
\]
Moreover, if we denote by  $\co_{\kk,z}$  the strict Henselization of $\co_\kk$ at $z$,
the Kodaira-Spencer map
\[
\Lie(G_{z/U_z})^*  \otimes_{R_{ z , \eta_z} }  \Lie(G_{z/U_z})^*
\to  \big( \Omega^1_{ X_z/\co_{\kk,z} }  \big)_{/U_z}
\]
of (\ref{KS}) extends, as in \cite[Proposition 6.2.5.18]{kai-wen}, to  a surjection of $R_z$-modules
\begin{equation}\label{mumford etale}
 \Lie(G_{z})^*  \otimes_{R_z}  \Lie(G_{z})^*   \to \Omega^1_{ X_z/\co_{\kk,z}}(d\log\infty),
\end{equation}
where
$
\Omega^1_{ X_z/\co_{\kk,z}}(d\log\infty) = 
\Omega^1_{ X_z/\co_{\kk,z}} \otimes_{R_z} I^{-1}_{\mathfrak{n},z}.
$

As in \cite[Proposition IV.4.4]{faltings-chai} or  
\cite[Proposition 6.3.2.6]{kai-wen},   there is  some \'etale neighborhood $X^{(z)}$ of 
$z$ in $X_\mathfrak{n}$  such that   $(G_z,\kappa_z,\psi_z)$
descends to a degenerating abelian scheme $(G^{(z)},\kappa^{(z)},\psi^{(z)},\mathfrak{n})$
relative to $(X^{(z)} , Z^{(z)} , U^{(z)})$. By   Lemma \ref{Lem:taut extension},
the subsheaf 
\[
\mathcal{F}^{(z)} = \mathrm{ker} \big(
\bm{j} : \Lie(G^{(z)}) \to \Lie(G^{(z)})
\big)
\]
satisfies  Kr\"amer's conditions locally at $z$, 
and so after shrinking $X^{(z)}$ we may assume that it satisfies these conditions
everywhere on $X^{(z)}$. In particular, 
\[
(G^{(z)},\kappa^{(z)},\psi^{(z)},\mathcal{F}^{(z)} )_{/U^{(z)}} \in
\mathcal{M}_{(m,1)}(U^{(z)}),
\]
and the map
\[
U^{(z)} \to \mathcal{M}^\pappas_{(m,1)}
\]
corresponding to $(G^{(z)},\kappa^{(z)},\psi^{(z)})$ does not meet the singular
locus $\mathrm{Sing}$ (this is clear from the characterization of $\mathrm{Sing}$ found in Theorem
\ref{Thm:KP}).  Combining the surjectivity of (\ref{mumford etale}) with 
Proposition \ref{Prop:KS etale} shows that after further shrinking $X^{(z)}$ the map
$U^{(z)} \to \mathcal{M}^\pappas_{(m,1)}$ is unramified.   By \cite[Corollary 6.3.1.13]{kai-wen}
it is also \'etale.
\end{proof}

As  $\mathcal{M}_{(m,1)} \to \mathcal{M}^\pappas_{(m,1)}$
is an isomorphism away from the closed set $\mathrm{Sing}$,
each of the maps  $U^{(z)} \to \mathcal{M}^\pappas_{(m,1)}$ of the 
proposition admits a unique lift to an \'etale morphism 
\[
U^{(z)} \to \mathcal{M}_{(m,1)}.
\]
 By the quasi-compactness of $Z_\mathfrak{n}$
we may choose finitely many geometric  points  $z$  so that the union of the images of 
$X^{(z)} \to X_\mathfrak{n}$ cover $Z_\mathfrak{n}$.
Letting $\mathfrak{n}$ vary over all isomorphism classes of 
projective $\co_\kk$-modules of rank one, let $\mathcal{X}$ be the 
disjoint union of   the finitely many   $X^{(z)}$'s so constructed,
and let $\mathcal{U}\subset \mathcal{X}$ be the disjoint union of the finitely many
$U^{(z)}$'s.  The obvious map
\[
\mathcal{M}_{(m,1)} \sqcup \mathcal{U}    \to  \mathcal{M}_{(m,1)}
\]
 is an   \'etale surjection, and realizes 
$\mathcal{M}_{(m,1)}$ as the quotient of 
$\mathcal{M}_{(m,1)}\sqcup \mathcal{U}$ by an
\'etale equivalence relation 
\[
\mathcal{R}_0 \to 
( \mathcal{M}_{(m,1)} \sqcup \mathcal{U}  ) \times_{ \co_\kk } 
  ( \mathcal{M}_{(m,1)} \sqcup \mathcal{U}  ) .
\]

The normalization of 
$( \mathcal{M}_{(m,1)} \sqcup \mathcal{X}  )
 \times_{ \co_\kk  }   ( \mathcal{M}_{(m,1)} \sqcup \mathcal{X}  )$ 
in $\mathcal{R}_0$ 
defines a new stack, $\mathcal{R}$, sitting in the commutative diagram
\[
\xymatrix{
{\mathcal{R}}  \ar[d]_{ r } &  {\mathcal{R}_0 }   \ar[l]  \ar[d]  \\
 {    ( \mathcal{M}_{(m,1)} \sqcup \mathcal{X}  ) \times_{ \co_\kk }  
  ( \mathcal{M}_{(m,1)} \sqcup \mathcal{X}  )  } 
  & {  ( \mathcal{M}_{(m,1)} \sqcup \mathcal{U}  ) \times_{ \co_\kk  } 
    ( \mathcal{M}_{(m,1)} \sqcup \mathcal{U}  ) . }    \ar[l]
}
\]
Exactly as in \cite[Proposition IV.5.4]{faltings-chai} or \cite[Proposition 6.3.3.13]{kai-wen}, 
the  morphism $r$  is an \'etale equivalence relation.  
Let $\mathcal{M}^*_{(m,1)}$ be the quotient of   $\mathcal{M}_{(m,1)} \sqcup \mathcal{X}$
by  $r$. The following theorem summarizes the important properties of $\mathcal{M}^*_{(m,1)}$,
and the proofs are exactly as in  \cite[Theorem 6.4.1.1]{kai-wen}.

\begin{Thm}\label{Thm:compactify}
The stack $\mathcal{M}^*_{(m,1)}$ is regular.  It is proper and flat over $\co_\kk$
of relative dimension $m$,  is smooth over $\co_\kk[1/d_\kk]$, and contains
$\mathcal{M}_{(m,1)}$  as an open dense substack.  The boundary
\[
\partial\mathcal{M}^*_{(m,1)} =  \mathcal{M}^*_{(m,1)}\smallsetminus  \mathcal{M}_{(m,1)}
\]
is a smooth divisor,   isomorphic to  the stack $\bigsqcup_\mathfrak{n} \mu(\kk)\backslash Z_\mathfrak{n}$ 
of (\ref{the boundary}), and the formal 
completion of $\mathcal{M}^*_{(m,1)}$ along $\partial\mathcal{M}^*_{(m,1)} $
is isomorphic to the formal completion of $\bigsqcup_\mathfrak{n} \mu(\kk)\backslash X_\mathfrak{n}$ along 
$\bigsqcup_\mathfrak{n} \mu(\kk)\backslash Z_\mathfrak{n}$.

The universal abelian scheme with $\co_\kk$-action  over
 $\mathcal{M}_{(m,1)}$   extends to a  semi-abelian scheme with $\co_\kk$-action,
 $G$, over  $\mathcal{M}^*_{(m,1)}$. At a closed geometric point $\Spec(\F) \to \partial\mathcal{M}^*_{(m,1)}$ 
 of the boundary,  the semi-abelian scheme $G$ is an  extension
\begin{equation}\label{ext boundary}
0 \to T \to G_{/\F} \to B\to 0
\end{equation}
of an abelian scheme $B$ by a torus. The character group of $T$ is a projective 
$\co_\kk$-module of rank one, and $B$ extends in a canonical way 
to a triple $(B,\kappa,\psi) \in \mathcal{M}_{(m-1,0)}(\F)$.

The universal subsheaf $\mathcal{F}\subset \Lie(A)$ 
 extends canonically to a subsheaf $\mathcal{F}\subset \Lie(G)$, which again satisfies 
 Kr\"amer's conditions.   On the complement of the divisor
 $\rho^{-1}(\mathrm{Sing}) \subset  \mathcal{M}^*_{(m,1)}$, the subsheaf $\mathcal{F}$ is the
 kernel of  $\bm{j} : \Lie(G) \to \Lie(G)$.
\end{Thm}


\subsection{Galois action on boundary components}
\label{ss:boundary labels}


Recall from Section \ref{ss:good models}
that  $\mathcal{M}_{(m-1,0)}(\C)$ has finitely
many isomorphism classes of objects, which are indexed by isomorphism classes of
self dual Hermitian lattices  $(\mathfrak{B},h_\mathfrak{B})$  of signature $(m-1,0)$.  
Under this correspondence
\begin{equation}\label{hida labels}
(\mathfrak{B},h_\mathfrak{B}) \mapsto (\mathfrak{B} \otimes_{\co_\kk} \C) /\mathfrak{B}
\end{equation}
where the complex torus on the right hand side is 
equipped with its obvious action of $\co_\kk$, and with the principal polarization induced by 
the symplectic form (\ref{herm to pol}).

\begin{Def}\label{Def:cusp label}
A \emph{cusp label} is an isomorphism class of pairs $(\mathfrak{n} , \mathfrak{B})$ in which 
$\mathfrak{n}$ is a projective $\co_\kk$-module of rank one, and $\mathfrak{B}$ is a self dual 
Hermitian  lattice of signature $(m-1,0)$.
\end{Def}

There is a   bijection between  the connected components of $\partial\mathcal{M}^*_{(m,1)}(\C)$ and the cusp labels
$(\mathfrak{n},\mathfrak{B})$.  As $z$ varies over the  component of $\partial\mathcal{M}^*_{(m,1)}(\C)$ 
indexed by $(\mathfrak{n},\mathfrak{B})$ the extensions (\ref{ext boundary}) take the form
\begin{equation}\label{cm degeneration}
0 \to T_{\mathfrak{n}/\C}  \to G_z \to  (\mathfrak{B} \otimes_{\co_\kk} \C) /\mathfrak{B}  \to 0.
\end{equation}

  For an $\co_\kk$-module $M$, write $M_\kk=M\otimes_{\co_\kk} \kk$.
Given a fractional $\co_\kk$-ideal $\mathfrak{s}$ and a positive definite self dual
Hermitian lattice $\mathfrak{B}$, we obtain a  new positive definite self dual
Hermitian lattice $\mathfrak{B}\otimes \mathfrak{s}^{-1}$.  
The new Hermitian form is 
\[
h_{\mathfrak{B}\otimes \mathfrak{s}^{-1}}( x,y) = \mathrm{Nm}(\mathfrak{s})
\cdot h_\mathfrak{B}( x,y),
\]
where we  identify $ ( \mathfrak{B}\otimes\mathfrak{s}^{-1})_\kk \iso \mathfrak{B}_\kk$
using the obvious isomorphism $(\mathfrak{s}^{-1})_\kk \iso \kk$.
Let $H$ be the Hilbert class field of $\kk$, $\mathrm{Cl}_\kk$ the ideal class group, and 
\[
\mathrm{rec}_\kk : \mathrm{Cl}_\kk \iso \Gal(H/\kk)
\] 
the reciprocity map of  class field theory.

\begin{Prop}\label{Prop:galois bnd action}
The action of $\Gal(\kk^\alg/ \kk)$ on the  components of $\partial\mathcal{M}^*_{(m,1)/\kk^\alg}$
factors through $\Gal(H/\kk)$. 
 For any  $\mathfrak{s}\in \mathrm{Cl}_\kk$, the Galois automorphism $\mathrm{rec}_\kk(\mathfrak{s})$
carries the irreducible component indexed by $(\mathfrak{n},\mathfrak{B})$ to the 
component indexed by $(\mathfrak{n},\mathfrak{B}\otimes\mathfrak{s}^{-1})$.
\end{Prop}

\begin{proof}
The Galois automorphism
\[
 \mathcal{M}_{(m-1,0)}(\C) \map{\sigma} \mathcal{M}_{(m-1,0)}(\C)
\]
corresponds, under the bijection (\ref{hida labels}),  to 
\[
(\mathfrak{B},h_\mathfrak{B}) \mapsto 
(\mathfrak{B} \otimes \mathfrak{s}^{-1} ,h_{ \mathfrak{B} \otimes\mathfrak{s}^{-1} } ),
\]
 where $\mathfrak{s}$  is any fractional ideal representing the image of $\sigma$ under
\[
\Aut(\C/\kk) \to \Gal(H/\kk) \iso \mathrm{Cl}_\kk.
\]
Of course this is identical to the formula for the Galois action on elliptic curves with complex multiplication;
the details of the proof  can be found  in \cite{zavosh}.
 Both claims now follow by applying   $\sigma$ throughout (\ref{cm degeneration}).
 \end{proof}

There is an alternate way to parametrize the components of  $\partial\mathcal{M}^*_{(m,1)}(\C)$. 
Suppose we start with a pair $\mathfrak{m} \subset \mathfrak{A}$ in which $\mathfrak{A}$ is a self dual Hermitian
lattice of signature $(m,1)$  and $\mathfrak{m}$ is an  isotropic direct summand of rank one.  
A \emph{normal decomposition} of $\mathfrak{m} \subset \mathfrak{A}$ is an $\co_\kk$-module direct 
sum decomposition 
\[
\mathfrak{A} = \mathfrak{m} \oplus \mathfrak{B} \oplus \mathfrak{n}
\]
in which  $\mathfrak{n}$ is an isotropic direct summand of rank one, and
$\mathfrak{B} = ( \mathfrak{m}\oplus \mathfrak{n} )^\perp$.  The Hermitian form on $\mathfrak{A}$ 
restricts to a perfect  pairing  $\mathfrak{m} \times \mathfrak{n}\to \co_\kk$, and makes  
$\mathfrak{B} \iso  \mathfrak{m} ^\perp / \mathfrak{m}$   into a positive
definite self dual Hermitian lattice of signature $(m-1,0)$.

\begin{Prop}\label{Prop:normal decomp}
Every pair $\mathfrak{m} \subset \mathfrak{A}$ as above  admits a  normal decomposition.  The 
rule  
\[
\mathfrak{m} \subset \mathfrak{A} \mapsto (  \mathfrak{A}/\mathfrak{m}^\perp  ,  \mathfrak{m}^\perp/\mathfrak{m}  )  
\iso  (\mathfrak{n},\mathfrak{B})
\] 
establishes a bijection between 
the isomorphism classes of pairs $\mathfrak{m} \subset \mathfrak{A}$ as above, and the set of cusp labels.
\end{Prop}

\begin{proof}
Denote by $\langle\cdot,\cdot\rangle$ the Hermitian form on $\mathfrak{A}$.
It is easy to see that $\mathfrak{m}^\perp$ is a projective $\co_\kk$-module of rank $m$, and that
$\mathfrak{A}/\mathfrak{m}^\perp$ is projective of rank $1$.  Hence 
\[
\mathfrak{A}=\mathfrak{m}^\perp\oplus \mathfrak{n}
\]
for  some rank one direct summand $\mathfrak{n} \subset \mathfrak{A}$.   In particular
$\langle \mathfrak{m} , \mathfrak{n} \rangle\not=0$.  We now modify $\mathfrak{n}$ to 
make it isotropic.
Fix  $e\in \mathfrak{m}_\kk$ and $e'\in \mathfrak{n}_\kk$ such that  
$\langle e,e'\rangle=1$.    There are fractional $\co_\kk$-ideals $\mathfrak{n}_0$ and
$\mathfrak{m}_0$ such that $\mathfrak{m}=\mathfrak{m}_0 e$ and 
$\mathfrak{n}=\mathfrak{n}_0 e'$.  The Hermitian 
form on $\mathfrak{A}$ restricts to a perfect pairing between $\mathfrak{n}$ and $\mathfrak{m}$,
and hence  $\overline{ \mathfrak{m}}_0 \mathfrak{n}_0 =\co_\kk$.

\emph{Because we assume that $d_\kk$ is odd}, the trace map $\co_\kk \to \Z$ is surjective.
It is easy to see that $\langle e',e'\rangle \in \mathrm{N}(\mathfrak{m}_0) \Z$, and therefore
we may choose an $x\in \mathrm{N}( \mathfrak{m}_0) \co_\kk$ with 
$x+\overline{x}  = - \langle e',e'\rangle$.
Now replace $\mathfrak{n}$ by  $\mathfrak{n}_0 ( xe+e')$.  It is still true that 
$\mathfrak{A}=\mathfrak{m}^\perp \oplus \mathfrak{n}$, but now $\mathfrak{n}$ is isotropic.  
Defining $\mathfrak{B}=( \mathfrak{m} \oplus \mathfrak{n})^\perp$  gives the desired normal 
decomposition.

For the second claim, start with a cusp label $(\mathfrak{n},\mathfrak{B})$.  Let $\mathfrak{m}$
be the set of $\co_\kk$-conjugate linear maps $\mathfrak{n} \to \co_\kk$, and let $\co_\kk$ act
on $\mathfrak{m}$ by $(x\cdot \mu)(\nu)=x\cdot \mu(\nu)$.
Define a Hermitian form on $\mathfrak{m} \oplus \mathfrak{n}$ by 
$\langle \mu_1+ \nu_1 , \mu_2+ \nu_2  \rangle = \mu_1(\nu_2) + \overline{ \mu_2(\nu_1)}$, and make
$\mathfrak{A}=\mathfrak{m}\oplus \mathfrak{B} \oplus \mathfrak{n}$ into a Hermitian lattice in the 
obvious way, with $\mathfrak{B} \perp (\mathfrak{m} \oplus\mathfrak{n})$.
The construction $(\mathfrak{n},\mathfrak{B}) \mapsto  \mathfrak{m} \subset \mathfrak{A}$ is
surjective, by the existence of normal decompositions, and 
is  easily seen to give an inverse to the map in the statement of the proposition.
\end{proof}

\begin{Rem}
Proposition \ref{Prop:normal decomp} is false without the hypothesis that $d_\kk$ is odd.
For example, endow $\mathfrak{A}=\co_\kk\oplus\co_\kk$ with the signature $(1,1)$ Hermitian form
\[
\langle x,y\rangle = {}^t x  \left[\begin{matrix} 
0 & 1 \\
1 & 1
\end{matrix}\right]  \overline{y},
\]
and let 
\[
\mathfrak{m} = \co_\kk\cdot \left[\begin{matrix}
1 \\ 0
\end{matrix}\right].
\]
If $d_\kk$ is even, the pair $\mathfrak{m} \subset \mathfrak{A}$ does not admit a normal decomposition.
\end{Rem}


\section{Arithmetic divisors}


Fix a positive integer $n$  and consider the  flat and regular $\co_\kk$-stack 
\[
\mathcal{M}  =  \mathcal{M}_{ (1,0)  } \times_{\co_\kk} \mathcal{M}_{(n-1,1)  }
\]
of (absolute) dimension $n$,  and its flat and regular compactification
\begin{equation*}
\mathcal{M}^* = \mathcal{M}_{(1,0) } \times_{\co_\kk} \mathcal{M}^*_{(n-1,1)}
\end{equation*}
 constructed in Section \ref{s:integral models}.  The boundary
 $\partial\mathcal{M}^* = \mathcal{M}^* \smallsetminus \mathcal{M}$
 (with its reduced substack structure)  is a divisor on $\mathcal{M}^*$, proper and smooth over $\co_\kk$.

The stack $\mathcal{M}$ is the moduli stack of septuples 
$(A_0,\kappa_0,\psi_0,A,\kappa,\psi,\mathcal{F})$
in which $A_0 \to S$ is an elliptic curve over an $\co_\kk$-scheme, 
$\psi_0$ is its  principal polarization,
$\kappa_0:\co_\kk\to \End(A_0)$ is an action of $\co_\kk$ on $A_0$ whose induced action 
on $\Lie(A_0)$ is through the structure morphism $i_S:\co_\kk \to \co_S$, $A$
is an abelian scheme over $S$, $\kappa:\co_\kk\to \End(A)$ is an action of $\co_\kk$ on $A$,
$\psi$ is an $\co_\kk$-linear principal polarization of $A$, and $\mathcal{F}\subset \Lie(A)$
satisfies Kr\"amer's conditions. For simplicity
we will usually shorten such a septuple simply to $(A_0,A) \in \mathcal{M}(S)$.

 We will construct  special 
 divisors $\mathcal{Z}(m)$ on $\mathcal{M}$.  These divisors  come equipped with 
 natural Green  functions, and we will analyze the behavior of these Green functions 
 near the boundary of  $\mathcal{M}^*$.  The goal is  to show that   adding a
 particular  $\R$-linear combination of  boundary components  to $\mathcal{Z}(m)$ 
  yields  a class in the arithmetic Chow group  $\widehat{\mathrm{CH}}_\R^1(\mathcal{M}^*)$.


\subsection{Arithmetic Chow groups}
\label{ss:BKK}


The work of Gillet and Soul\'e 
\cite{BGS,gillet-soule90,soule92} gives us, for any flat, regular, and proper $\co_\kk$-scheme 
of finite type, a theory of   \emph{arithmetic Chow groups} formed from  cycles equipped with Green currents,
up to  a suitable notion of rational equivalence.
Burgos-Kramer-K\"uhn \cite{bruinier-burgos-kuhn,BKK} extended this theory  by 
allowing the use of  Green currents with  certain mild \emph{log-log} singularities 
along a fixed normal crossing  divisor.  The original Gillet-Soul\'e theory  was extended 
from schemes to algebraic stacks by Gillet \cite{gillet09}, but that work does not 
allow for the log-log singularities covered in \cite{BKK}.
We will be working with  Green functions on the stack $\mathcal{M}^*$  
with  log-log singularities along the boundary,  so we need a 
common generalization of \cite{gillet09} and \cite{BKK}.  Doing this in full generality
is a task best left to the experts, so we will develop in an ad hoc way only the 
minimal theory we need.

\begin{Def}
Suppose $M^*$ is a complex manifold of dimension $n-1$,  $\partial M^*\subset M^*$ is a smooth 
codimension one submanifold,  $M=M^*\smallsetminus \partial M^*$,
and $z_0\in \partial M^*$.  On some open neighborhood $V\subset M^*$ of $z_0$ there are
coordinates $q,u_1,\ldots,u_{n-2}$ such that $\partial M^*$ is defined by the equation $q=0$.
After shrinking $V$ we may always assume that $\log|q^{-1}| >1$ on $V$.
The open set $V$ and its coordinates are then said to be \emph{adapted to $\partial M^*$}.
\end{Def}

\begin{Def}
Suppose  $f$ is a  $C^\infty$ function on an 
open subset   $U\subset M$.  We say that 
$f$ has \emph{log-log growth along $\partial M^*$} if around any point of $\partial M^*$ there is an 
open neighborhood $V\subset M^*$ and coordinates  $q,u_1,\ldots,u_{n-2}$ adapted to $\partial M^*$ such that 
\begin{equation}\label{log-log def}
f =O ( \log \log|q^{-1} | )
\end{equation}
on $U\cap V$.
A smooth differential form  $\omega$ on $U$ has \emph{log-log growth along $\partial M^*$}
if around any point of $\partial M^*$ there is an  open neighborhood $V\subset M^*$ and coordinates  
$q,u_1,\ldots,u_{n-2}$ adapted to $\partial M^*$ such that $\omega|_{U \cap V}$ lies in the 
subring (of the ring of all smooth forms on $U\cap V$) generated by 
\begin{align*}
\frac{d q}{ q \log|q| }, \frac{d \overline{q}}{ \overline{q} \log|q| }, 
d u_1, \ldots, d u_{n-2} , d\overline{u}_1, \ldots,  d\overline{u}_{n-2},
\end{align*}
and the functions satisfying (\ref{log-log def}).
\end{Def}

\begin{Rem}
Our definition of log-log growth is slightly stronger than that  of \cite{BKK}.
What we call log-log growth, those authors call \emph{Poincar\'e growth} \cite[Section 7.1]{BKK}.
\end{Rem}

\begin{Rem}
A smooth function $f$ on an open subset of $M$
is a  \emph{pre-log-log form} if $f$, $\partial f$, $\overline{\partial} f$, and 
$\partial \overline{\partial} f$ all have log-log growth along $\partial M^*$.  We will 
not use this terminology, but mention it  for ease of comparison with \cite{BKK},
where this term is used routinely.
\end{Rem}

The notion of log-log growth can be extended from complex manifolds to the 
orbifold fibers of $\mathcal{M}^*$ in the  following way.  
It is always possible to write  $\mathcal{M}^*(\C)$ as the 
quotient of a complex manifold $M^*$ by the action of a finite group $H$.    For example,
by (as in \cite{kai-wen, kai-wen_2})  
adding level structure to the moduli problem defining $\mathcal{M}(\C)$
and then compactifying the result.
In any case, the morphisms of complex orbifolds
\[
 \partial\mathcal{M}^*(\C)  \to \mathcal{M}^*(\C) \leftarrow  \mathcal{M}(\C)
\]
arise as the quotients  of $H$-invariant morphisms of complex manifolds
\[
\partial M^* \to M^* \leftarrow M
\]
with $H$-actions.  A smooth form on the orbifold 
$\mathcal{M}(\C)$ pulls back to  a smooth $H$-invariant form on
the complex manifold $M$, and is said to have
 \emph{log-log growth} along the boundary $\partial\mathcal{M}^*(\C)$
if the corresponding $H$-invariant form on $M$ has log-log growth along $\partial M^*$.

Let $\mathcal{Z}$ be a divisor on $\mathcal{M}^*$ with real coefficients, and write
$\mathcal{Z} =\sum m_i \mathcal{Z}_i$
as a finite $\R$-linear combination of  pairwise distinct
irreducible closed substacks of codimension one.
We allow the possibility that some $\mathcal{Z}_i$ are  components of the boundary $\partial\mathcal{M}^*$.
In our applications such boundary components  will appear
with real multiplicities, while the non-boundary $\mathcal{Z}_i$'s  will have
integer multiplicities.

\begin{Def}
A \emph{Green function} for $\mathcal{Z}$ consists of a smooth real-valued function
 $\Gr(\mathcal{Z},\cdot)$ on
$
\mathcal{M}(\C)  \smallsetminus  \mathcal{Z} (\C) .
$
satisfying the following properties:
\begin{enumerate}
\item
The function $\Gr(\mathcal{Z} ,z)$ has a logarithmic singularity
along $\mathcal{Z}$ in the following sense:  around every  point of $\mathcal{M}^*(\C)$ there is an
open neighborhood $V$ and local equations $\psi_i(z)=0$ for the divisors $\mathcal{Z}_i(\C)$
such that the function 
\[
\mathcal{E}(z)= \Gr(\mathcal{Z} ,z) + \sum_i m_i \log |\psi_i(z) |^2
\]
on $V \cap \big( \mathcal{M}(\C) \smallsetminus  \mathcal{Z}(\C) \big)$
extends smoothly to $V\cap \mathcal{M}(\C)$.
\item
The  forms $\mathcal{E}$, $\partial\mathcal{E}$, $\overline{\partial}\mathcal{E}$, and 
$\partial \overline{\partial}\mathcal{E}$  on $V\cap \mathcal{M}(\C)$ have log-log growth along $\partial\mathcal{M}^*(\C)$.
\end{enumerate}
\end{Def}

As a simple example, if $f$ is any rational function on $\mathcal{M}^*$ then $- \log|f|^2$
is a  Green function for the divisor $\mathrm{div}(f)$.

\begin{Rem}
As in \cite{mumford77} or \cite[Proposition 7.6]{BKK},
the log-log growth of $\mathcal{E}$ and of
\[
\partial \overline{\partial} \mathcal{E} = -2\pi i \cdot dd^c \mathcal{E}
\] 
imply that both $\mathcal{E}$ and $dd^c \mathcal{E}$ are locally integrable, and so define
currents on $ \mathcal{M}^*(\C)$.   The log-log growth conditions further imply the equality
of currents $[dd^c \mathcal{E}] = dd^c [\mathcal{E}]$, from which one  deduces the 
Green equation
\[
 [dd^c\mathcal{E}] = dd^c [\Gr(\mathcal{Z},\cdot )] + \delta_{ \mathcal{Z} } 
\]
of \cite[Definition 1.2.3]{gillet-soule90}. 
\end{Rem}

\begin{Def}
An \emph{arithmetic divisor} on $\mathcal{M}^*$ is a pair $(\mathcal{Z},\Gr(\mathcal{Z},\cdot))$
consisting of a divisor $\mathcal{Z}$ on $\mathcal{M}^*$ with real coefficients, 
and a Green function for $\mathcal{Z}$.  
A \emph{principal arithmetic divisor} is 
an arithmetic divisor of the form 
\[
\widehat{\mathrm{div}}(f) = (\mathrm{div}(f) , -\log|f|^2)
\] 
for some rational function $f$ on $\mathcal{M}^*$.  
The \emph{codimension one arithmetic Chow group}
$\widehat{\mathrm{CH}}^1_\R(\mathcal{M}^*)$ is the quotient of the 
$\R$-vector space of all arithmetic divisors by the $\R$-span of the principal
arithmetic divisors.
\end{Def}

Let  $\mathcal{X}$ be a regular algebraic stack,  finite and flat over $\co_\kk$.  In particular
$\mathcal{X}$ has dimension one.  The stack 
$\mathcal{X}$ has an arithmetic Chow group $\widehat{\mathrm{CH}}_\R^1(\mathcal{X})$  defined 
exactly as for $\mathcal{M}^*$ (taking $\mathcal{X}=\mathcal{X}^*$ and $\partial\mathcal{X}^*=\emptyset$, as the finiteness of $\mathcal{X}$ implies
that it is already proper over $\co_\kk$).  
 Of course the study of Green functions for divisors on $\mathcal{X}$
is simplified by the fact that $\mathcal{X}$ has all of its divisors are supported in nonzero
characteristic: a Green function for a divisor $\mathcal{Z}$ on $\mathcal{X}$ is simply any
 function on the $0$-dimensional orbifold
$\mathcal{X}(\C)$.

Now suppose $\mathcal{X}$ is equipped with a representable morphism $\pi:\mathcal{X}\to\mathcal{M}$.
The finiteness of $\mathcal{X}$ over $\co_\kk$ implies that $\pi$ is a proper map.  
In particular, the Zariski closure of the image of $\mathcal{X}$ in
$\mathcal{M}^*$ does not meet the boundary. We will construct a pullback map 
\[
\pi^* : \widehat{\mathrm{CH}}^1_\R(\mathcal{M}^*) \to \widehat{\mathrm{CH}}^1_\R(\mathcal{X}).
\]
Suppose  $\mathcal{Z}$ is an irreducible divisor on $\mathcal{M}^*$ intersecting 
$\mathcal{X}$ properly, in the sense that
\[
\mathcal{X} \cap \mathcal{Z} = \mathcal{X} \times_{\mathcal{M}^*} \mathcal{Z}
\] 
has dimension $0$.    Of course this is equivalent to 
$\mathcal{X}$ and $\mathcal{Z}$ having empty intersection in the generic fiber of $\mathcal{M}^*$.
The   \emph{Serre intersection multiplicity} 
at a geometric point $z \in (\mathcal{X} \cap \mathcal{Z})(k)$  is defined by
\[
I^\serre_z (\mathcal{Z} : \mathcal{X}) = 
 \sum_{ \ell\ge 0 }  (-1)^\ell 
\mathrm{length}_{ \co_{\mathcal{X} \cap \mathcal{Z} ,z}  }   \mathrm{Tor}_\ell^{  \co_{\mathcal{M},z}   }  
\big( 
\co_{\mathcal{X},z} , \co_{\mathcal{Z} ,z }
 \big),
\]
where all local rings are for the \'etale topology.    From $z$ we may construct 
a divisor $[z]$ on $\mathcal{X}$  as follows.  Fix an \'etale presentation 
$X\to \mathcal{X}$ with $X$ a scheme.  The fiber product $\tilde{z} = X \times_\mathcal{X} z$ is
a scheme, and is finite \'etale over $z=\Spec(k)$.  Thus $\tilde{z}$ is a disjoint union of 
copies of $\Spec(k)$, say $\tilde{z}=\bigsqcup z_i$, where each $z_i$ is a geometric point of 
$X$.  Let $[z_i]$ denote the image of $z_i:\Spec(k) \to X$, so that $[z_i]$ is a closed point of 
$X$.  Then $[\tilde{z}] = \sum [z_i]$ is a divisor on $X$, and descends uniquely to a divisor on $\mathcal{X}$
denoted $[z]$.    Define a divisor on $\mathcal{X}$
\[
\pi^*\mathcal{Z} = \sum_{ z \in |\mathcal{X} \cap \mathcal{Z}| } I_z^\serre(\mathcal{Z}:\mathcal{X})  \cdot [z].
\]
Here $|\mathcal{X} \cap \mathcal{Z}|$ is the topological space underlying the
$\co_\kk$-stack $\mathcal{X} \cap \mathcal{Z}$, in the sense of \cite[Chapter 5]{laumon}.  
Each point $z\in |\mathcal{X} \cap \mathcal{Z}|$ is, by definition,  an equivalence class of maps 
$z:\Spec(k) \to \mathcal{X} \cap \mathcal{Z}$ with $k$ a field, 
and we may always choose a representative of this equivalence class
for which $k$ is algebraically closed.
Extend the  definition of $\pi^*\mathcal{Z}$ linearly to all divisors $\mathcal{Z}$ with real coefficients whose 
support meets $\mathcal{X}$ properly.
If  $\Gr(\mathcal{Z},\cdot)$ is a Green function for  such a divisor $\mathcal{Z}$, then
the image of the orbifold morphism $\mathcal{X}(\C) \to \mathcal{M}^*(\C)$ is disjoint 
from the divisor $\mathcal{Z}(\C)$, as well as from the boundary of $\mathcal{M}^*(\C)$.
 Thus the image of 
$\mathcal{X}(\C) \to \mathcal{M}(\C)$ is disjoint from all singularities of $\Gr(\mathcal{Z},\cdot)$,
and so we may form the pullback  $\pi^* \Gr(\mathcal{Z},\cdot)$ 
in the usual sense.  This defines 
\[
\pi^*\widehat{\mathcal{Z}} = ( \pi^* \mathcal{Z}, \pi^* \Gr(\mathcal{Z},\cdot))
\] 
whenever  $\widehat{\mathcal{Z}} = (\mathcal{Z},\Gr(\mathcal{Z},\cdot))$ with $\mathcal{Z}$ 
intersecting $\mathcal{X}$ properly.  

The argument of  \cite[Theorem III.3.1]{soule92} 
allows one  to extend the definition of $\pi^*$ to all arithmetic divisors.
Briefly, given an arithmetic divisor $\widehat{\mathcal{Z}}=(\mathcal{Z},\Gr(\mathcal{Z},\cdot) )$
one can use Chow's moving lemma \cite{roberts72} (working on an \'etale presentation of 
the generic fiber of $\mathcal{M}^*$)
in order to modify $\widehat{\mathcal{Z}}$
by a principal arithmetic divisor in such a way that the resulting divisor meets $\mathcal{X}$
properly, and  then check that the resulting pullback
does not depend on the choice of principal arithmetic divisor used in the modification.
It may be worth pointing out that there is a gap in the proof of  \cite[Theorem III.3.1]{soule92}, 
identified and corrected by Gubler  \cite{gubler}.  The gap is in the proof of the
``Moving Lemma for $K_1$-chains", but  if one only works with Chow groups of codimension one 
(as we do) then  the use of  this lemma is unnecessary, and there is no gap.

There is a canonical linear functional 
\[
\widehat{\deg}: \widehat{\mathrm{CH}}^1_\R(\mathcal{X}) \to \R,
\]
defined, if $\mathcal{X}$ is a scheme, in \cite{gillet-soule90} as the composition
\[
\widehat{\mathrm{CH}}^1_\R(\mathcal{X}) \to \widehat{\mathrm{CH}}^1_\R(\Spec(\co_\kk) ) \to \R,
\]
where the first arrow is the proper pushforward by the structure morphism $\mathcal{X}\to\Spec(\co_\kk)$,
and the second arrow is defined in \cite[Section 3.4.3]{gillet-soule90}.
The generalization to stacks can be found in \cite{KRY}.  In any case,
\[
\widehat{\deg} (\mathcal{Z},\Gr(\mathcal{Z},\cdot) ) =
\sum_{ \mathfrak{q} \subset \co_\kk }
\sum_{z\in \mathcal{Z}(\F_\mathfrak{q}^\alg) } \frac{ \log ( \mathrm{N}(\mathfrak{q}))}{\#\Aut_\mathcal{X}(z)}
+   \sum_{z\in \mathcal{X}(\C)} 
\frac{  \Gr(\mathcal{Z} , z) }{ \#\Aut_{\mathcal{X}(\C) }(z) } 
\]
whenever $\mathcal{Z}$ is  irreducible.  Here $\F_\mathfrak{q}^\alg$ is an algebraic closure
of the residue field $\co_\kk/\mathfrak{q}$, and $\mathrm{N}(\mathfrak{q}) = \# (\co_\kk/\mathfrak{q})$.

\begin{Def}
Suppose  $\mathcal{X}$ is a regular stack, finite and flat over $\co_\kk$,
and equipped with a morphism $\pi:\mathcal{X} \to \mathcal{M}$.
\emph{Arithmetic intersection against $\mathcal{X}$} is the linear functional
\[
[\cdot : \mathcal{X} ] : \widehat{\mathrm{CH}}^1_\R(\mathcal{M}^*) \to \R
\]
defined as the composition
\[
\widehat{\mathrm{CH}}^1_\R(\mathcal{M}^*) \map{\pi^*} \widehat{\mathrm{CH}}^1_\R(\mathcal{X}) 
  \map{\widehat{\deg}} \R.
\]
\end{Def}

The arithmetic intersection $[\cdot : \mathcal{X}]$  is characterized by 
\begin{equation}\label{proper degree}
[  \widehat{ \mathcal{Z} } : \mathcal{X} ]  = 
  I_\mathrm{fin} (\mathcal{Z} : \mathcal{X})  +
\Gr(\mathcal{Z}, \mathcal{X}) 
\end{equation}
for every arithmetic divisor $\widehat{\mathcal{Z}} = (\mathcal{Z} , \Gr(\mathcal{Z},\cdot) )$ 
on $\mathcal{M}^*$  such that $\mathcal{Z}$ and $\mathcal{X}$ intersect properly, where
\[
I_\mathrm{fin} ( \mathcal{Z} : \mathcal{X}) = 
\sum_{\mathfrak{q} \subset\co_\kk}  
 \sum_{ z\in (\mathcal{X} \cap\mathcal{Z})(\F_\mathfrak{q}^\alg )  }  
\frac{ \log(\mathrm{N}(\mathfrak{q}) ) }{ \# \Aut_\mathcal{X} (z) }   I_z^\serre (\mathcal{Z} : \mathcal{X}) 
\]
and
\[
\Gr(\mathcal{Z}, \mathcal{X}) = \sum_{ z\in \mathcal{X}(\C) }
\frac{ \Gr(\mathcal{Z},z)}{\#\Aut_{\mathcal{X}(\C) }(z) } .
\]


\subsection{Kudla-Rapoport divisors}


The stack $\mathcal{M}$ admits an obvious morphism to
\[
\mathcal{M}^\naive  = \mathcal{M}_{(1,0)} \times_{\co_\kk} \mathcal{M}_{(n-1,1)}^\naive.
\]
This stack is  a moduli space of sextuples $(A_0,\kappa_0,\psi_0,A,\kappa,\psi)$,
but we usually abbreviate such a sextuple to 
\[
(A_0,A) \in \mathcal{M}^\naive(S).
\]
The $\co_\kk$-module $\Hom_{\co_\kk}(A_0,A)$ is equipped with a positive definite
Hermitian form 
\[
\langle f_1,f_2\rangle =  \psi_0^{-1}\circ  f_2^\vee \circ  \psi\circ  f_1 .
\]
The right hand side is an element of $\End_{\co_\kk}(A_0)\iso \co_\kk$.

\begin{Def}
For any $m\not=0$, the \emph{naive Kudla-Rapoport divisor}  
$\mathcal{Z}^\naive(m)$ is the moduli stack of tuples 
$(A_0,A,f)$ over $\co_\kk$-schemes $S$,
  in which
\begin{itemize}
\item
$(A_0,A) \in \mathcal{M}^\naive(S)$
\item
$f\in \Hom_{\co_\kk}(A_0,A)$ satisfies $\langle f , f \rangle =m$.
\end{itemize}
Of course $\mathcal{Z}^\naive(m) =\emptyset$ if $m<0$.  Denote by
\[
\mathcal{Z}(m) = \mathcal{Z}^\naive(m) \times_{\mathcal{M}^\naive } \mathcal{M}
\]
the pullback of $\mathcal{Z}^\naive(m)$ to  $\mathcal{M}$.   
\end{Def}

The morphism  
$\mathcal{Z}(m)\to \mathcal{M}$ is finite and unramified,
and (in light of Proposition \ref{Prop:local divisor} below)
we view $\mathcal{Z}(m)$ as a divisor on $\mathcal{M}$ in the usual way.
To be more precise, $\mathcal{Z}(m)$ is a closed substack of itself, and so 
determines a cycle $[\mathcal{Z}(m)]$ on  $\mathcal{Z}(m)$ by \cite[Definition (3.5)]{vistoli}.
We then pushforward this cycle to a divisor on $\mathcal{M}^*$ using \cite[Definition (3.6)]{vistoli}.

\begin{Def}
For every $m\not=0$ the \emph{Kudla-Rapoport divisor} $\mathcal{Z}^*(m)$
is the Zariski closure of the divisor $\mathcal{Z}(m)$ in the compactification 
$\mathcal{M}^*$ of $\mathcal{M}$. 
\end{Def}

The following result tells us not only that $\mathcal{Z}^*(m)$ has the 
expected dimension, but also, as we shall see in the proof of Theorem \ref{Thm:finite part},
ensures that the higher $\mathrm{Tor}$ terms in the Serre intersection multiplicity
will not contribute to our final formulas.

 \begin{Prop}\label{Prop:local divisor}
 Suppose  $\F$ is an algebraically closed field,
 $z\in \mathcal{Z}(m)(\F)$ is a geometric point, and 
 denote by $R_z$ the completed \'etale local ring of $\mathcal{Z}(m)$ at $z$. 
 Let  $y\in \mathcal{M}(\F)$ be the point below $z$, and denote by $R_y$
 the completed \'etale local ring of $\mathcal{M}$ at $y$.  The natural map $R_y \to R_z$
 is surjective, and the kernel is generated by a single nonzero element.
 
In particular,  the local rings of  $\mathcal{Z}(m)$ are complete intersections, and its irreducible components
all have dimension $n-1$.
\end{Prop}
  
\begin{proof}  
Let $(A_0,A,f)$ be the triple over $\F$ determined by the point $z$.   
Let $\widehat{\co}_{\kk,z}$ be the completion of the strict Henselization of $\co_\kk$ with respect
to the geometric point 
\[
\Spec(\F) \map{z} \mathcal{Z}(m) \to \Spec(\co_\kk),
\]
 and let  $\mathrm{CLN}$ be the category of complete local Noetherian 
 $\widehat{\co}_{\kk,z}$-algebras with residue field $\F$.  
 The ring $R_y$ represents the functor assigning to every 
object $S$ of $\mathrm{CLN}$ the set of isomorphism classes of deformations of 
$(A_0,A)$ to $S$.  Here deformation always means deformation to an object of 
$\mathcal{M}(S)$; that is, the deformations of $A_0$ and $A$ are also equipped with 
$\co_\kk$-actions, polarizations, etc.~ lifting those on $A_0$ and $A$.  Similarly, 
$R_z$ represents the deformation functor of the triple $(A_0,A,f)$. 
The surjectivity of the
tautological map $R_y \to R_z$ is equivalent to the injectivity of
\[
\Hom(R_z, S) \to \Hom(R_y,S)
\] 
for every object $S$ of $\mathrm{CLN}$, and this injectivity is proved in  
\cite[Lemma 2.2.2.1]{kai-wen}.  Let 
\[
I = \mathrm{ker} ( R_y \to R_z),
\]
let $\mathfrak{m}$ be the maximal ideal of $R_y$, and set $S=R_y/\mathfrak{m}I$.
The kernel of the natural surjection $S \to R_z$ is $\mathcal{I}=I/\mathfrak{m}I$, and 
satisfies $\mathcal{I}^2=0$.  By Nakayama's lemma, to prove that $I$ is a principal ideal, it suffices to 
prove that $\mathcal{I}$ is.

Let $(\mathbf{A}_0,\mathbf{A})$ be the universal deformation of $(A_0,A)$ to $R_y$.
The reduction $(\mathbf{A}_{0/R_z} , \mathbf{A}_{/R_z})$ comes with a universal
map $\mathbf{f} : \mathbf{A}_{0/R_z}  \to \mathbf{A}_{/R_z}$, and we will 
use the deformation theory arguments of \cite[Chapter 2]{kai-wen} to
 study the obstruction to lifting $\mathbf{f}$ to a map $\mathbf{A}_{0/S} \to \mathbf{A}_{/S}$.
 As in the proof of Proposition \ref{Prop:zero etale}, the de Rham homology groups
$H_1^\mathrm{dR}(\mathbf{A}_0)$ and $H_1^\mathrm{dR}(\mathbf{A})$  are free of ranks 
$1$ and $n$ over $\co_\kk\otimes_\Z R_y$, and sit in short exact sequences
\[
0 \to \mathrm{Fil}\, H_1^\mathrm{dR} ( \mathbf{A}_0) 
\to H_1^\mathrm{dR} ( \mathbf{A}_0)  \to \Lie(\mathbf{A}_0) \to 0
\]
and 
\[
0 \to \mathrm{Fil}\, H_1^\mathrm{dR} ( \mathbf{A}) 
\to H_1^\mathrm{dR} ( \mathbf{A})  \to \Lie(\mathbf{A}) \to 0.
\]
Furthermore, again by the proof of Proposition \ref{Prop:zero etale},
\[
\mathrm{Fil}\, H_1^\mathrm{dR} ( \mathbf{A}_0)  
= \bm{j}\cdot H_1^\mathrm{dR} ( \mathbf{A}_0).
\]
The same holds with $\mathbf{A}_0$ and $\mathbf{A}$ replaced by their reductions to $S$ or 
$R_z$. 
Fix once and for all
an $\co_\kk\otimes_\Z R_y$-module generator 
\[
\sigma\in H_1^\mathrm{dR} ( \mathbf{A}_0),
\]
and a basis $\epsilon_1,\ldots, \epsilon_n$ of $\Lie(\mathbf{A})$ such that
$\epsilon_1,\ldots, \epsilon_{n-1}$ is a basis of the universal $R_y$-submodule 
$\mathcal{F} \subset \Lie(\mathbf{A})$ satisfying Kr\"amer's conditions.  
In particular $\bm{j}\cdot \epsilon_i =0$ for $1\le i<n$, and  the operator 
$\bm{j}$ on $\Lie(\mathbf{A} )$ has the form
\begin{equation}\label{j mat}
\bm{j} = \left[\begin{matrix}
0& \cdots & 0 & j_1 \\
\vdots & \ddots & \vdots & \vdots \\
0& \cdots & 0 & j_n
\end{matrix}\right]
\end{equation}
for some $j_1,\ldots, j_n\in R_y$.

The map $\mathbf{f}$ induces a map
\[
\mathbf{f} : H_1^\mathrm{dR} ( \mathbf{A}_{0/R_z}) \to H_1^\mathrm{dR} ( \mathbf{A}_{/R_z})
\]
respecting the Hodge filtrations, and by \cite[Proposition 2.1.6.4]{kai-wen} there is  a canonical lift
\[
\tilde{\mathbf{f}} : H_1^\mathrm{dR} ( \mathbf{A}_{0/S}) \to H_1^\mathrm{dR} ( \mathbf{A}_{/S})
\]
(which need not respect the Hodge filtrations).  The obstruction to lifting $\mathbf{f}$
to a map $\mathbf{A}_{0/S} \to \mathbf{A}_{/S}$ is given by the composition
\[
 \bm{j}\cdot H_1^\mathrm{dR} ( \mathbf{A}_{0/S}) \to 
 H_1^\mathrm{dR} ( \mathbf{A}_{0/S}) \map{\tilde{ \mathbf{f} }} 
  H_1^\mathrm{dR} ( \mathbf{A}_{/S}) \to \Lie(\mathbf{A}_{/S}).
\]
which we denote by $\mathrm{obst}_\mathbf{f}$.  
The image of $\tilde{\mathbf{f}}(\sigma)$ in $\Lie(\mathbf{A}_{/S})$ is 
$x_1 \epsilon_1+\cdots +x_n \epsilon_n$ for some $x_1,\ldots,x_n\in S$,
and 
\[
\mathrm{obst}_\mathbf{f}(\bm{j}\sigma) = x_1 \bm{j} \epsilon_1 + \cdots +x_n\bm{j} \epsilon_n
=    x_n \bm{j} \epsilon_n  = x_n( j_1 \epsilon_1 + \cdots + j_n \epsilon_n).
\]  
The map $\mathrm{obst}_\mathbf{f}$ becomes trivial after applying $\otimes_S R_z$,
and hence the ideal $x_n (j_1, \ldots ,  j_n)$ of $S$ is contained in $\mathcal{I}$.  
On the other hand,  the map $\mathrm{obst}_\mathbf{f}$ becomes trivial upon reduction to  
$S/x_n (j_1,\ldots, j_n)$, which implies that the universal triple 
$(\mathbf{A}_{0/R_z}, \mathbf{A}_{/R_z},\mathbf{f})$ lifts to $S/x_n(j_1,\ldots, j_n)$.
By the universality of $(\mathbf{A}_{0/R_z}, \mathbf{A}_{/R_z},\mathbf{f})$, this 
lift corresponds to a section to the natural surjection $S/x_n(j_1,\ldots, j_n) \to R_z$,
which is therefore an isomorphism.  In other words
\[
\mathcal{I}=x_n (j_1,\ldots, j_n).
\]
To complete the proof that $I$ is principal, it now suffices to  
show that $(j_1,\ldots,j_n)$
is a principal ideal of $R_y$. After Theorem \ref{Thm:KP}, it is natural to consider separately the cases
 $\bm{j} \cdot\Lie(A)\not=0$ and $\bm{j} \cdot\Lie(A)=0$.

The case $\bm{j} \cdot\Lie(A)\not=0$ is easy.  This assumption  implies
that (\ref{j mat}) is nonzero after reduction to the residue field $\F$, and hence
at least one of the  $j_i$'s is a unit in $R_y$.  In particular
$(j_1,\ldots, j_n)=R_y$ is a principal ideal.

Now assume that $\bm{j} \cdot\Lie(A)=0$.  By Theorem \ref{Thm:KP},  this implies that
$\mathrm{char}(\F) \mid d_\kk$.  Let $\mathfrak{p}$ be the kernel of the structure
map $i_\F:\co_\kk\to \F$, and let $p$ be the rational prime below $\mathfrak{p}$.  
As we assume that $d_\kk$ is odd, we may fix a uniformizer 
$\pi\in \co_{\kk,\mathfrak{p}}$ in such a way that $\overline{\pi}=-\pi$.
Note that $i_\F(\pi)=0$, and so $\bm{j}=\pi$ as endomorphisms of $\Lie(A)$.  In particular
$\pi\Lie(A)=0$.  Using the fact that   $H_1^\mathrm{dR}(A)$ is free of rank $n$ over 
$\co_\kk \otimes_\Z \F$, it is easy to deduce from the exactness of 
\[
0 \to \mathrm{Fil}\, H_1^\mathrm{dR}(A) \to H_1^\mathrm{dR}(A) \to \Lie(A) \to 0
\]
that  
\begin{equation}\label{degenerate hodge}
 \mathrm{Fil}\, H_1^\mathrm{dR}(A) = \pi \cdot H_1^\mathrm{dR}(A).
\end{equation}

Using the coordinates of \cite{kramer,pappas00}, we examine the structure 
of the universal $R_y$-module short exact sequence
\[
0 \to \mathrm{Fil}\, H_1^\mathrm{dR}(\mathbf{A}) \to
H_1^\mathrm{dR}(\mathbf{A}) \to \Lie(\mathbf{A}) \to 0,
\]
and of the universal submodule $\mathcal{F}\subset \Lie(\mathbf{A})$.  
Fix an isomorphism $H_1^\mathrm{dR}(\mathbf{A}) \iso (\co_{\kk,\mathfrak{p}} \otimes_{\Z_p} R_y)^n$,
and identify $\co_{\kk,\mathfrak{p}} \otimes_{\Z_p} R_y = R_y \oplus \pi R_y$.
This gives a decomposition
\[
H_1^\mathrm{dR}(\mathbf{A}) \iso R_y^n \oplus \pi R_y^n.
\]
By \cite[Lemma 3.6]{pappas00}, these choices may be made in such a way that the 
symplectic form on the left hand side determined by the principal polarization on $\mathbf{A}$
is identified with the symplectic form $\psi$ on the right hand side determined by
$\psi(e_i,e_j)=0$ and $\psi(e_i,\pi e_j)=\delta_{i,j}$.  Here
$e_1,\ldots, e_n\in R_y^n$ are the standard basis vectors.
Combining (\ref{degenerate hodge}) with Nakayama's lemma shows that the composition
\[
 \mathrm{Fil}\, H_1^\mathrm{dR}(\mathbf{A}) \to R_y^n \oplus \pi R_y^n \to \pi R_y^n
\]
is an isomorphism (the first arrow is the inclusion, the second the projection).  This implies 
that there is a unique $X\in M_n(R_y)$ such that the vectors
\begin{align*}
v_1 & = X e_1 - \pi e_1 \\
& \vdots \\
v_n &= X e_n - \pi e_n
\end{align*}
are a basis for $\mathrm{Fil}\,H_1^\mathrm{dR}(\mathbf{A})$, and such that the 
images of $e_1,\ldots, e_n$ in $\Lie(\mathbf{A})$ form a basis.  With respect to this 
basis, the action of $\pi$ on $\Lie(\mathbf{A})$ is through the matrix $X$.  The Hodge filtration
$\mathrm{Fil}\, H_1^\mathrm{dR}(\mathbf{A})$ is isotropic for the symplectic form $\psi$,
and one easily checks that this implies  ${}^t X=X$.

The matrices (\ref{j mat}) and $X-i_{R_y}(\pi)$ are conjugate by an element of $\mathrm{GL}_n(R_y)$,
as they  represent the operator $\pi - i_{R_y}(\pi)$ with respect to different bases of 
$\Lie(\mathbf{A})$.  Noting that (\ref{j mat}) factors as
\[
\bm{j} = j\cdot {}^te_n,
\] 
where ${}^tj = [\begin{matrix} j_1 & \cdots & j_n \end{matrix}]$, there is a 
$\gamma\in \mathrm{GL}_n(R_y)$, such that 
\[
 \gamma j\cdot {}^t e_n \gamma^{-1} = X-i_{R_y}(\pi).
\]
Define  $v,w\in R_y^n$ by 
\[
 \left[\begin{matrix} v_1 \\ \vdots \\ v_n \end{matrix} \right]  =  \gamma j
\qquad
\mbox{and} 
\qquad
 \left[\begin{matrix} w_1 \\ \vdots \\ w_n \end{matrix}\right]    =  {}^t \gamma^{-1}  e_n ,
\]
so that $v\cdot {}^t w = X-i_{R_y}(\pi)$. 
The essential point is that the symmetry of $X$ implies the symmetry of $v\cdot {}^tw$, 
so that $v_i w_j = v_j w_i$ for all $i,j$.  
At least one component of $w$, say $w_\ell$, is a unit, and so $(v_\ell  w_\ell^{-1} )w_j = v_j$.
This means that $v$ is a scalar multiple of $w$:  $v=c w$ where $c= v_\ell  w_\ell^{-1} \in R_y$.
From the definitions of $v$ and $w$ we deduce
\[
{}^t\gamma \gamma  \cdot   j = c  \cdot   e_n,  
\]
and hence $(j_1,\ldots,j_n)=(c)$ is a principal ideal.

Having now proved that $R_y\to R_z$ is surjective with principal kernel, it only 
remains to prove that the map is not an isomorphism.  Suppose it is an isomorphism.
This implies, by dimension considerations,  that $\mathcal{Z}(m)$
contains an entire  irreducible component of $\mathcal{M}$.  As $\mathcal{M}$
is flat over $\co_\kk$, it follows that $\mathcal{Z}(m)_{/\C}$
contains an irreducible component of $\mathcal{M}_{/\C}$.   But $\mathcal{Z}(m)_{/\C}$
is a divisor on $\mathcal{M}_{/\C}$, as one can check using the explicit complex
uniformization of \cite{KRunitaryII} or \cite{howardCM}.
\end{proof}


\subsection{Analytic compactification}
\label{ss:complex uniform}


Recall that we have fixed an embedding $\iota:\kk\to \C$.
We now construct explicit coordinates on  the complex orbifold 
$\mathcal{M}(\C) \iso \mathcal{M}^\naive(\C)$, and give a 
purely analytic construction of the  compactification $\mathcal{M}^*(\C)$.
The complex uniformization of $\mathcal{M}(\C)$ is described both in 
\cite{KRunitaryII} and in \cite{howardCM}, and we only sketch the main ideas.
The compactification is a special case of the vastly more general constructions
of \cite{AMRT,kai-wen_2}, and generalizes the signature $(2,1)$ case
studied in   \cite{cogdell81,cogdell85,miller09}.

The orbifold $\mathcal{M}(\C) \iso \mathcal{M}_{(1,0)} (\C)\times \mathcal{M}_{(n-1,1)}(\C)$ 
is disconnected,  and its connected components are indexed by isomorphism classes of pairs
$(\mathfrak{A}_0,\mathfrak{A})$ in which  
\begin{itemize}
\item
$\mathfrak{A}_0$ is a projective $\co_\kk$-module of 
rank $1$, equipped with a positive definite Hermitian form $h_{\mathfrak{A}_0}(x,y)$,
under which $\mathfrak{A}_0$ is self dual,
\item
$\mathfrak{A}$ is a projective $\co_\kk$-module of rank $n$, equipped with a
Hermitian form $h_{\mathfrak{A}}(x,y)$ of signature $(n-1,1)$, under which $\mathfrak{A}$ is self dual.
\end{itemize}
The $\mathfrak{A}_0$'s index the (finitely many) points of $\mathcal{M}_{(1,0)} (\C)$, 
while the $\mathfrak{A}$'s index the connected components of $\mathcal{M}_{(n-1,1)}(\C)$.
Fix one such pair $(\mathfrak{A}_0 , \mathfrak{A})$, and set
\[
L =\Hom_{\co_\kk}(\mathfrak{A}_0,\mathfrak{A})
\]
and $L_\kk=L\otimes_{\co_\kk} \kk$. There is a Hermitian form $\langle f,g\rangle$
on $L$ of signature $(n-1,1)$ characterized by 
\[
h_{\mathfrak{A}_0}(x,x) \cdot \langle f, g \rangle = h_{\mathfrak{A}}( f(x), g(x) )
\]
for all $x\in \mathfrak{A}_0$.    The discrete group 
\[
\tilde{\Gamma} = \Aut(\mathfrak{A}_0,h_{\mathfrak{A}_0}) \times \Aut(\mathfrak{A}, h_\mathfrak{A})
\]
acts  on $L$ by  $(\gamma_0,\gamma) \action f = \gamma\circ f\circ \gamma_0^{-1},$
and sits in an exact sequence
\[
1 \to \mu(\kk) \to \tilde{\Gamma} \to \Gamma \to 1
\]
where $\Gamma = \Aut(L,\langle \cdot,\cdot\rangle)$, and $\mu(\kk)\subset \tilde{\Gamma}$
is embedded diagonally.
Let $\mathcal{D}$ be the space of negative lines in  
\[
V=L_\kk\otimes_{\kk} \C.
\]
Such  lines are denoted with the symbol $\hh$.   The group $\Gamma$, and hence also $\tilde{\Gamma}$, acts on $\mathcal{D}$, and 
there is a morphism of complex orbifolds
\[
\tilde{ \Gamma} \backslash \mathcal{D} \to \mathcal{M}(\C)
\]
identifying $ \tilde{\Gamma} \backslash \mathcal{D}$ with the 
connected component of  $\mathcal{M}(\C)$ indexed by $(\mathfrak{A}_0,\mathfrak{A})$.

As in Section \ref{ss:boundary labels}, the boundary components of $\mathcal{M}^*(\C)$
are indexed by isomorphisms classes of triples $(\mathfrak{A}_0,\mathfrak{m} \subset \mathfrak{A}  )$,
where $\mathfrak{A}_0$ and $\mathfrak{A}$ are as above, and  
$\mathfrak{m}\subset \mathfrak{A}$ is an isotropic direct summand of rank one.  
We now give an analytic construction of the boundary component 
indexed by $(\mathfrak{A}_0,  \mathfrak{m}  \subset  \mathfrak{A})$.    
The $\co_\kk$-submodule
\[
\mathfrak{a} = \{ f\in L : f(\mathfrak{A}_0) \subset \mathfrak{m} \}
\]
is an isotropic rank one direct summand of $L$. 
There is a canonical filtration  $\mathfrak{a} \subset \mathfrak{a}^\perp \subset L$, and by 
Proposition \ref{Prop:normal decomp}  we may fix a decomposition
\[
L=\mathfrak{a}  \oplus \Lambda \oplus \mathfrak{c}
\]
in such a way that $\mathfrak{c}$ is isotropic, $\mathfrak{a}^\perp = \mathfrak{a} \oplus \Lambda$,
and $\Lambda = ( \mathfrak{a}  \oplus \mathfrak{c} )^\perp$ is positive definite and self dual.
Let $\mathfrak{a}_\kk$, $\Lambda_\kk$, and $\mathfrak{c}_\kk$ be the $\kk$-spans of $\mathfrak{a}$, 
$\Lambda$, and $\mathfrak{c}$ in $L_\kk$.
We may choose basis elements $e,e_1\ldots, e_{n-2}, e'\in L_\kk$ in such a way that 
\begin{align*}
\mathfrak{a}_\kk   & = \kk e \\
\Lambda_\kk  & = \kk e_1 + \cdots +\kk e_{n-2}  \\
\mathfrak{c}_\kk   & = \kk e'  ,
\end{align*}
and so that  the Hermitian form on $L_\kk$ is  given by
\[
\langle f,g \rangle = {}^t f \cdot \left( \begin{smallmatrix} 
& &  \delta_\kk \\
&  A & \\
- \delta_\kk
\end{smallmatrix}\right)  \cdot  \overline{g}
\]
for a diagonal matrix $A\in M_{n-2}(\Q)$ with positive diagonal entries. 
There are fractional $\co_\kk$-ideals $\mathfrak{a}_0$ and $\mathfrak{c}_0$ defined by 
$\mathfrak{a}  = \mathfrak{a}_0 e$ and  $\mathfrak{c} =\mathfrak{c}_0e'$,
and related by $\delta_\kk \overline{\mathfrak{c}}_0  \mathfrak{a}_0=\co_\kk$.

Define a bijection
\begin{equation}\label{cusp coordinates}
\mathcal{D} \iso \{  (z ,u)\in \C\times \C^{n-2} : i\sqrt{d_\kk} (z - \overline{z})  +  {}^t u A \overline{u} < 0    \}
\end{equation}
 by associating to $(z,u)$  the span of
\[
\left[ \begin{smallmatrix} z \\ u \\ 1 \end{smallmatrix} \right] \in \C^n \iso V,
\]
and define a positive real analytic function of the variable  $\hh\in \mathcal{D}$ by
\[
\xi(\hh) = - d_\kk \frac{ \langle v,v\rangle  }{ | \langle v,e\rangle|^2 } 
= - i\sqrt{d_\kk} (z - \overline{z})  -  {}^t u A \overline{u} .
\]
In the middle expression $v$ is any nonzero vector on the line $\hh$.
For every $\epsilon>0$,  define 
\[
\mathcal{D}^\epsilon =
  \left\{ \hh\in \mathcal{D} :  \frac{ 1}{  \xi(\hh) } <  \epsilon    \right\}.
\]
Sets of this form should be thought of as tubular neighborhoods of our boundary component,
and the function $1/\xi$ is to be thought of as ``distance to the boundary".

To the isotropic line $\kk e=\mathfrak{a}_\kk$ there are associated subgroups
\[
C_\Gamma\subset N_\Gamma\subset P_\Gamma\subset \Gamma
\]
defined as follows.  Let $P\subset  \Aut(L_\kk ,\langle\cdot,\cdot\rangle)$ be
the subgroup of automorphisms preserving the isotropic line $\mathfrak{a}_\kk$.
As $\mathfrak{a}_\kk^\perp = \mathfrak{a}_\kk \oplus \Lambda_\kk$, elements of $P$ 
necessarily preserve the filtration
\[
\mathfrak{a}_\kk \subset \mathfrak{a}_\kk \oplus \Lambda_\kk \subset L_\kk.
\] 
The unipotent radical $N \subset P$ is  the subgroup of elements acting trivially on the 
graded pieces of the filtration.  In our coordinates
\[
N = \left\{ 
\left( \begin{matrix}   1 & {}^tT  & X \\ & I_{n-2} & S \\ & & 1   \end{matrix} \right)
:  \begin{array}{cc}
S,T\in \kk^{n-2},  & \delta_\kk T=   -  A\overline{S}, \\
X\in \kk,  & \delta_\kk (X-\overline{X}) + {}^tS A\overline{S} =0
\end{array}
\right\}.
\]
The center  $C \subset N$ consists of those matrices for which $S=T=0$.
Abbreviate
\[
P_\Gamma=P\cap \Gamma \qquad N_\Gamma=N\cap \Gamma \qquad C_\Gamma=C \cap\Gamma.
\]
There is a unique $r\in \Q^+$  such that 
\[
C_\Gamma = \left\{ \left( \begin{matrix}   1 & 0  & X \\ & I_{n-2} & 0 \\ & & 1   \end{matrix} \right)   : X \in r\Z \right\}.
\]
The value of $r$ depends on the choice of $e\in \mathfrak{a}_\kk$, and satisfies
$r\Z = \delta_\kk\mathfrak{a}_0\overline{\mathfrak{a}}_0 \cap \Q$.   Using this (and the 
hypothesis that $d_\kk$ is odd) it is easy to see that
\[
r=   d_\kk \mathrm{N}(\mathfrak{a}_0).
\]

The function $\xi(\hh)$ is invariant under the action of $P_\Gamma$ on $\mathcal{D}$,
and hence $\mathcal{D}^\epsilon$ is stable under $P_\Gamma$.
The action of $N_\Gamma$ on $\mathcal{D}$ has the explicit form
\[
 \left( \begin{matrix}   1 & {}^tT  & X \\ & I_{n-2} & S \\ & & 1   \end{matrix} \right)   \action (z ,u) 
 = (  z +{}^tTu + X, u +S    ),
\]
and if we set $q=e^{2\pi i z/r}$, then $(z,u)\mapsto (q,u)$ defines an isomorphism
\[
C_\Gamma\backslash \mathcal{D}^\epsilon \iso
 \{  (q,u) \in \C \times \C^{n-2} :  0<  |q|  < e^{-\rho(\epsilon,u)}    \}
\]
where
\[
\rho(\epsilon,u) = \frac{\pi}{r \sqrt{d_\kk} } \left(\frac{1}{  \epsilon} +{}^t u A \overline{u} \right).
\]
Our coordinates exhibit the quotient $C_\Gamma \backslash \mathcal{D}^\epsilon$
as a punctured disk bundle over  $\C^{n-2}$, and as such there is a natural partial
compactification
\begin{equation}\label{boundary coords}
\overline{C_\Gamma\backslash \mathcal{D}^\epsilon} 
= \{  (q,u) \in \C \times \C^{n-2} :    |q|  < e^{-\rho(\epsilon,u)}    \}.
\end{equation}
In the coordinates (\ref{boundary coords}),
\begin{equation}\label{tau coords}
\xi(\hh) =  - \frac{ d_\kk^{3/2} \mathrm{N}(\mathfrak{a}_0)  }{  2 \pi } \cdot   \log|q|^2   -  {}^t u A \overline{u} .
\end{equation}

For   sufficiently small $\epsilon$ the map
$P_\Gamma \backslash \mathcal{D}^\epsilon  \to \Gamma \backslash \mathcal{D}$
is an open immersion of orbifolds.    The  action of $P_\Gamma$ on 
$C_\Gamma \backslash \mathcal{D}^\epsilon$  extends to an action on (\ref{boundary coords})
leaving the boundary divisor $q=0$ invariant, and if we set 
\[
\overline{P_\Gamma\backslash \mathcal{D}^\epsilon} = 
 (  P_\Gamma / C_\Gamma ) \backslash ( \overline{C_\Gamma\backslash \mathcal{D}^\epsilon})
\]
 there is an open immersion of orbifolds
$ P_\Gamma \backslash \mathcal{D}^\epsilon \to \overline{P_\Gamma\backslash \mathcal{D}^\epsilon}$.
This allows us to  glue $\overline{P_\Gamma\backslash \mathcal{D}^\epsilon}$ onto 
$\Gamma\backslash \mathcal{D}$
to create a partial compactification of $\Gamma\backslash \mathcal{D}$.  Taking the 
orbifold quotient by $\mu(\kk)$ acting trivially then gives a partial compactification 
of $\tilde{\Gamma}\backslash \mathcal{D} \iso \mu(\kk) \backslash (\Gamma\backslash \mathcal{D})$,
obtained by glueing  
\[
X^\epsilon = \mu(\kk) \backslash \left(  \overline{P_\Gamma\backslash \mathcal{D}^\epsilon} \right).
\]

The results of \cite{kai-wen_2}, comparing analytic and algebraic compactifications, 
show that $X^\epsilon$ is a tubular neighborhood of the algebraically defined boundary component 
indexed by $(\mathfrak{A}_0, \mathfrak{m}\subset \mathfrak{A})$.  The boundary component itself
is defined by the equation $q=0$ on $X^\epsilon$.


\subsection{Green functions}


In \cite{howardCM} one can find the construction, for every $m\not=0$,  of a Green function
$\Gr(m,v,\cdot)$ for the divisor $\mathcal{Z}(m)$ on the open Shimura variety $\mathcal{M}$, but 
with no claims about how it behaves near the boundary.
This Green function depends on a choice of auxiliary parameter $v\in \R^+$, and its construction is 
based on ideas of Kudla \cite{kudla97}.  In this subsection we begin the task of
analyzing the behavior of this Green function near the boundary of the analytic compactification
$\mathcal{M}^*(\C)$. 

We continue with our fixed  triple $(\mathfrak{A}_0, \mathfrak{m} \subset \mathfrak{A})$ indexing 
a boundary component of $\mathcal{M}^*(\C)$.  The pair $(\mathfrak{A}_0,\mathfrak{A})$ 
indexes a connected component $\tilde{\Gamma}\backslash \mathcal{D} \subset \mathcal{M}(\C)$.
We also keep the basis  $e,e_1\ldots, e_{n-2}, e'\in L_\kk$ of the previous subsection.
Given a nonisotropic  $f\in L_\kk$,  write 
\begin{equation}\label{f coords}
f= ae + b_1e_1+\cdots + b_{n-2} e_{n-2} + ce',
\end{equation}
and define a  holomorphic function on $\mathcal{D}$, in the coordinates (\ref{cusp coordinates}), by
\[
\Psi_f(\hh) = \big\langle \left[ \begin{matrix}
z \\ u\\ 1
\end{matrix} \right], f \big\rangle =  \delta_\kk \overline{c} z  -\delta_\kk \overline{a}   + {}^t\overline{b}  A u
\]
(from this point on we will no longer  distinguish between elements of $\kk$ and their images
under the fixed embedding $\iota:\kk\to \C$).  Define an analytic divisor
 $\mathcal{D}(f) \subset \mathcal{D}$ by
\begin{align*}
\mathcal{D}(f) &= \{ \hh \in \mathcal{D} :  \hh \perp f\} \\
&= \{\hh\in \mathcal{D} : \Psi_f(\hh) =0 \}.
\end{align*}

For a positive real number $x$, the function 
\begin{equation}\label{beta}
\beta_1(x) = \int_1^\infty e^{-xu}  \, \frac{du}{u}
\end{equation}
decays exponentially as $x\to\infty$.  There is a power series expansion
\[
\beta_1(x)+\log(x) =-\gamma + \sum_{k=1}^\infty \frac{ (-1)^{k+1}x^k}{k\cdot k!}
\]
($\gamma=  0.577216\ldots$ is Euler's constant), and so the 
left hand side extends to a smooth function on $\R$.
Given a parameter $v\in \R^+$ and an integer $m\not=0$, the Green function
of \cite{howardCM}, restricted to the component $\tilde{\Gamma}\backslash \mathcal{D}$, 
is given by the formula
\begin{equation}\label{green def}
\Gr(m,v,\hh) =   \sum_{  \substack{  f\in L \\  \langle f,f\rangle =m  }  }  
\beta_1 \left(  \frac{  4\pi v  | \Psi_f(\hh) |^2  }{ \xi(\hh) }   \right).
\end{equation}
If $\psi_m(\hh)=0$ is a local equation (on some open subset $U\subset \mathcal{D}$)
for the analytic divisor
\[
\mathcal{D}(m) = \sum_{ \substack{ f\in L  \\  \langle f,f\rangle =m }} \mathcal{D}(f)
\]
then  $\Gr(m,v,\hh) + \log|\psi_m(\hh)|^2$
extends to a smooth  function on $U$.  The orbifold divisor 
$\tilde{\Gamma} \backslash \mathcal{D}(m)$ is none other than
the restriction of $\mathcal{Z}(m)(\C)$
to the component $\tilde{\Gamma}\backslash \mathcal{D} \subset \mathcal{M}(\C)$.
 If $m<0$  then $\mathcal{D}(m)$ is empty, and $\Gr(m,v,\hh)$  is simply a smooth  function on  
 $\tilde{\Gamma}\backslash\mathcal{D}$.

The problem is to understand the behavior of $\Gr(m,v,\hh)$ near the boundary divisor 
$q=0$ of $X^\epsilon$, or equivalently near the divisor $q=0$ of  (\ref{boundary coords}).
Define $P_\Gamma$-stable subsets of $L$ by
\begin{align}\label{L decomp}
L^\bndry(m) &= \{ f\in L: \langle f,f \rangle=m \mbox{ and }  \langle f,  \mathfrak{a}  \rangle =0 \} \\
L^\interior(m) &= \{ f\in L: \langle f,f \rangle=m \mbox{ and }\langle f ,  \mathfrak{a}   \rangle\not=0 \}. \nonumber
\end{align}
In the coordinates (\ref{f coords}) we have  $\langle f, e \rangle=-\delta_\kk c$, and so
$\langle f,\mathfrak{a}\rangle=0$ if and only if $c=0$.
To study  $\Gr(m,v,\hh)$ near the boundary,  we break (\ref{green def}) into two sums: 
\[
\Gr(m,v,\hh)=
\Gr^\bndry(m,v,\hh)  + \Gr^\interior(m,v,\hh) 
\]
where
\begin{align*}
\Gr^\bndry(m,v,\hh) &= \sum_{  f\in L^\bndry(m) } 
  \beta_1\left(  \frac{  4\pi v | \Psi_f(\hh) |^2  }{ \xi(\hh) }  \right) \\
\Gr^\interior(m,v,\hh) &= \sum_{    f\in L^\interior(m) } 
  \beta_1\left(  \frac{  4\pi v | \Psi_f(\hh) |^2  }{ \xi(\hh) }  \right).
\end{align*}

Consider the image of  $L$ under 
$f\mapsto \langle f,e \rangle=-\delta_\kk c$.  The image is a $\Z$-lattice in $\C$, 
and so there is some positive  
real number $\mathbf{c}_\mathrm{min}$ satisfying $|c|^2> 8 \mathbf{c}_\mathrm{min} $ 
for every $f\in L^\interior(m)$. For every $\hh \in \mathcal{D}$ let $\hh^\perp \subset V$ be the 
orthogonal complement of $\hh$ under $\langle \cdot,\cdot\rangle$.   There is unique 
positive definite quadratic form $Q_\hh$ on the 
real vector space underlying $V$ satisfying the conditions
\begin{itemize}
\item $\hh$ and $\hh^\perp$ are orthogonal under $Q_\hh$,
\item the restriction of $Q_\hh$ to $\hh$ is $Q_\hh(f) = -\langle f,f\rangle$,
\item the restriction of $Q_\hh$ to $\hh^\perp$ is $Q_\hh(f) = \langle f,f\rangle$.
\end{itemize}
With a bit of algebra one can show that $Q_\hh$ is given by the explicit formula
\[
Q_\hh(f) = \langle f,f \rangle + 2\cdot  \frac{ | \Psi_f(\hh) |^2  }{ \xi(\hh) },
\]
and by the even more explicit formula
\begin{align}\nonumber
Q_\hh(f)    & = 
 \frac{ |c|^2 \xi(\hh) }{2}  +  {}^t(b-cu) A ( \overline{ b} - \overline{cu} )   \nonumber \\
& \quad 
 + \frac{ 2 }{    \xi(\hh)} 
 \left|   \frac{ \delta_\kk c (z+\overline{z})}{2} - \delta_\kk a + {}^t\overline{u} A b - \frac{ c }{2} \cdot  {}^tuA\overline{u}    \right|^2  
  \label{Q bound}
\end{align}
in the coordinates (\ref{cusp coordinates}).
Ignoring all but the first term on the right shows that $Q_\hh(f)   >   |c|^2 \xi(\hh) /2 $.
It follows that for  sufficiently small  $\epsilon$ (depending on $m$), 
\begin{equation}\label{simple noncusp bound}
   \frac{ | \Psi_f(\hh) |^2  }{ \xi(\hh) }
      > \frac{Q_\hh(f)}{4}  > 
 \mathbf{c}_\mathrm{min}\xi(\hh)  > \frac{\mathbf{c}_\mathrm{min}}{\epsilon}
\end{equation}
for all $\hh\in \mathcal{D}^\epsilon$ and all $f\in L^\interior(m)$.

A more geometric motivation for the decomposition $\Gr=\Gr^\bndry+\Gr^\interior$
 is that the sum defining $\Gr^\interior$ is over those $f$'s for which the image of 
 $\mathcal{D}(f)$ in  $\tilde{\Gamma}\backslash \mathcal{D}$ does not intersect the 
 boundary divisor $q=0$ of 
(\ref{boundary coords}), while $\Gr^\bndry$ corresponds to those $f$ for which $\mathcal{D}(f)$
does meet the boundary.
 Indeed, (\ref{simple noncusp bound}) shows that by shrinking $\epsilon$ 
 we may assume that $\Psi_f(\hh)$
is nonvanishing on $\mathcal{D}^\epsilon$ for all $f\in L^\interior(m)$.  This proves that
on  $\mathcal{D}^\epsilon$ we have the equality of divisors
\begin{equation}\label{boundary KR}
\mathcal{Z}(m) (\C) = \sum_{ \substack{ f\in L \\ \langle f,f\rangle =m} } \mathcal{D}(f) 
= \sum_{ f \in L^\bndry(m)} \mathcal{D}(f).
\end{equation}
The motivates the notation ``$\interior$" and ``$\bndry$", which are short for ``interior"
and ``boundary".


\subsection{Behavior near the boundary, part I}


In this subsection we study the behavior of $\Gr^\interior(m,v,\hh)$ near the boundary.
This is relatively easy. For $f\in L^\interior(m)$,  the estimates (\ref{simple noncusp bound}) tells us that  
$ | \Psi_f(\hh) |^2 / \xi(\hh) $ 
grows without bound as $\xi(\hh) \to \infty$.
As $\beta_1$ decays exponentially at $\infty$,  the sum defining $\Gr^\interior(m,v,\hh)$
 converges to $0$ term-by-term as $\xi(\hh) \to \infty$, or, equivalently by (\ref{tau coords}),
 as $q\to 0$. Thus one does not expect 
$\Gr^\interior(m,v,\hh)$ to contribute significantly to the behavior of  $\Gr(m,v,\hh)$ near $q=0$.  
The following proposition makes this more precise.

\begin{Prop}\label{Prop:growth I}
The  function $\mathcal{E}^\interior(\hh) = \Gr^\interior(m,v,\hh)$, 
initially defined on $C_\Gamma\backslash \mathcal{D}^\epsilon$,
extends continuously to $\overline{C_\Gamma\backslash \mathcal{D}^\epsilon}$ and vanishes identically
on the boundary divisor $q=0$. The differential forms $\partial \mathcal{E}^\interior$, $\overline{\partial} \mathcal{E}^\interior$, and 
$\partial\overline{\partial}\mathcal{E}^\interior$
have $\log$-$\log$ growth along the divisor $q=0$.
\end{Prop}

The proposition will be  a consequence of  the following lemma.  Abbreviate
\[
R_f(\hh) = \frac{ | \Psi_f(\hh) |^2  }{ \xi(\hh) }  .
\]

\begin{Lem}\label{Lem:noncusp lemma}
Suppose  $\beta$ is any complex valued function on $(0,\infty)$ for which there are 
positive constants  $ C_1$ and $T$ 
satisfying $ 0< | \beta(t) |< e^{- C_1 t}$ for all $t> T$.  There are  $\epsilon, C_2, C_3, C_4>0$ such that
\[
\Big|    \sum_{ f\in L^\interior(m)  } 
  \beta\left(   R_f(\hh) \right)  \Big| <  C_4
|q|^{ C_3}    e^{  C_2   {}^tuA\overline{u}   }
\]
for all $\hh\in \mathcal{D}^\epsilon$.
\end{Lem}

\begin{proof}
By (\ref{simple noncusp bound}), we may shrink  $\epsilon$ in order to  assume that 
$R_f(\hh)>T$ for all $f\in L^\interior(m)$.  Thus
\[
\Big|    \sum_{ f\in L^\interior(m) } 
  \beta\left(   R_f(\hh) \right)  \Big|
  <   \sum_{   f\in L^\interior(m)   }   e^{-   C_1 Q_\hh(f) }
  <   \sum_{ a \in \mathfrak{a}_0 }  \sum_{   f_\Lambda \in   \Lambda  }   \sum_{ \substack{ c\in \mathfrak{c}_0 \\ c\not=0 } }
   e^{-   C_1 Q_\hh(  a e+ f_\Lambda + c e'   ) }.
\]
 Some elementary but slightly tedious estimates using  (\ref{Q bound}) show that the 
 final term on the right is bounded by $  C'_4 e^{- C_2 \xi(\hh)}$
 for some $ C'_4, C_2>0$.  The claim now follows from (\ref{tau coords}).
\end{proof}

\begin{proof}[Proof of Proposition \ref{Prop:growth I}]
The first claim is immediate from Lemma \ref{Lem:noncusp lemma}, by taking $\beta(t)=\beta_1(4\pi v t)$.
Next we bound the growth of $\partial \mathcal{E}^\interior$.  For any $f\in L^\interior(m)$ we compute, 
using 
\begin{equation}\label{derivative}
\frac{d}{dt} \beta_1 ( t )= - \frac{ 1 }{ t } e^{-t},
\end{equation}
the first derivative
\begin{align*}
\frac{ \partial }{\partial z} \beta_1(4\pi v R_f(\hh))
&=
e^{-4\pi v R_f(\hh)} \left[
\frac{\overline{\Psi}_f }{R_f } \frac{\partial \Psi_f}{\partial z} 
+ \frac{1}{\xi} \frac{\partial \xi}{\partial z}
\right] \\
&=
\delta_\kk  e^{-4\pi v R_f(\hh)} \left[
\frac{  \overline{c} \overline{\Psi}_f  }{R_f }
- \frac{ 1 }{\xi}. 
\right]
\end{align*}
Combining (\ref{Q bound}) with $|\Psi_f(\hh)|^2 = R_f(\hh) \xi(\hh)$ 
shows that on $\mathcal{D}^\epsilon$,
\[
\sqrt{d_\kk}\cdot \left|
\frac{  \overline{c} \overline{\Psi}_f  }{R_f  }
- \frac{1}{\xi} \right|  <C
\]
for a constant $C$ independent of $\hh$ and $f\in L^\interior(m)$.  Therefore
\[
\left|
\frac{ \partial }{\partial z} \beta_1(4\pi v R_f(\hh))
\right|
<  C e^{-4\pi v R_f(\hh)}.
\]
 It now follows from Lemma \ref{Lem:noncusp lemma} that
$
\left| \partial \mathcal{E}^\interior / \partial z  \right|  <  C_4 |q|^{ C_3} e^{ C_2 {}^tuA\overline{u}}
$
for some $ C_2, C_3, C_4>0$, and so
\[
\frac{\partial \mathcal{E}^\interior}{\partial q} = \frac{r}{2\pi i q} \frac{\partial \mathcal{E}^\interior}{\partial z} 
= \frac{1}{q \log|q|} \cdot F(\hh)
\]
for some continuous function $F$  vanishing  along $q=0$.  Similarly,  if we write $u={}^t [ u_1,\ldots, u_{n-2}]$ then
\[
\left|
\frac{ \partial }{\partial u_i} \beta_1(4\pi v R_f(\hh))
\right|
<  C e^{-4\pi v R_f(\hh)}
\]
for a constant $C>0$ independent of $f$, and Lemma \ref{Lem:noncusp lemma} implies
 that $\frac{\partial \mathcal{E}^\interior}{\partial u_i}$ vanishes along $q=0$.  This proves that
$\partial \mathcal{E}^\interior$ has $\log$-$\log$ growth along $q=0$.  
The proofs for $\overline{\partial} \mathcal{E}^\interior$
and $\partial\overline{\partial} \mathcal{E}^\interior$ are similar.
\end{proof}


\subsection{Behavior near the boundary, part II}


Now we turn to the more difficult analysis of $\Gr^\bndry(m,v,\hh)$.
 Suppose $f \in L^\bndry(m)$.  
 If we write  $f$ in the coordinates (\ref{f coords}), and recall that
 $\langle f, \mathfrak{a} \rangle=0$ implies  $c=0$, we see that
 \[
 \Psi_f  (\hh)  = -\delta_\kk \overline{a}  + {}^t\overline{b}  A u.
 \]
 In particular, for any $a e \in \mathfrak{a}_0 e=\mathfrak{a}$ we have
 \[
 \Psi_{ ae +f} = \Psi_f -\delta_\kk\overline{a}.
 \]

 The function $\Psi_f$ is invariant  under the action of 
$C_\Gamma \subset \mathrm{Stab}_{N_\Gamma}(f)$ on $\mathcal{D}$, and is visibly 
independent of the coordinate $z$.  Thus $\Psi_f$   defines a function on 
$C_\Gamma\backslash \mathcal{D}^\epsilon$
independent of the variable $q$ in (\ref{boundary coords}).  In other 
words, if we view $C_\Gamma\backslash \mathcal{D}^\epsilon$ as a punctured disk
bundle over $\C^{n-2}$, then $\Psi_f$ is constant on fibers.  It follows that $\Psi_f$
 extends uniquely to a holomorphic function on 
 the partial compactification $\overline{C_\Gamma\backslash \mathcal{D}^\epsilon}$.
By  (\ref{boundary KR})  the pullback of  $\mathcal{Z}^*(m)(\C)$ to 
$\overline{C_\Gamma\backslash \mathcal{D}^\epsilon}$ is 
\[
\mathcal{Z}^*(m)(\C) =    \sum_{f\in  C_\Gamma\backslash L^\bndry(m)} 
\mathcal{D}^*(f)
\] 
where $\mathcal{D}^*(f) \subset \overline{C_\Gamma\backslash \mathcal{D}^\epsilon}$
is the zero locus of $\Psi_f$.

  \begin{Prop}\label{Prop:main growth}
Suppose $\psi_m(\hh)=0$  is an equation for the divisor $\mathcal{Z}^*(m)(\C)$
on some open subset of  $\overline{C_\Gamma\backslash \mathcal{D}^\epsilon}$,
and set
\[
\mathrm{Ind}(m) = \# \{ f\in \mathfrak{a}^\perp/\mathfrak{a} : \langle f,f\rangle =m\}.
\]
The smooth function
\[
\mathcal{E}^\bndry(\hh) =    \log|\psi_m(\hh)|^2  
 -     \frac{ \mathrm{Ind}(m)  \xi(\hh)   }{  4v \mathrm{Vol}(\C/\delta_\kk\overline{\mathfrak{a}}_0) }  
+ \Gr^\bndry(m,v,\hh)  
\]
on $C_\Gamma\backslash \mathcal{D}^\epsilon$ is bounded,  and the differential forms   
$\partial \mathcal{E}^\bndry$,  $\overline{\partial}\mathcal{E}^\bndry$, and 
$\partial \overline{\partial}\mathcal{E}^\bndry$   have $\log$-$\log$ growth along the divisor $q=0$.
\end{Prop}

 The proof of Proposition \ref{Prop:main growth} will occupy the remainder of this subsection.
We  begin by noting that
\begin{align*}
 \Gr^\bndry(m,v,\hh) 
& =  \sum_{    f \in   L^\bndry(m)  }   
 \beta_1\left(    \frac{  4\pi v |\Psi_{ f}(\hh)  |^2 }{\xi(\hh)} \right) \\
  & =\sum_{    \substack { f \in  \Lambda \\ \langle f,f\rangle =m   } } 
  \sum_{ ae  \in  \mathfrak{a} }
 \beta_1\left(    \frac{  4\pi v |\Psi_{ ae+ f}(\hh)  |^2 }{\xi(\hh)} \right) \\
  & =\sum_{    \substack { f \in  \mathfrak{a}^\perp/\mathfrak{a} \\ \langle f,f\rangle =m   } } 
  \sum_{ \eta  \in  \delta_\kk \overline{\mathfrak{a}_0 } }
 \beta_1\left(    \frac{  4\pi v |\Psi_{ f}(\hh) + \eta |^2 }{\xi(\hh)} \right).
\end{align*}
On $\overline{ C_\Gamma\backslash \mathcal{D}^\epsilon }$, the divisor
$\mathcal{Z}^*(m)(\C)$ is the sum over  
$\{ f\in \mathfrak{a}^\perp/ \mathfrak{a} : \langle f,f\rangle = m \}$ of the divisors 
 \begin{equation}\label{partial divisor}
 \sum_{\eta\in \delta_\kk\overline{\mathfrak{a}}_0} \mathrm{Div}( \Psi_f (\hh) +\eta).
\end{equation}
Fix a point $\hh_0$ on the boundary $q=0$ of 
$\overline{C_\Gamma\backslash \mathcal{D}^\epsilon}$.  If  $\psi_f(\hh)=0$
is a local equation for (\ref{partial divisor}) on some open neighborhood of $\hh_0$,
we consider, for each $f\in \mathfrak{a}^\perp/\mathfrak{a}$,  the function
\begin{equation}\label{partial error}
\mathcal{E}_f(\hh) = \log|\psi_f(\hh)|^2  
  -   \frac{ \xi(\hh)}{ 4v  \mathrm{Vol}(\C/\delta_\kk\overline{\mathfrak{a}}_0)}   +
\sum_{ \eta\in \delta_\kk\overline{\mathfrak{a}}_0 }
  \beta_1\left(    \frac{ 4\pi v  | \Psi_f ( \hh ) + \eta   |^2 }{ \xi ( \hh ) } \right).
\end{equation}
This function is defined so that
\begin{equation}\label{f main growth}
\mathcal{E}^\bndry(\hh) = 
\sum_{  \substack{  f \in  \mathfrak{a}^\perp/\mathfrak{a} \\ \langle f,f \rangle =m  } } \mathcal{E}_f(\hh).
\end{equation}
We will prove that $\mathcal{E}_f$   is bounded on a neighborhood of $\hh_0$, 
and that the differential forms  $\partial \mathcal{E}_f$,  
$\overline{\partial}\mathcal{E}_f$, and 
$\partial \overline{\partial}\mathcal{E}_f$   have $\log$-$\log$ growth along the divisor $q=0$.
Proposition  \ref{Prop:main growth} will then follow easily.

Fix complex numbers $\varpi_1$ and $\varpi_2$ such that  
$\delta_\kk\overline{\mathfrak{a}}_0=\Z\varpi_1+\Z\varpi_2$.  As $\varpi_1$ and $\varpi_2$ form 
a basis for $\C$ as a real vector space, we   may define
real valued functions $\mu$ and $\nu$ on $ \overline{ C_\Gamma \backslash \mathcal{D}^\epsilon}$ 
by the relation
\[
\Psi_f(\hh) = \mu(\hh)  \varpi_1 + \nu(\hh) \varpi_2.
\]
The function $\Psi_f$ depends on a lift of $f\in \mathfrak{a}^\perp/\mathfrak{a}$ to $\mathfrak{a}^\perp$.
As the lift varies, $\Psi_f$ is replaced by  $\Psi_f + \eta$ for $\eta\in \delta_\kk\overline{\mathfrak{a}}_0$.
Thus we may choose the lift $f$ to assume that $| \mu(\hh_0) |  \le 1/2$ and 
$|\nu(\hh_0)|  \le 1/2$.  Define a neighborhood of $\hh_0$ by 
\[
\Omega  =    \big\{\hh \in \overline{C_\Gamma\backslash \mathcal{D}^\epsilon} : 
|\mu(\hh)| <3/4,\, |\nu(\hh)| < 3/4,\, \xi_v(\hh) >1\big\},
\]
where  
\[
\xi_v(\hh)     =   \frac{   \xi(\hh)  } { 4\pi v }  .
\]
For every nonzero $\eta\in \delta_\kk\overline{\mathfrak{a}}_0$ the function $\Psi_f(\hh) +\eta$ 
is nonvanishing on  $\Omega$, and so we may take $\psi_f = \Psi_f$ as our  local equation 
for  (\ref{partial divisor}).

Set $Q(x,y)=|x\varpi_1+y\varpi_2|^2$, so that  
\[
Q(\mu(\hh),\nu(\hh))  = |\Psi_f(\hh)|^2 
\]  
and
\[
\mathcal{E}_f(\hh)  =\log \big(Q(\mu(\hh),\nu(\hh))  \big)   
- \frac{\pi \xi_v(\hh) }{  \mathrm{Vol}(\C/\delta_\kk\overline{\mathfrak{a}}_0 )} +
 \sum_{m,n\in\Z} \beta_1 \left(  \frac{ Q(m+  \mu(\hh)  , n+\nu(\hh) )  }{ \xi_v(\hh) } \right) .
\]
The restriction of  $\mathcal{E}_f$ to a function on $\Omega$ (with possible
singularities along $\Psi_f=0$ and  $q=0$) is henceforth viewed as a
function on the domain
\[
\Omega_0= \left\{ (\mu,\nu,\xi_v) \in  \R^3 :   
\begin{array}{c} |\mu| < 3/4 \\ |\nu| < 3/4 \\  \xi_v>1 \end{array}\right\} 
\]
(with possible singularities along the line $\mu=\nu=0$, and possibly blowing up as $\xi_v\to \infty$).

Define real numbers $\mathbf{a}$, $\mathbf{b}$, and $\mathbf{c}$  by 
$Q(x,y) = \mathbf{a}x^2+\mathbf{b} xy + \mathbf{c}y^2$,  and abbreviate
\[
\Delta=\sqrt{4\mathbf{a}\mathbf{c}-\mathbf{b}^2}= 2\cdot \mathrm{Vol}(\C/\delta_\kk\overline{\mathfrak{a}}_0).
\] 
We will estimate the growth of the function
\[
\sum_{ m, n\in\Z }  \beta_1 \left(  \frac{ Q(m+  \mu  , n+\nu )  }{ \xi_v } \right) 
\]
on $\Omega_0$ by comparing it with the integral
\[
   \int_{\R\times\R}   \beta_1 \left(  \frac{ Q(x  ,y )  }{ \xi_v } \right)   \, dx\, dy 
 =  \frac{ 2 \xi_v } {\Delta}  \int_{\R\times\R}  \beta_1( x^2+y^2 )  \, dx\, dy   
=\frac{  \pi \xi_v } {\mathrm{Vol}(\C/\delta_\kk\overline{\mathfrak{a}}_0) } .
\]
  In order to compare the sum and integral, we decompose
\begin{equation}\label{omega decomp}
\sum_{ m, n\in\Z }   \beta_1 \left(  \frac{ Q(m+  \mu  , n+\nu )  }{ \xi_v } \right)   
=  \frac{  \pi \xi_v } {\mathrm{Vol}(\C/\delta_\kk\overline{\mathfrak{a}}_0) }   +   \sum_{i=1}^4\omega_i(\mu,\nu,\xi_v)
\end{equation}
in which 
\begin{align*}
\omega_1(\mu,\nu,\xi_v)  
&= \sum_m   \beta_1 \left(  \frac{ Q(m+\mu  ,\nu )  }{ \xi_v } \right)  
-   \int_\R  \beta_1 \left(  \frac{ Q(x  ,\nu )  }{ \xi_v } \right)   \,   dx  \\
\omega_2(\mu,\nu,\xi_v) 
&=  \sum_{n\not=0}\left[ \sum_{m}   \beta_1 \left(  \frac{ Q(m+\mu  ,n+\nu )  }{ \xi_v } \right)    
-\int_\R  \beta_1 \left(  \frac{ Q( x  ,n+\nu )  }{ \xi_v } \right)   \,   dx \right]  \\
\omega_3(\mu,\nu,\xi_v) 
&=  \int_{|x|<1} \left[\sum_{n}   \beta_1 \left(  \frac{ Q( x  ,n+\nu )  }{ \xi_v } \right) 
-   \int_\R \beta_1 \left(  \frac{ Q( x  ,y )  }{ \xi_v } \right)  \,   dy \right] \, dx \\
\omega_4(\mu,\nu,\xi_v) 
&=  \int_{|x|>1}  \left[\sum_{n}   \beta_1 \left(  \frac{ Q( x  ,n+\nu )  }{ \xi_v } \right) 
-  \int_\R \beta_1 \left(  \frac{ Q( x  ,y )  }{ \xi_v } \right)  \,   dy \right]  \, dx,
\end{align*}
and study each term individually.

\begin{Lem}\label{Lem:omega24}
The functions $\omega_2$ and $\omega_4$ are bounded on $\Omega_0$.
\end{Lem}

\begin{proof}
For any nonzero $y\in\R$ there is a Fourier expansion
\[
 \sum_{m\in \Z}   \beta_1 \left(  \frac{ Q( m+\mu  ,y )  }{ \xi_v } \right) 
  =\sum_{k\in \Z} g_k(y,\xi_v)  e^{2\pi i k \mu}
\]
in which
\begin{eqnarray*}
g_k(y,\xi_v)  & = & 
\int_{-\infty}^\infty e^{-2\pi i k x} \beta_1 \left(  \frac{ Q( x  ,y )  }{ \xi_v } \right)   \, dx \\
& = &
\int_1^\infty  \left[\int_{-\infty}^\infty e^{-2\pi i k x}  e^{-  \frac{ u Q(x,y)}{ \xi_v } }\, dx \right]\,\frac{du}{u} \\
&=&
\int_1^\infty  \left[   \sqrt{\frac{\pi \xi_v }{\mathbf{a}u} } e^{-\frac{k^2\pi^2 \xi_v }{\mathbf{a}u}} e^{\frac{i\mathbf{b}k\pi y}{\mathbf{a}}}
e^{\frac{-\Delta^2 uy^2}{4\mathbf{a} \xi_v }}          \right]\,\frac{du}{u} \\
&=&
e^{\frac{i\mathbf{b}k\pi y}{\mathbf{a}}}   \sqrt{\frac{\pi }{\mathbf{a}} }    \int_0^{\xi_v}  e^{-\frac{k^2\pi^2 u}{\mathbf{a}}} 
e^{\frac{-\Delta^2 y^2}{4\mathbf{a}u }}          \,\frac{du}{\sqrt{u}} .
\end{eqnarray*}
This allows us to estimate, for $k\not=0$,
\[
|g_k(y,\xi_v)|  < 
\sqrt{\frac{\pi }{\mathbf{a}} }  \int_0 ^\infty     e^{-\frac{k^2\pi^2 u}{\mathbf{a}}} 
e^{\frac{-\Delta^2 y^2}{4\mathbf{a}u}}     \,\frac{du}{\sqrt{u}}  \\
=
\frac{1}{|k|} \cdot e^{-\frac{\Delta \pi  |ky|  }{\mathbf{a}}}.
\]
It follows that 
\begin{eqnarray*} \lefteqn{
\left| \sum_{m\in \Z}   \beta_1 \left(  \frac{ Q( m+\mu  ,y )  }{ \xi_v } \right)  
 - \int_{-\infty}^\infty   \beta_1 \left(  \frac{ Q( x  ,y )  }{ \xi_v } \right)  \, dx
\right|   } \\
& < &     \sum_{k\not=0} |g_k(y,\xi_v)|  <     \sum_{k\not=0}   e^{-\frac{\Delta \pi  |ky|  }{\mathbf{a}}} 
<\frac{C}{y} e^{- \frac{\Delta |y|}{\mathbf{a}}}
\end{eqnarray*}
for some constant $C$ independent of $\mu$ and $\xi_v$.  
Taking $y=n+\nu$ and summing over all nonzero $n$ shows that
$\omega_2$ is bounded on $\Omega_0$.  

Similarly
\[
\left| \sum_{n\in \Z}   \beta_1 \left(  \frac{ Q( x, n+\nu  )  }{ \xi_v } \right) 
- \int_{-\infty}^\infty    \beta_1 \left(  \frac{ Q( x  ,y )  }{ \xi_v } \right)  \, dy
\right|  <\frac{C}{x} e^{- \frac{\Delta |x|}{\mathbf{a}}}
\]
for some constant $C$ independent of $\nu$ and $\xi_v$, and integrating 
over $|x|>1$ shows that $\omega_4$ is bounded.
\end{proof}

Estimating the behavior of $\omega_1$ and $\omega_3$  requires a little more work.

\begin{Lem}\label{Lem:omega13}
The functions $\omega_1(\mu,\nu,\xi_v) + \log(Q(\mu,\nu))$ and $\omega_3(\mu,\nu,\xi_v)$
are bounded on $\Omega_0$.
\end{Lem}

\begin{proof}
Abbreviate
\[
\phi(x,y) =  \beta_1 \left(  \frac{ Q( x  ,y )  }{ \xi_v } \right).
\]
Computing the partial derivative of $\phi(x,y)$ with respect to $x$ we find
\[
\phi_x(x,y) = -\frac{Q_x(x,y) }{ Q(x,y)  }  \cdot  e^{-Q(x,y)/\xi_v},
\]
which shows that
\[
|\phi_x(x,y)| < \frac{|Q_x(x,y)|}{Q(x,y)}  = \frac{1}{|x|} \cdot 
\frac{  |   2\mathbf{a} + yx^{-1}  \mathbf{b} |  }
{ | \varpi_1+yx^{-1}\varpi_2|^2 } < \frac{ C_1}{|x|}
\]
for some constant $ C_1$ independent of $x$, $y$, and $\xi_v$.  Similarly
\[
\phi_{xx} (x,y)|   = 
\left[
\frac{  Q_x(x,y)^2  }{  \xi_v Q(x,y) }  + \frac{    Q_x(x,y)^2 }{Q(x,y)^2} - \frac{ Q_{xx}(x,y)    }{Q(x,y)}  \right]
 e^{-Q(x,y)/ \xi_v }
\]
and so
\begin{align*}
|\phi_{xx} (x,y) |  &  \le 
\frac{  |Q_x(x,y)|^2  }{  \xi_v Q(x,y) }  e^{-Q(x,y) / \xi_v } 
+ \frac{  |  Q_x(x,y)^2 - Q_{xx}(x,y) Q(x,y) |  }{Q(x,y)^2}  \\
& < \frac{ C_2}{\xi_v}  e^{-Q(x,y)/\xi_v}  +   \frac{ C_2}{Q(x,y)}
\end{align*}
for some constant $ C_2$ independent of $x$, $y$, and $\xi_v$.  It
follows easily that
\[
\int_{1/4}^\infty |\phi_{xx} (x,y)|\, dx <  C_3
\]
for some $ C_3$ independent of $y$ and $\xi_v$. 
Let $b_2(y) = y^2-y+1/6$ be the second Bernoulli polynomial.  
If we apply the   Euler-Maclaurin summation formula
\[
\sum_{m=1}^N g(m) = \int_1^N g(x)\, dx + \frac{g(1)+g(N)}{2} + \frac{ g'(N)-g'(1)}{12}
- \int_1^N g''(x) \cdot \frac{ b_2(x-\lfloor x \rfloor ) }{2} \, dx
\]
with  $g(x)=\phi(x+\mu,y)$ and let $N\to \infty$, the above bounds show that 
\begin{equation}\label{euler summation}
\left|
\sum_{m=1}^\infty \phi(m+\mu,y) - \int_1^\infty \phi(x+\mu,y)\, dx - \frac{\phi(1+\mu,y)}{2}
\right|  < C_4
\end{equation}
for some $ C_4$ independent of $\mu$, $y$, and $\xi_v$.

We next  estimate 
\[
\left|   \frac{\phi(\mu,y)}{2}   +  \frac{\phi(1+\mu,y)}{2} -  \int_0^1 \phi  (x+\mu,y)\, dx  \right|
\]
for small values of $y$.   Set 
\[
L(x,y) = - \log \left(\frac{Q(x,y) }{\xi_v}\right) 
\]
 As $\beta_1(s)+\log(s)$ is bounded on compact subsets of $\R^{\ge 0}$, 
 there is some  $ C_5$,  independent of 
 $x$, $y$, $\mu$, and $\xi_v$, such that
\[
|  \phi(x+\mu,y) - L(x+\mu,y) | <  C_5,
\]
provided we restrict to  $x$ and $y$ to, say, the closed interval $[-2,2]$.  
Using elementary calculus we compute the indefinite integral
\[
\int \log (Q(x,y)) \, dx = 
-2  x+ \left(\frac{\mathbf{b} y}{2 \mathbf{a} }+x\right) \log(Q(x,y))  
 + \frac{y \Delta}{\mathbf{a}}  \cdot \arctan\left(  \frac{ \mathbf{b} y +2\mathbf{a}x}{y \Delta }\right),
\]
from which it follows that
\begin{eqnarray*}\lefteqn{
 \int_{0}^1 L(x+\mu,y)  \, dx  - \frac{L(\mu,y) }{2} - \frac{ L (1+\mu,y ) }{2}  }  \\
&=&   
2   + \left(\frac{\mathbf{b} y}{2 \mathbf{a} }+\mu + \frac{1}{2}\right)  \big[ \log(Q(\mu,y))   -  \log(Q(1+\mu,y))   \big]
 \\
& &  + \frac{y \Delta}{\mathbf{a}}   \left[  \arctan\left( \frac{\mathbf{b}}{\Delta}+ \frac{ 2\mathbf{a}\mu}{y \Delta } \right)
 - \arctan\left(  \frac{\mathbf{b}}{\Delta} + \frac{ 2\mathbf{a}(1+\mu)}{y \Delta }\right)  \right].
\end{eqnarray*}
For any $\alpha$ and $\beta$ the function 
$y\cdot \mathrm{arctan}(\alpha+\beta y^{-1})$ is continuous at $y=0$.
From what we have said, there is a function $ C_6= C_6(\mu,y,\xi_v)$ bounded on the domain 
$|\mu|\le 3/4$, $|y|<1$, $1<\xi_v$ and satisfying
\begin{eqnarray*}\lefteqn{
 \int_{0}^1 \phi(x+\mu,y)  \, dx  - \frac{\phi(\mu,y ) }{2} - \frac{ \phi(1+\mu,y ) }{2}  }  \\
&=&   
2   + \left(\frac{\mathbf{b} y}{2 \mathbf{a} }+\mu + \frac{1}{2}\right)  \big[ \log(Q(\mu,y))   -  \log(Q(1+\mu,y))   \big] +  C_6.
\end{eqnarray*}
If we combine this with (\ref{euler summation}) we obtain
\begin{eqnarray*}\lefteqn{
\frac{\phi(\mu,y)}{2}  + \sum_{m=1}^\infty \phi(m+\mu,y )   =   \int_0^\infty \phi(x+\mu,y )\, dx  } \\
& & - \left(\frac{\mathbf{b} y}{2 \mathbf{a} }+\mu + \frac{1}{2}\right)  \big[ \log(Q(\mu,y))   -  \log(Q(\mu+1,y))   \big]
+  C_7
\end{eqnarray*}
for some $ C_7= C_7(\mu , y ,\xi_v)$ bounded on the domain 
$|\mu| < 3/4$, $|y|<1$, and $1<\xi_v$.
Repeating the entire argument with $\mu$ replaced by $-\mu$ and $\varpi_1$ replaced by $-\varpi_1$
shows that
\begin{eqnarray*}\lefteqn{
\frac{\phi(\mu,y)}{2}  + \sum_{m=-\infty}^{-1} \phi(m+\mu,y)   =   \int_{-\infty}^0 \phi(x+\mu,y)\, dx  } \\
& & - \left(\frac{-\mathbf{b} y}{2 \mathbf{a} }-\mu + \frac{1}{2}\right)  \big[ \log(Q(\mu,y))   -  \log(Q(\mu-1,y))   \big]
+  C_7
\end{eqnarray*}
Adding these two estimates together, 
we have proved the existence of a constant $C$, independent of $y$, $\mu$, and $\xi_v$, such that
\begin{equation}\label{second bound}
\left|    \sum_{m=-\infty}^\infty  \phi(m+\mu,y) -  \int_{-\infty}^\infty  \phi(x,y)  \, dx  + \log(Q(\mu,y))
\right|  <C
 \end{equation}
whenever $|y|<1$.   Taking $y=\nu$  shows that $\omega_1(\mu,\nu,\xi_v) + \log(Q(\mu,\nu))$
 is bounded on $\Omega_0$.

Reversing the roles of $x$ and $y$ in the discussion leading to  
(\ref{second bound}) shows that 
\[
 \sum_{n\in\Z}\phi(x ,n+\nu)  - \int_{-\infty}^\infty   \phi(x,y)\, dy + \log(Q(x,\nu)) 
\]
is bounded independently of $\nu$ and $\xi_v$ on the domain  $|x|<1$.
Using the integrability of $\log(x^2)$ near $x=0$, it is easy to check that
\[
 \int_{-1}^1 \log(Q(x,\nu))\, dx   
\]
is bounded as $\nu$ varies over $|\nu|< 3/4$.  It  follows that $\omega_3$ is bounded on $\Omega_0$.
\end{proof}

The boundedness of $\mathcal{E}_f$ on $\Omega_0$, and hence on $\Omega$, 
is now a consequence of (\ref{omega decomp}), 
Lemma \ref{Lem:omega24}, and Lemma \ref{Lem:omega13}.  The desired estimates on the 
growth of $\partial \mathcal{E}_f$, $\overline{\partial} \mathcal{E}_f$, and $\partial\overline{\partial} \mathcal{E}_f$ will require
the following lemma.

\begin{Lem}\label{Lem:derivative estimates}
For any $k>0$, on the domain $\Omega$ 
\[
\sum_{\eta\in \delta_\kk\overline{\mathfrak{a}}_0}  \mathrm{exp}\left( - \frac{  |\Psi_f +\eta|^2  }{ \xi_v } \right)   
=   \frac{\pi \xi_v}{\mathrm{Vol}(\C/\delta_\kk\overline{\mathfrak{a}}_0)} +  O\left( 1/\xi_v^k \right)
\]
\[
\sum_{\eta\in \delta_\kk\overline{\mathfrak{a}}_0}  \mathrm{exp}\left( - \frac{  |\Psi_f+\eta|^2 }{\xi_v} \right)   ( \Psi_f+\eta)  
= O(1/\xi_v^k)
\]
\[
\sum_{\eta\in \delta_\kk\overline{\mathfrak{a}}_0}  \mathrm{exp}\left( - \frac{  |\Psi_f +\eta|^2 }{\xi_v} \right)  
 |\Psi_f +\eta|^2
= \frac{\pi \xi_v^2}{\mathrm{Vol}(\C/\Lambda)}+O( 1/\xi_v^k ) 
\]
and
\[
\sum_{ \substack{\eta\in \delta_\kk\overline{\mathfrak{a}}_0 \\ \eta \not=0}  }    
 \mathrm{exp}\left( - \frac{  |\Psi_f + \eta|^2 }{ \xi_v } \right)
  \cdot \frac{1}{ (  \Psi_f +\eta )  }  =  O(\log(\xi_v))  .
 \]
\end{Lem}

\begin{proof}
Recalling that 
\[
\sum_{\eta\in \delta_\kk\overline{\mathfrak{a}}_0}  \mathrm{exp}\left( - \frac{  |\Psi_f +\eta|^2  }{ \xi_v } \right) 
=
\sum_{m,n\in \Z}    \mathrm{exp}\left( - \frac{  Q(m+\mu,n+\nu)  }{ \xi_v } \right) ,
\]
there is a Fourier expansion 
\[
\sum_{\eta\in \delta_\kk\overline{\mathfrak{a}}_0}  \mathrm{exp}\left( - \frac{  |\Psi_f +\eta|^2  }{ \xi_v } \right) 
=
 \xi_v \sum_{k,\ell \in \Z }  A \big(k \sqrt{\xi_v} ,\ell \sqrt{\xi_v} \big) \cdot  e^{2\pi i  k\mu} e^{2\pi i \ell \nu}
\]
in which 
\[
A(s,t)  =  \int \int    e^{ -   Q(x,y)  }     e^{-2\pi i s x} e^{-2\pi i t y} \, dx\, dy 
\]
is the Fourier transform of $e^{-Q(x,y)}$.  In particular $A(s,t)$ is a Schwartz function, and it follows that
\[
\sum_{\eta\in \delta_\kk\overline{\mathfrak{a}}_0}  \mathrm{exp}\left( - \frac{  |\Psi_f +\eta|^2  }{ \xi_v } \right) 
   =  \xi_v A(0,0)+ O(1/\xi_v^k) 
= \frac{ 2\pi \xi_v}{\Delta} +O( 1/\xi_v^k ). 
\]
This proves the first claim.  The proof of the second  and third claims are similar.

For the fourth claim, set
\begin{align*}
\phi_1 ( \mu , \nu)  &=  \mathrm{exp}\left( - \frac{  Q(\mu,\nu)  }{ \xi_v } \right) \cdot \frac{\mu}{Q(\mu,\nu)} \\
\phi_2 ( \mu , \nu)  &=  \mathrm{exp}\left( - \frac{  Q(\mu,\nu)  }{ \xi_v } \right) \cdot \frac{\nu}{Q(\mu,\nu)}
\end{align*}
so that 
\[
\sum_{ \substack{\eta\in \delta_\kk\overline{\mathfrak{a}}_0 \\ \eta \not=(0,0)}  }    
 \mathrm{exp}\left( - \frac{  |\Psi_f + \eta|^2 }{ \xi_v } \right)  \cdot \frac{1}{ (  \Psi_f +\eta )  }
 = \phi_1(\mu,\nu) \overline{\varpi}_1 + \phi_2(\mu,\nu) \overline{\varpi}_2.
  \]
The Fourier analysis argument used above breaks down  due to the  
singularity of $\phi_i(\mu,\nu)$ at the origin,
so we resort to less sophisticated methods. The relation $\phi_i( -x , -y  ) = - \phi_i( x , y  )$ implies that
\[
\big|  \sum_{ \substack{m,n\in \Z \\ (m,n)\not= (0,0)}}   \phi_i(m+\mu,n+\nu)  \big|
\le \sum_{ \substack{m,n\in \Z \\ (m,n)\not= (0,0)}} \big| \phi_i(m+\mu,n+\nu) - \phi_1(m-\mu,n-\nu) \big|.
\]
By directly computing partial derivatives, it is easy to see that there is a constant $C$, independent of
$x$, $y$, and $\xi_v$, such that
\begin{align*}
\big| \frac{\partial}{\partial x}  \phi_i(x,y) \big|  &<  
C \cdot \mathrm{exp}\left( - \frac{Q(x,y)}{ \xi_v } \right)  \left[  \frac{1}{Q(x,y)} + \frac{1}{\xi_v} \right] \\
\big| \frac{\partial}{\partial y}  \phi_i(x,y) \big|  &<  
C \cdot \mathrm{exp}\left( - \frac{Q(x,y)}{ \xi_v } \right)  \left[  \frac{1}{Q(x,y)} + \frac{1}{\xi_v} \right],
\end{align*}
and it follows  that
\begin{eqnarray*}\lefteqn{
| \phi_i(m+\mu,n+\nu) - \phi_i(m-\mu,n-\nu)  |  } \\
& \le&  
2\cdot \mathrm{sup}_{ |s| <1} \ \big|  \frac{d}{ds}  \phi_i(m+s\mu,n+s \nu) \big| \\
& \le & 
4 C \cdot \mathrm{sup}_{| s| <1} \mathrm{exp}\left( - \frac{Q(m+s\mu,n+s\nu)}{ \xi_v } \right) 
 \left[  \frac{1}{Q( m+s\mu,n+s\nu)} + \frac{1}{\xi_v} \right].
\end{eqnarray*}
Some elementary estimates  show that there are  constants  $C_1$ and $C_2$, 
independent of $\mu$, $\nu$, and $\xi_v$, such that
\begin{eqnarray*}\lefteqn{
\sum_{ \substack{m,n\in \Z \\ (m,n)\not= (0,0)}}  
 \mathrm{sup}_{| s| <1} \mathrm{exp}\left( - \frac{Q(m+s\mu,n+s\nu)}{ \xi_v } \right) 
 \left[  \frac{1}{Q( m+s\mu,n+s\nu)} + \frac{1}{\xi_v} \right] }  \\
&\le&  
C_1 \int_{x^2+y^2 > 1} \mathrm{exp} \left( -\frac{ C_2\cdot (x^2+y^2)}{\xi_v} \right)
 \left[  \frac{1}{x^2+y^2} + \frac{1}{\xi_v} \right] \, dx\, dy \\
&=&  
\pi C_1 \int_1^\infty    e^{- u C_2/\xi_v }  \frac{1}{u}  \, du
+
\pi C_1 \int_{1/\xi_v}^\infty    e^{- u C_2 }    \, du
\\
&=& \pi C_1 \cdot  \beta_1\left(\frac {C_2}{\xi_v} \right)   + O(1) \\
&=& O(\log ( \xi_v)),
\end{eqnarray*}
and the fourth claim follows.  
\end{proof}

\begin{proof}[Proof of Proposition \ref{Prop:main growth}]
As noted earlier, the boundedness of $\mathcal{E}_f$ on $\Omega$ is a consequence of  
(\ref{omega decomp}), 
Lemma \ref{Lem:omega24}, and Lemma \ref{Lem:omega13}.
Next we compute, using (\ref{derivative})  and the first estimate
of Lemma \ref{Lem:derivative estimates}, 
\begin{eqnarray*}
\frac{\partial }{\partial  q}  \sum_{\eta\in \delta_\kk\overline{\mathfrak{a}}_0}  
\beta_1 \left(  \frac{   |  \Psi_f  +\eta  |^2}{ \xi_v } \right)  
&=& 
  \frac{1}{ \xi_v } 
 \sum_{\eta\in \delta_\kk\overline{\mathfrak{a}}_0} \mathrm{exp}  \left(  - \frac{  |\Psi_f +\eta|^2  }{ \xi_v }  \right)  
    \frac{\partial \xi_v}{\partial q}  \\
&=&
\left[ \frac{\pi}{\mathrm{Vol}(\C/\delta_\kk\overline{\mathfrak{a}}_0) }   + O\left( 1/\xi_v \right) \right]  \frac{\partial \xi_v}{\partial q}.
\end{eqnarray*}
It now follows  from (\ref{partial error}) and (\ref{tau coords})  that
\[
\frac{\partial \mathcal{E}_f}{\partial q} = O\left(  \frac{1}{\xi} \frac{\partial \xi}{\partial q} \right) 
= O\left( \frac{1}{ |q|  \log|q| } \right)   .
\]
Similarly,  writing $u={}^t [u_1, \ldots, u_{n-2}]$ and 
using  the first and fourth estimates of Lemma \ref{Lem:derivative estimates},
\begin{eqnarray*}\lefteqn{
\frac{\partial }{\partial  u_i}  \sum_{  \eta\in \delta_\kk\overline{\mathfrak{a}}_0 } 
  \beta_1 \left(  \frac{  |  \Psi_f  +\eta |^2}{ \xi_v } \right)  } \\
& = & - \sum_{  \eta\in \delta_\kk\overline{\mathfrak{a}}_0}  
\mathrm{exp} \left( - \frac{  |  \Psi_f +\eta |^2}{ \xi_v } \right)  
\frac{  1}{  (  \Psi_f+\eta ) }  \frac{ \partial \Psi_f}{\partial  u_i}   
+  \sum_{  \eta\in \delta_\kk\overline{\mathfrak{a}}_0}  
\mathrm{exp} \left( - \frac{  |  \Psi_f +\eta |^2}{ \xi_v} \right)  
\frac{1}{\xi}  \ \frac{\partial \xi_v }{\partial u_i}   \\
& = &
-   \mathrm{exp} \left( - \frac{  |  \Psi_f  |^2}{ \xi_v } \right)  \frac{1}{\Psi_f}   +O(\log(\xi_v))     \\
& = &
-\frac{\partial }{\partial  u_i} \left(  \log|\Psi_f|^2 
 +   \frac{\pi \xi_v }{\mathrm{Vol}(\C/\delta_\kk\overline{\mathfrak{a}}_0) }   \right)
+ O(\log(\xi_v)),
\end{eqnarray*}
which shows that \[ \frac{\partial \mathcal{E}_f}{\partial u_i} = O(\log\log|q^{-1}|).\]
These calculations show that $\partial \mathcal{E}_f$ has $\log$-$\log$ growth along $q=0$, 
and the proof for $\overline{\partial}\mathcal{E}_f$ is the same.

The growth of $\partial \overline{\partial} \mathcal{E}_f$ is controlled in the same way.  
Lemma \ref{Lem:derivative estimates} implies  that
\begin{eqnarray*}
\frac{\partial ^2}{  \partial  q \partial \overline{u}_i}      \sum_{   \eta\in \delta_\kk\overline{\mathfrak{a}}_0  }
\beta_1 \left(  \frac{  |  \Psi_f +\eta  |^2}{ \xi_v} \right)  &= &
 -   \sum_{  \eta\in \delta_\kk\overline{\mathfrak{a}}_0   }
 \mathrm{exp} \left( - \frac{  |  \Psi_f +\eta |^2}{ \xi_v} \right) 
 \frac{  (\Psi_f+\eta)}{\xi_v^2}   \frac{\partial \overline{\Psi}_f }{\partial \overline{u}_i}    \frac{\partial \xi_v}{\partial q} \\
& & 
+   \sum_{    \eta\in \delta_\kk\overline{\mathfrak{a}}_0  }  \mathrm{exp} \left( - \frac{  |  \Psi_f +\eta |^2}{ \xi_v} \right)  \frac{   |\Psi_f+\eta|^2 - \xi_v }{\xi_v^3} \frac{\partial \xi_v}{\partial \overline{u}_i}
    \frac{\partial \xi_v}{\partial q} \\
& = & 
O(1/\xi_v)  \frac{\partial \overline{\Psi}_f  }{\partial \overline{u}_i}    \frac{\partial \xi_v}{\partial q} 
+ O(1/\xi_v) \frac{\partial \xi_v}{\partial \overline{u}_i}  \frac{\partial \xi_v}{\partial q} \\
& = & O\left( \frac{1}{ |q| \log|q|} \right),
\end{eqnarray*}
which implies 
\[ 
\frac{\partial ^2\mathcal{E}_f}{  \partial  q \partial \overline{u}_i} = O\left(\frac{1}{|q| \log|q|}\right).
\]
The same method shows that
\[
\frac{\partial ^2\mathcal{E}_f}{  \partial \overline{ q} \partial u_i}= O\left(\frac{1}{ |q| \log|q|}\right),
\]
and  similarly
\begin{eqnarray*}
\frac{\partial ^2}{   \partial  q  \partial \overline{q} }   \sum_{   \eta\in \delta_\kk\overline{\mathfrak{a}}_0  } 
 \beta_1 \left(  \frac{   |  \Psi_f +\eta  |^2}{ \xi_v} \right)
&=&
   \sum_{   \eta\in \delta_\kk\overline{\mathfrak{a}}_0  }  \mathrm{exp} \left( - \frac{ |  \Psi_f +\eta |^2}{ \xi_v} \right)  
\left[  
\frac{ |\Psi_f+\eta|^2 - \xi_v}{\xi_v^3}
\right]  \frac{\partial \xi_v}{\partial q} \frac{\partial \xi_v}{\partial \overline{q}} \\
&=&
O(1/\xi_v^2) \cdot  \frac{\partial \xi_v}{\partial q} \frac{\partial \xi_v}{\partial \overline{q}}
\end{eqnarray*}
shows that 
\[
\frac{\partial ^2\mathcal{E}_f}{   \partial  q  \partial \overline{q} } =O\left(  \frac{1}{  (|q| \log|q|)^2} \right).
\]
Finally, Lemma \ref{Lem:derivative estimates} shows that
\begin{eqnarray*}\lefteqn{
\frac{\partial^2}{  \partial u_i  \partial \overline{u}_j} 
   \sum_{   \eta\in \delta_\kk\overline{\mathfrak{a}}_0  }   \beta_1 \left(  \frac{  |  \Psi_f +\eta  |^2}{ \xi_v} \right)  } \\
&=&
 \frac{1}{\xi_v}   \sum_{   \eta\in \delta_\kk\overline{\mathfrak{a}}_0  }   \mathrm{exp} \left( - \frac{  |  \Psi_f +\eta |^2}{ \xi_v} \right) 
\left(
\frac{\partial^2 \xi_v}{  \partial u_i  \partial \overline{u}_ j}
- \frac{1}{\xi_v} \frac{\partial \xi_v}{\partial u_i}   \frac{\partial \xi_v}{\partial \overline{u}_j} \right)  \\
& & 
 +   \frac{1}{\xi_v}   \sum_{   \eta\in \delta_\kk\overline{\mathfrak{a}}_0  }   \mathrm{exp} \left( - \frac{  |  \Psi_f +\eta |^2}{ \xi_v} \right) 
  \left(   \frac{\partial \Psi_f}{\partial u_i} - \frac{ (\Psi_f+\Lambda) }{\xi_v} \frac{\partial \xi_v}{\partial u_i} \right)
\left(   \frac{\partial \overline{ \Psi_f }}{\partial \overline{u}_j } -  \frac{\overline{(\Psi_f+\Lambda)} }{\xi_v} \frac{\partial \xi_v}{\partial \overline{u}_j} \right)  \\
& = & O(1),
\end{eqnarray*}
and it follows that 
\[
\frac{\partial^2 \mathcal{E}_f}{  \partial u_i  \partial \overline{u}_j} =O(1) .
\]
Thus $\partial\overline{\partial}\mathcal{E}_f$ has $\log$-$\log$ growth.

All claims of  Proposition \ref{Prop:main growth} follow from   (\ref{f main growth}) and the 
discussion above.
\end{proof}


\subsection{Arithmetic Kudla-Rapoport divisors}
\label{ss:arith KR}


Fix $m\not=0$ and $v\in \R^+$.  We will define a class
\[
\widehat{\mathcal{Z}} (m,v) \in \widehat{\mathrm{CH}}_\R^1(\mathcal{M}^*)
\]
in the arithmetic Chow group of Section \ref{ss:BKK}.

  The boundary $\partial\mathcal{M}^*$ of  $\mathcal{M}^*$ is a disjoint union 
of smooth irreducible divisors.  Denote  by $\pi_0(\partial\mathcal{M}^*)$  
the set of connected  components of $\partial\mathcal{M}^*_{/\kk^\alg}$.
Recall  from Section \ref{ss:complex uniform} 
that the  components of $\partial\mathcal{M}^*_{/\kk^\alg}$ are indexed by triples 
$(\mathfrak{A}_0 , \mathfrak{m} \subset \mathfrak{A} )$,
and  to each triple there is an associated self dual Hermitian lattice
$L= \Hom_{\co_\kk}(\mathfrak{A}_0,\mathfrak{A})$ of signature $(n-1,1)$.
The isotropic line $\mathfrak{m}\subset \mathfrak{A}$ determines an isotropic line
 \[
 \mathfrak{a}=\Hom_{\co_\kk}( \mathfrak{A}_0 , \mathfrak{m} ) \subset L
 \] 
 and the quotient 
 \[
 \mathfrak{a}^\perp/\mathfrak{a} \iso \Hom_{\co_\kk}(\mathfrak{A}_0, \mathfrak{m}^\perp/\mathfrak{m})
 \]
is  a self dual Hermitian lattice of signature $(n-2,0)$.

\begin{Def}
The \emph{$m$-index} of the  boundary component $\mathcal{C}\in \pi_0(\partial\mathcal{M}^*)$ indexed by
$(\mathfrak{A}_0,  \mathfrak{m} \subset  \mathfrak{A})$  is 
\[
\mathrm{Ind}_\mathcal{C} (m) =   \# \{ f \in \mathfrak{a}^\perp/\mathfrak{a} : \langle f,f\rangle =m \}.
\]
\end{Def}

\begin{Prop}
For any boundary component $\mathcal{C}\in \pi_0(\partial\mathcal{M}^*)$,
\[
\mathrm{Ind}_\mathcal{C}(m) =0 \iff  \mathcal{C}(\C)\cap \mathcal{Z}^*(m) (\C) =\emptyset .
\]
Furthermore, $\mathrm{Ind}_\mathcal{C}(m) =\mathrm{Ind}_{\mathcal{C}^\sigma}(m)$
for every  $\sigma\in \Gal(\kk^\alg/\kk)$.
\end{Prop}

\begin{proof}
It is clear from (\ref{L decomp}) that $\mathrm{Ind}_\mathcal{C}(m) =0$ if and only if $L^\bndry(m) = \emptyset$.
If $L^\bndry(m) = \emptyset$ then (\ref{boundary KR}) shows that the 
component $\mathcal{C}(\C)$ has an open neighborhood in $\mathcal{M}^*(\C)$
which does not intersect $\mathcal{Z}(m)(\C)$, and so 
$\mathcal{C}(\C)\cap \mathcal{Z}^*(m) (\C) =\emptyset$.  On the other hand, if 
$L^\bndry(m) \not= \emptyset$ then (\ref{boundary KR}) shows that
for every $f\in L^\bndry(m)$
the pullback of (the support of) $\mathcal{Z}(m)(\C)$ to   
$C_\Gamma \backslash \mathcal{D}^\epsilon$  contains the vanishing locus of the function 
$\Psi_f(\hh) = -\delta_\kk \overline{a} + {}^t\overline{b} Au$.  If we  fix any solution 
$u_0\in \C^{m-2}$ to ${}^t\overline{b} Au = \delta_\kk\overline{a}$, then the point $(q,u)=(0,u_0)$
of  (\ref{boundary coords}) lies in $\mathcal{C}(\C)\cap \mathcal{Z}^*(m) (\C)$.

For the second claim, recall from Section \ref{ss:boundary labels} 
that the boundary components of $\mathcal{M}_{(n-1,1)/\kk^\alg}$ are indexed 
by the cusp labels of Definition \ref{Def:cusp label} (with $m=n-1$).  As the points of 
$\mathcal{M}_{(1,0)/\kk^\alg}$ are indexed by self dual Hermitian lattices of signature
$(1,0)$, the boundary components 
$\mathcal{C} \in \pi_0(\partial\mathcal{M}^*)$ are indexed by \emph{extended cusp labels}: triples
$(\mathfrak{A}_0,\mathfrak{n},\mathfrak{B})$ where $\mathfrak{A}_0$ a 
self dual Hermitian lattice of signature $(1,0)$, $\mathfrak{n}$ is a projective $\co_\kk$-module 
of rank one, and $\mathfrak{B}$ is a self dual Hermitian lattice of signature $(n-2,0)$.  
If $\mathcal{C}$ is indexed, in our old language, by the triple 
$(\mathfrak{A}_0, \mathfrak{m}  \subset \mathfrak{A} )$,
then its associated extended cusp label is $(\mathfrak{A}_0,\mathfrak{n},\mathfrak{B})$
where  $\mathfrak{n} = \mathfrak{A} / \mathfrak{m}^\perp$
and  $\mathfrak{B} = \mathfrak{m}^\perp  / \mathfrak{m}$.
In this notation,
\[
\mathfrak{a}^\perp / \mathfrak{a} \iso \Hom_{\co_\kk} (\mathfrak{A}_0 , \mathfrak{B})
\]
as Hermitian lattices.  The essential point is that 
Proposition \ref{Prop:galois bnd action} (and the usual theory of complex multiplication 
for elliptic curves) tell us that replacing $\mathcal{C}$ by  $\mathcal{C}^\sigma$ has the effect 
of replacing $\Hom_{\co_\kk} (\mathfrak{A}_0 , \mathfrak{B})$ by 
$\Hom_{\co_\kk} (\mathfrak{A}_0 \otimes \mathfrak{s}^{-1} , \mathfrak{B}\otimes \mathfrak{s}^{-1})$,
where the fractional ideal $\mathfrak{s}$ is chosen so that $\mathrm{rec}_\kk(\mathfrak{s})$
and $\sigma$ agree on the Hilbert class field of $\kk$.
But  the canonical isomorphism  of $\co_\kk$-modules
\[
\Hom_{\co_\kk} (\mathfrak{A}_0 , \mathfrak{B}) \iso 
\Hom_{\co_\kk} (\mathfrak{A}_0 \otimes \mathfrak{s}^{-1} , \mathfrak{B} \otimes \mathfrak{s}^{-1})
\]
respects the Hermitian forms, and so  
$\mathrm{Ind}_\mathcal{C}(m)$ and $\mathrm{Ind}_{\mathcal{C}^\sigma}(m)$ 
count the number of vectors of norm $m$ in isomorphic Hermitian lattices.
\end{proof}

By the previous proposition, for any $m\not=0$ the formal linear combination 
of geometric boundary components 
\begin{equation}\label{KR compactifier}
\mathcal{B}(m)_{/\kk} = \sum_{\mathcal {C}\in \pi_0(\partial\mathcal{M}^*) } \mathrm{Ind}_\mathcal{C}(m) \mathcal{C},
\end{equation}
\emph{a priori} a divisor on $\mathcal{M}^*_{/\kk^\alg}$,
descends to a divisor on $\mathcal{M}^*_{/\kk}$.  Denote
by $\mathcal{B}(m)$ the Zariski closure of (\ref{KR compactifier}) in $\mathcal{M}^*$. 
If $m<0$ then $\mathcal{Z}^*(m)=\emptyset$, and the previous proposition shows that
$\mathcal{B}(m)=\emptyset$.

\begin{Def}\label{Def:arithmetic KR}
Given  $m\not=0$ and $v\in \R^+$, define a divisor on $\mathcal{M}^*$ with real coefficients
\[
\mathcal{B}(m,v) = \frac{1}{4\pi v}\mathcal{B}(m) .
\]
The \emph{arithmetic Kudla-Rapoport divisor} is
\[
\widehat{\mathcal{Z}} (m,v) = \big( \mathcal{Z}^*(m) 
+  \mathcal{B}(m,v) , \Gr(m,v,\cdot)  \big)
\in \widehat{\mathrm{CH}}_\R^1(\mathcal{M}^*).
\]
\end{Def}

Of course for the definition to make sense we need to know that $\Gr(m,v,\cdot)$
is a Green function for the divisor $\mathcal{Z}^*(m) +  \mathcal{B}(m,v)$.
This is the content of the following theorem.

\begin{Thm}\label{Thm:good green}
Suppose $z$ is a complex point of some boundary component $\mathcal{C}\in\pi_0(\partial\mathcal{M}^*)$.
There is an open neighborhood  $V\subset \mathcal{M}^*(\C)$ of $z$ such that
the  smooth function 
\[
\mathcal{E}(\hh)=  \Gr(m,v,\hh)+  \log|\psi_m(\hh)|^2  
+  \frac{  \mathrm{Ind}_\mathcal{C} (m) }{4 \pi v} \log|q(\hh)|^2 
\]
on $V\smallsetminus \mathcal{C}(\C)$ is bounded, and the differential forms 
$\partial\mathcal{E}$, $\overline{\partial} \mathcal{E}$, and 
$\partial\overline{\partial} \mathcal{E}$  have  log-log growth along $\mathcal{C}(\C)$.
Here $\psi_m(\hh) =0$ is a local equation for  $\mathcal{Z}^*(m)(\C)$,
and $q(\hh)=0$ is a local equation for  the boundary component $\mathcal{C}(\C)$.
\end{Thm}

\begin{proof}
Recalling that $\Gr(m,v,\hh) =  \Gr^\bndry(m,v,\hh) + \Gr^\interior(m,v,\hh)$,
 Propositions \ref{Prop:growth I} and \ref{Prop:main growth} show that
 \[
 \Gr(m,v,\hh) + \log |\psi_m(\hh) |^2  
 - \frac{ \mathrm{Ind}_\mathcal{C}(m) \xi(\hh) }{ 4 v \mathrm{Vol}(\C/\delta_\kk \overline{\mathfrak{a}}_0) }
 \]
 is bounded, and its first and second order derivatives have log-log growth.  
 The equality (\ref{tau coords})  shows that the function 
 \[
 \frac{ d_\kk^{3/2} \mathrm{N}(\mathfrak{a}_0) }{2\pi} \log|q(\hh)|^2
 =   \frac{ \mathrm{Vol}(\C/\delta_\kk \overline{\mathfrak{a}}_0 )}{\pi} \log|q(\hh)|^2
 \]
differs from $-\xi(\hh)$ by a function extending smoothly across $\mathcal{C}(\C)$,
and the claim follows.
\end{proof}


\section{Intersections with CM cycles}


In this section we define a one dimensional stack $\mathcal{X}_\Phi$  
as a moduli space of abelian schemes with complex multiplication.
The stack $\mathcal{X}_\Phi$ admits a canonical morphism to $\mathcal{M}$.
By  reducing to the calculations of \cite{howardCM}, we  
compute the  arithmetic intersection agains $\mathcal{X}_\Phi$ of the Kudla-Rapoport divisors of 
Definition \ref{Def:arithmetic KR}, and relate these intersection numbers
to the Fourier coefficients of Eisenstein series.

Let $F$ be a totally real \'etale $\Q$-algebra  (in other words,  a product of totally real fields) 
of degree $n$.  In the main results $F$ will be a totally real field.   Define a CM algebra
\[
K = \kk\otimes_\Q F.
\]
 

\subsection{CM cycles}
\label{ss:cm cycles}


Recall that we have fixed an embedding  $\iota: \kk\to \C$.  
A  CM type $\Phi \subset\Hom_{\Q-\mathrm{alg}}(K,\C)$
is said to have \emph{signature}  $(n-1,1)$  if there is a unique $\varphi^\mathrm{sp} \in\Phi$
whose restriction to $\kk$ is equal to $\overline{\iota}$.  This distinguished $\varphi^\mathrm{sp}$
is called the \emph{special element} of $\Phi$.  There are $n$ distinct CM types of 
signature $(n-1,1)$, and each is uniquely determined by its special element.  
This fact implies that the subfield  of the complex numbers
\[
K_\Phi = \varphi^\mathrm{sp}(K)
\]
 contains the reflex field of $\Phi$, and in fact is equal to the reflex field except in 
 the degenerate case $n=2$.   In this degenerate  case  $K_\Phi$ is a biquadratic CM field, and the 
 reflex field of $\Phi$ is the unique quadratic imaginary subfield of $K_\Phi$ which is not 
 isomorphic to $\kk$.  In any case, let  $\co_\Phi \subset K_\Phi$ be the ring of integers.

 Fix a CM type $\Phi$ of signature $(n-1,1)$.
 The fixed embedding  $\iota:\co_\kk \to \C$ takes values in $\co_\Phi$, and we use this 
map to view $\co_\Phi$ as an $\co_\kk$-algebra.  Note that this map 
is the complex  conjugate of the composition
\[
\co_\kk \to \co_K \map{\varphi^\mathrm{sp}} \co_\Phi . 
\]
Fix also a set $x_1,\ldots, x_r \in \co_K$ of $\Z$-module generators of $\co_K$, and 
define a polynomial
\begin{equation}\label{det}
\mathrm{det}_\Phi(T_1,\ldots, T_r) = \prod_{\varphi\in\Phi}
( T_1\varphi(x_1) + \cdots + T_r\varphi(x_r) ) \in \co_\Phi[T_1,\ldots, T_r].
\end{equation}

As in Section \ref{ss:notation}, if $A\to S$ is an abelian scheme over an arbitrary base scheme, 
equipped with  an action $\kappa:\co_K\to \End(A)$, there is an induced action 
$x\mapsto  \kappa(\overline{x})^\vee$ of $\co_K$ on the dual abelian scheme $A^\vee$.

\begin{Def}
The \emph{CM cycle} $\mathcal{CM}_\Phi$ is the $\co_\Phi$-stack
parametrizing triples $(A,\kappa,\psi)$ in which
\begin{itemize}
\item $A\to S$ is an abelian scheme of relative dimension $n$ over an $\co_\Phi$-scheme $S$,
\item
$\kappa :\co_K \to \End(A)$ is an action of $\co_K$ on $A$,
\item
$\psi:A \to A^\vee$ is an $\co_K$-linear principal polarization of $A$,
\item
the pair $(A,\kappa)$ has CM type $\Phi$, 
in the sense that  the determinant 
\[
\det( T_1 x_1 + \cdots + T_r x_r ; \Lie(A))
\] 
is equal to the image of (\ref{det}) under $\co_\Phi[T_1,\ldots,T_r] \to \co_S[T_1,\ldots,T_r]$.
\end{itemize}
\end{Def}

The stack $\mathcal{CM}_\Phi$ is smooth and proper of relative dimension $0$ over
$\co_\Phi$, by \cite[Proposition 3.1.2]{howardCM}.

\begin{Lem}
Suppose $R$ is an $\co_\Phi$-algebra,  let
$\bm{J}_{\varphi^\mathrm{sp}}$ be the kernel of the ring homomorphism
$\co_K \otimes_\Z R \to R$
defined by $x\otimes r \to \varphi^\mathrm{sp}(x) r$, and  use the composition 
\[
i_R: \co_\kk \map{\iota} \co_\Phi \to R
\]
to view $R$ as an $\co_\kk$-algebra.
For any  $(A,\kappa,\psi) \in\mathcal{CM}_\Phi(R)$,  the
$\co_K$-stable $R$-submodule $\mathcal{F} = \bm{J}_{\varphi^\mathrm{sp}} \Lie(A)$ of 
$\Lie(A)$ satisfies the following properties: the quotient $\Lie(A)/\mathcal{F}$ is 
a locally free $R$-module of rank one, $\co_\kk$ acts on $\mathcal{F}$ through the 
structure map $i_R:\co_\kk\to R$, and $\co_\kk$ acts on $\Lie(A)/\mathcal{F}$ through
the complex conjugate of the structure map.
\end{Lem}

\begin{proof}
It suffices to prove this for the universal object over $\mathcal{CM}_\Phi$.
One can easily reduce further  to the case where $R$ is the completion of the \'etale local ring
of a closed geometric point  $z$ of $\mathcal{CM}_\Phi$, and $(A,\kappa,\psi)$ is the 
pullback of the universal object.
Such a geometric point has the form $z\in \mathcal{CM}_\Phi(\F_\mathfrak{p}^\alg)$  for some prime 
$\mathfrak{p} \subset \co_\Phi$.   Let $\C_\mathfrak{p}$ be the completion of 
an algebraic closure of $K_{\Phi,\mathfrak{p}}$, identify the residue field of  
$\co_{\C_\mathfrak{p}}$ with our fixed copy of  $\F_\mathfrak{p}^\alg$,
and fix a $K_\Phi$-algebra isomorphism $\C \iso \C_\mathfrak{p}$.  
As $\mathcal{CM}_\Phi$ is smooth of relative dimension $0$ 
over $\co_\Phi$, the ring  $R$ is isomorphic to the ring of integers of the 
completion of the maximal unramified 
extension of $K_{\Phi,\mathfrak{p}}$ inside $\C_\mathfrak{p}$.

We now view each $\varphi\in \Phi$ as taking values in $\C_\mathfrak{p}$, and 
let $\bm{J}_\Phi$ be the kernel of the ring homomorphism
\[
\co_K \otimes_\Z R \to \prod_{\varphi\in\Phi} \C_\mathfrak{p}
\]
defined by sending $x\otimes r$ to the tuple $( \varphi(x)r)_{\varphi}$. 
The results of \cite[Section 2.1]{howardCM} tell us that 
\[
\Lie(A) \iso (\co_K \otimes_\Z R)/\bm{J}_\Phi
\]
as $\co_K\otimes_\Z R$-modules.  Obviously $\bm{J}_\Phi\subset \bm{J}_{\varphi^\mathrm{sp}}$, and  so 
\[
\Lie(A) / \bm{J}_{\varphi^\mathrm{sp}}\Lie(A) \iso (\co_K \otimes_\Z R) /\bm{J}_{\varphi^\mathrm{sp}} \iso R.
\]
It is clear that $\co_K$ acts on $\Lie(A) / \bm{J}_{\varphi^\mathrm{sp}}\Lie(A)$ through $\varphi^\mathrm{sp}$,
and so $\co_\kk$ acts through $\varphi^\mathrm{sp}|_{\co_\kk}=\overline{ \iota }$, as desired.  
Finally, we must show that $\co_\kk$ acts on $\bm{J}_{\varphi^\mathrm{sp}}\Lie(A)$ through $\iota$.
This is clear from
\[
\bm{J}_{\varphi^\mathrm{sp}} \cdot \Lie(A) \iso \bm{J}_{\varphi^\mathrm{sp}}/\bm{J}_\Phi \hookrightarrow
\prod_{ \substack{  \varphi\in\Phi \\    \varphi\not=\varphi^\mathrm{sp} } } \C_\mathfrak{p}
\]
and the hypothesis $\varphi|_\kk =\iota$ for every $\varphi$ appearing 
in the product.
\end{proof}

The condition that $\Phi$ has signature $(n-1,1)$ implies that 
$(A,\kappa,\psi) \mapsto (A,\kappa|_{\co_\kk},\psi)$ defines a morphism
\[
\mathcal{CM}_\Phi \to \mathcal{M}^\naive_{(n-1,1) / \co_\Phi}.
\]  
The lemma says that this morphism lifts  to a canonical morphism
\[
\mathcal{CM}_\Phi \to \mathcal{M}_{(n-1,1) / \co_\Phi}
\]
defined by $(A,\kappa,\psi) \to (A,\kappa|_{\co_\kk},\psi,\mathcal{F})$.
In particular,   the  $\co_\Phi$-stack
\[
\mathcal{X}_\Phi =   \mathcal{M}_{(1,0)} \times_{\co_\kk} \mathcal{CM}_\Phi  
\]
admits a canonical morphism $\mathcal{X}_\Phi  \to \mathcal{M}_{/ \co_\Phi }.$


\subsection{The intersection formula}
\label{ss:final results}


In this subsection we assume that  the discriminant of  $F$ is odd and relatively  prime to $d_\kk$.

The structure morphism $\mathcal{X}_\Phi \to  \Spec(\co_\Phi)$ is proper and smooth of relative dimension $0$.
In particular $\mathcal{X}_\Phi$, now viewed as a stack over $\co_\kk$, is regular, and the structure 
map $\mathcal{X}_\Phi \to \Spec(\co_\kk)$ is finite and flat. 
Using the $\co_\kk$-morphism $\mathcal{X}_\Phi \to \mathcal{M}$, we obtain from  Section \ref{ss:BKK}  
a linear functional
\[
[ \cdot : \mathcal{X}_\Phi ] : \widehat{\mathrm{CH}}_\R^1(\mathcal{M}^* ) \to \R.
\]
We will evaluate this linear functional on the arithmetic Kudla-Rapoport divisors
of Definition \ref{Def:arithmetic KR}, at least under the hypothesis that $F$ is a field.
This hypothesis implies that $\mathcal{X}_\Phi$ and $\mathcal{Z}(m)$ 
intersect properly \cite[Theorem 3.8.4]{howardCM}, and so
\[
[  \widehat{\mathcal{Z}}(m,v)    :  \mathcal{X}_\Phi ] 
= I_\mathrm{fin}( \mathcal{Z}^*(m)  : \mathcal{X}_\Phi )
+  \Gr(m,v, \mathcal{X}_\Phi).
\]

\begin{Thm}\label{Thm:finite part}
Suppose  $m>0$.  If $F$ is a field then 
\[
I_\mathrm{fin}( \mathcal{Z}^*(m)  : \mathcal{X}_\Phi ) =
 \frac{ h(\kk)}{   w(\kk)  } 
 \sum_{  \substack{ \alpha\in F^{\gg 0}    \\  \mathrm{Tr}_{F/\Q}(\alpha)=m } }
\sum_{   \mathfrak{p} \subset \co_F      }    \log(\mathrm{N}(\mathfrak{p}))  
 \cdot \ord_{\mathfrak{p}} (\alpha \mathfrak{p}\mathfrak{d}_F  )
\cdot   \rho ( \alpha\mathfrak{p}^{-\epsilon_\mathfrak{p}}\mathfrak{d}_F ).
\]
Here $h(\kk)$ is the class number of $\kk$, $w(\kk)$ is the number of roots of unity in $\kk^\times$, 
the inner sum is over all primes  $\mathfrak{p}$ of $F$ nonsplit  in $K$, 
$\mathfrak{d}_F$ is the different of $F/\Q$, 
\[
\epsilon_\mathfrak{p} = \begin{cases}
1 & \mbox{if $\mathfrak{p}$ is unramified in $K$} \\
0 & \mbox{if $\mathfrak{p}$ is ramified in $K$,}
\end{cases}
\]
and 
\[
\rho(\mathfrak{a}) = 
\# \{ \mathfrak{A}\subset\co_K : \mathfrak{A}\overline{\mathfrak{A}} = \mathfrak{a} \co_K\}
\]
for any fractional $\co_F$-ideal $\mathfrak{a}$ (in particular, 
$\rho(\mathfrak{a})=0$ unless $\mathfrak{a}\subset\co_F$).
\end{Thm}

\begin{proof}
We reduce to the results of \cite{howardCM}.  The first observation is that the 
local rings of $\mathcal{Z}^*(m)$ and $\mathcal{X}_\Phi$  are Cohen-Macaulay,
by Proposition \ref{Prop:local divisor} and the smoothness of $\mathcal{X}_\Phi$ over $\co_\Phi$.
This implies, by a result of Serre \cite[p.~111]{serre00}, that the higher $\mathrm{Tor}$
terms vanish in the Serre intersection multiplicity.  Thus if we abbreviate
\begin{align*}
\mathcal{Y} &= \mathcal{Z}(m) \times_{\mathcal{M}} \mathcal{X}_\Phi \\
& \iso  ( \mathcal{Z}^\naive(m) \times_{\mathcal{M}^\naive} \mathcal{M} ) 
\times_{\mathcal{M}} \mathcal{X}_\Phi \\
& \iso \mathcal{Z}^\naive(m) \times_{\mathcal{M}^\naive} \mathcal{X}_\Phi
\end{align*}
we find
\begin{align*}
I_\mathrm{fin}(\mathcal{Z}^*(m)  : \mathcal{X}_\Phi ) 
=
\sum_{ \mathfrak{p} \subset\co_\kk }
\sum_{  y\in \mathcal{Y} (\F_\mathfrak{p}^\alg)     }
\frac{  \log(\mathrm{N}(\mathfrak{p}))}{\#\Aut(y) } 
\mathrm{length}_{ \co_{\mathcal{Y},y}  } (   \co_{ \mathcal{Y} , y} )  \\
= 
\sum_{ \mathfrak{q} \subset\co_\Phi }
\sum_{  y\in \mathcal{Y} (\F_\mathfrak{q}^\alg)     }
\frac{  \log(\mathrm{N}(\mathfrak{q}))}{\#\Aut(y) } 
\mathrm{length}_{ \co_{\mathcal{Y},y}  } (  \co_{ \mathcal{Y} , y} )
\end{align*}
(in the first line the inner sum is over morphisms $\Spec(\F_\mathfrak{p}^\alg) \to \mathcal{Y}$ of $\co_\kk$-stacks,
while in the second line the inner sum is over morphisms $\Spec(\F_\mathfrak{q}^\alg) \to \mathcal{Y}$ of 
$\co_\Phi$-stacks). The final expression is  computed in Theorems 3.7.2 and 3.8.4 of  \cite{howardCM}.
\end{proof}

The archimedean contribution to the intersection multiplicity was computed in 
Theorems 3.7.2 and 3.8.6 of \cite{howardCM}.  The result is as follows.

\begin{Thm}\label{Thm:infinite part}
Fix any nonzero $m\in \Z$, and any $v\in \R^+$.  If $F$ is a field then
\[
\Gr(m,v, \mathcal{X}_\Phi)=  \frac{ h(\kk)}{    w(\kk)  }   
 \sum_{  \substack{ \alpha\in F_-    \\  \mathrm{Tr}_{F/\Q}(\alpha)=m } }
 \beta_1(4\pi v |\alpha| ) \cdot  \rho(\alpha \mathfrak{d}_F).
\]
Here $F_-$ is the set of elements of $F$ that are negative at exactly one archimedean place, 
and  $|\alpha|$ is the absolute value of $\alpha$ at the unique archimedean 
place $v$ at which $\alpha_v<0$.
The notations $h(\kk)$, $w(\kk)$, and $\rho$ have the same meaning as in 
Theorem \ref{Thm:finite part}, and $\beta_1$ is the function (\ref{beta}).
\end{Thm}

Recall the nonholomorphic modular form
\[
\mathcal{E}'_\Phi(i_F(\tau),0) = \sum_{m\in \Z} c_\Phi(m,v)\cdot q^m
\]
of the introduction.  The coefficients $c_\Phi(m,v)$ were computed in \cite{howardCM},
using  results of Yang \cite{Yang05}.  It was shown there that 
\[
c_\Phi(m,v) = \sum_{ \substack{\alpha\in F \\ \mathrm{Tr}(\alpha)=m } } b_\Phi(\alpha,v),
\]
for some real numbers  $b_\Phi(\alpha,v)$ determined explicitly 
by the formulas of \cite[Corollary 4.2.2]{howardCM}. Comparing those
formulas with  Theorems \ref{Thm:finite part} and \ref{Thm:infinite part} yields the following result.

\begin{Thm}\label{Thm:main intersection}
Fix any nonzero $m\in \Z$, and any $v\in \R^+$.  If $F$ is a field then
\[
[  \widehat{\mathcal{Z}}(m,v)  :  \mathcal{X}_\Phi ]  = 
  -   \frac{  h(\kk)}{    w(\kk)  }  \cdot \frac{ \sqrt{\mathrm{N}( d_{K/F} ) }  }{2^{r-1} }  \cdot c_\Phi(m,v).
\]
Here $d_{K/F}$ is the discriminant of $K/F$, and 
$r$ is the number of primes of $F$ ramified in $K$, including the archimedean primes.
\end{Thm}

\begin{Conj}\label{Conj:improper}
Theorem \ref{Thm:main intersection} holds without the hypothesis that $F$ is a field.
That is to say, it also holds for $F$ a product of totally real fields.
\end{Conj}

The point, of course, is that if $F$ is not a field then $\mathcal{X}_\Phi$ and $\mathcal{Z}(m)$
intersect improperly in $\mathcal{M}$.  This means that 
$[  \widehat{\mathcal{Z}}(m,v) : \mathcal{X}_\Phi ] $ cannot
be computed using the simple formula (\ref{proper degree}), and so 
 Conjecture \ref{Conj:improper} is considerably  more challenging than Theorem \ref{Thm:main intersection}.

\bibliographystyle{plain}

\begin{thebibliography}{10}

\bibitem{zavosh}
Z.~Amir-Khosravi.
\newblock {\em Moduli of Abelian Schemes and Serre's Tensor Construction}.
\newblock PhD thesis, University of Toronto, 2013.

\bibitem{AMRT}
A.~Ash, D.~Mumford, M.~Rapoport, and Y.-S. Tai.
\newblock {\em Smooth compactifications of locally symmetric varieties}.
\newblock Cambridge Mathematical Library. Cambridge University Press,
  Cambridge, second edition, 2010.
\newblock With the collaboration of Peter Scholze.

\bibitem{BK}
R.~Berndt and U.~K{\"u}hn.
\newblock On {K}udla's {G}reen function for signature $(2,2)$, part {I}.
\newblock {\em Preprint}.

\bibitem{BGS}
J.-B. Bost, H.~Gillet, and C.~Soul{\'e}.
\newblock Heights of projective varieties and positive {G}reen forms.
\newblock {\em J. Amer. Math. Soc.}, 7(4):903--1027, 1994.

\bibitem{bruinier99}
J.~Bruinier.
\newblock Borcherds products and {C}hern classes of {H}irzebruch-{Z}agier
  divisors.
\newblock {\em Invent. Math.}, 138(1):51--83, 1999.

\bibitem{bruinier-burgos-kuhn}
J.~Bruinier, J.~I. Burgos~Gil, and U.~K{\"u}hn.
\newblock Borcherds products and arithmetic intersection theory on {H}ilbert
  modular surfaces.
\newblock {\em Duke Math. J.}, 139(1):1--88, 2007.

\bibitem{BKK}
J.~I. Burgos~Gil, J.~Kramer, and U.~K{\"u}hn.
\newblock Cohomological arithmetic {C}how rings.
\newblock {\em J. Inst. Math. Jussieu}, 6(1):1--172, 2007.

\bibitem{cogdell81}
J.~W. Cogdell.
\newblock {\em Arithmetic quotients of the complex $2$-ball and modular forms
  of {N}ebentypus}.
\newblock PhD thesis, Yale University, 1981.

\bibitem{cogdell85}
J.~W. Cogdell.
\newblock Arithmetic cycles on {P}icard modular surfaces and modular forms of
  {N}ebentypus.
\newblock {\em J. Reine Angew. Math.}, 357:115--137, 1985.


\bibitem{faltings-chai}
G.~Faltings and C.-L. Chai.
\newblock {\em Degeneration of abelian varieties}, volume~22 of {\em Ergebnisse
  der Mathematik und ihrer Grenzgebiete (3) [Results in Mathematics and Related
  Areas (3)]}.
\newblock Springer-Verlag, Berlin, 1990.
\newblock With an appendix by David Mumford.

\bibitem{gillet09}
H.~Gillet.
\newblock Arithmetic intersection theory on {D}eligne-{M}umford stacks.
\newblock In {\em Motives and algebraic cycles}, volume~56 of {\em Fields Inst.
  Commun.}, pages 93--109. Amer. Math. Soc., Providence, RI, 2009.

\bibitem{gillet-soule90}
H.~Gillet and C.~Soul{\'e}.
\newblock Arithmetic intersection theory.
\newblock {\em Inst. Hautes \'Etudes Sci. Publ. Math.}, (72):93--174,
  1990.

\bibitem{gubler}
W.~Gubler.
\newblock Moving lemma for {$K\sb 1$}-chains.
\newblock {\em J. Reine Angew. Math.}, 548:1--19, 2002.

\bibitem{howardCM}
B.~Howard.
\newblock Complex multiplication cycles and {K}udla-{R}apoport divisors.
\newblock {\em Ann. of Math.}, 176:1097--1171, 2012.

\bibitem{HY-HZ}
B.~Howard and T.~Yang.
\newblock {\em Intersections of {H}irzebruch-{Z}agier divisors and {C}{M}
  cycles}, volume 2041 of {\em Lecture Notes in Mathematics}.
\newblock Springer-Verlag, 2011.

\bibitem{kramer}
N.~Kr{\"a}mer.
\newblock Local models for ramified unitary groups.
\newblock {\em Abh. Math. Sem. Univ. Hamburg}, 73:67--80, 2003.

\bibitem{kudla97}
S.~Kudla.
\newblock Central derivatives of {E}isenstein series and height pairings.
\newblock {\em Ann. of Math. (2)}, 146(3):545--646, 1997.

\bibitem{kudla04b}
S.~Kudla.
\newblock Special cycles and derivatives of {E}isenstein series.
\newblock In {\em Heegner points and Rankin $L$-series}, volume~49 of {\em
  Math. Sci. Res. Inst. Publ.}, pages 243--270. Cambridge Univ. Press,
  Cambridge, 2004.


\bibitem{KRunitaryI}
S.~Kudla and M.~Rapoport.
\newblock Special cycles on unitary {S}himura varieties {I}. {U}nramified local
  theory.
\newblock {\em Invent. Math.}, 184(3):629--682, 2011.

\bibitem{KRunitaryII}
S.~Kudla and M.~Rapoport.
\newblock Special cycles on unitary {S}himura varieties {II}. {G}lobal theory.
\newblock {\em Preprint.}


\bibitem{KRY}
S.~Kudla, M.~Rapoport, and T.~Yang.
\newblock {\em Modular forms and special cycles on {S}himura curves}, 
  vol.~161 of {\em Annals of Mathematics Studies}.
\newblock Princeton University Press, Princeton, NJ, 2006.

\bibitem{kai-wen}
K.-W. Lan.
\newblock {\em Arithmetic compactifications of {PEL}-type Shimura varieties},  vol.~36 of
{\em London Mathematical Society Monographs,} Princeton University Press, 2013

\bibitem{kai-wen_2}
K.-W. Lan.
\newblock Comparison between analytic and algebraic constructions of toroidal
  compactifications of {PEL}-type {S}himura varieties.
\newblock {\em J.~Reine Angew.~Math.~}664, 163--228, 2012.

\bibitem{laumon}
G.~Laumon and L.~Moret-Bailly.
\newblock {\em Champs alg\'ebriques}, volume~39 of {\em Ergebnisse der
  Mathematik und ihrer Grenzgebiete. 3. Folge. A Series of Modern Surveys in
  Mathematics [Results in Mathematics and Related Areas. 3rd Series. A Series
  of Modern Surveys in Mathematics]}.
\newblock Springer-Verlag, Berlin, 2000.

\bibitem{messing72}
W.~Messing.
\newblock {\em The crystals associated to {B}arsotti-{T}ate groups: with
  applications to abelian schemes}.
\newblock Springer-Verlag, Berlin, 1972.
\newblock Lecture Notes in Mathematics, Vol. 264.

\bibitem{miller09}
A.~Miller.
\newblock Compactifications of universal abelian threefolds with {CM}.
\newblock {\em Geom. Dedicata}, 143:155--179, 2009.

\bibitem{mumford77}
D.~Mumford.
\newblock Hirzebruch's proportionality theorem in the noncompact case.
\newblock {\em Invent. Math.}, 42:239--272, 1977.

\bibitem{pappas00}
G.~Pappas.
\newblock On the arithmetic moduli schemes of {PEL} {S}himura varieties.
\newblock {\em J. Algebraic Geom.}, 9(3):577--605, 2000.

\bibitem{roberts72}
J.~Roberts.
\newblock Chow's moving lemma.
\newblock In {\em Algebraic geometry, {O}slo 1970 ({P}roc. {F}ifth {N}ordic
  {S}ummer {S}chool in {M}ath.)}, pages 89--96. Wolters-Noordhoff, Groningen,
  1972.
\newblock Appendix 2 to: ``Motives'' ({\it Algebraic geometry, Oslo 1970}
  (Proc. Fifth Nordic Summer School in Math.), pp. 53--82, Wolters-Noordhoff,
  Groningen, 1972) by Steven L. Kleiman.

\bibitem{serre00}
J.-P. Serre.
\newblock {\em Local algebra}.
\newblock Springer Monographs in Mathematics. Springer-Verlag, Berlin, 2000.
\newblock Translated from the French by CheeWhye Chin and revised by the
  author.

\bibitem{soule92}
C.~Soul{\'e}.
\newblock {\em Lectures on {A}rakelov geometry}, volume~33 of {\em Cambridge
  Studies in Advanced Mathematics}.
\newblock Cambridge University Press, Cambridge, 1992.
\newblock With the collaboration of D. Abramovich, J.-F.\ Burnol and J. Kramer.

\bibitem{stroh}
B.~Stroh.
\newblock Compactification de vari\'et\'es de {S}iegel aux places de mauvaise
  r\'eduction.
\newblock {\em Bull. Soc. Math. France}, 138(2):259--315, 2010.

\bibitem{vistoli}
A.~Vistoli.
\newblock Intersection theory on algebraic stacks and on their moduli spaces.
\newblock {\em Invent. Math.}, 97(3):613--670, 1989.

\bibitem{Yang05}
T.~Yang.
\newblock C{M} number fields and modular forms.
\newblock {\em Pure Appl. Math. Q.}, 1(2, part 1):305--340, 2005.


\end{thebibliography}

\end{document}